\RequirePackage{rotating} 

\documentclass[trsc,nonblindrev]{informs3} 

\OneAndAHalfSpacedXI 

\usepackage[justification=centering]{caption}
\usepackage{soul}
	\usepackage{ulem}
	\usepackage{amsmath}
	\usepackage{comment}
	\usepackage{graphicx}
	\usepackage{enumerate}
	\usepackage{natbib} 
	\PassOptionsToPackage{hyphens}{url} 
	\usepackage{hyperref}
	\usepackage{caption}
\usepackage{booktabs}
\usepackage{multirow}
\usepackage[table,xcdraw]{xcolor}
\usepackage{graphics}
\usepackage{graphicx}
\usepackage{amsfonts}
\usepackage[Algorithm]{algorithm}
\usepackage[noend]{algpseudocode}
\usepackage{float}
\usepackage{listings}
\usepackage{booktabs}
\usepackage{multirow}
\usepackage{subcaption}
\usepackage[table,xcdraw]{xcolor}
\usepackage{amssymb}
\usepackage{pifont}
\usepackage{tabularx}
\usepackage{rotating}
\usepackage[title]{appendix}

\usepackage{natbib}
 \bibpunct[, ]{(}{)}{,}{a}{}{,}%
 \def\bibfont{\small}%
\TheoremsNumberedThrough     
\ECRepeatTheorems
\EquationsNumberedThrough    

    \newcommand{\mathleft}{\@fleqntrue\@mathmargin0pt}
\newcommand{\IL}[1]{\textcolor{black}{#1}}

\newcommand{\cK}{\mathcal{K}}
\renewcommand{\qed}{\hfill\blacksquare}
\newcommand{\xmark}{\ding{55}}%

\newtheorem{prop}{Proposition}
\newtheorem{theo}{Theorem}

	\newcommand{\IfThen}[2]{
  \State \algorithmicif\ #1\ \algorithmicthen\ #2
  }
	\algrenewcommand\algorithmicensure{\textbf{Output:}}
\renewcommand{\emph}[1]{\textit{#1}}
\begin{document}


\RUNAUTHOR{Li, Archetti and Ljubi\'c}
\RUNTITLE{RL Approaches for the OP with Stochastic and Dynamic Release Dates}
\TITLE{Reinforcement Learning Approaches for the 
			Orienteering Problem with Stochastic and Dynamic Release Dates}
\ARTICLEAUTHORS{%
\AUTHOR{Yuanyuan Li, 
Claudia Archetti,  
Ivana Ljubi\'c\footnote{Corresponding author}} 
\AFF{
	IDS Department,		ESSEC Business School, Cergy-Pontoise, France, \\
\EMAIL{yuanyuan.li@essec.edu}
\EMAIL{archetti.archetti@essec.edu}
\EMAIL{ivana.ljubic@essec.edu}}
} 

\ABSTRACT{%

In this paper, we study a sequential decision-making problem faced by e-commerce carriers related to when to send out a vehicle from the central depot to serve customer requests, and in which order to provide the service, under the assumption that the time at which parcels arrive at the depot is stochastic and dynamic. 
The objective is to maximize the expected number of parcels that can be delivered during service hours. 
We propose two reinforcement learning (RL) approaches for solving this problem. 
These approaches rely on a look-ahead strategy in which future release dates are sampled in a Monte-Carlo fashion and a batch approach is used to approximate future routes.  Both RL approaches are based on value function approximation - one combines it with a consensus function (VFA-CF) and the other one with a two-stage stochastic integer linear programming model (VFA-2S). VFA-CF  and VFA-2S do not need extensive training as they are based on very few hyper-parameters and make good use of integer linear programming (ILP) and branch-and-cut-based exact methods to improve the quality of decisions. We also establish sufficient conditions for partial characterization of optimal policy and integrate them into  VFA-CF/VFA-2S. In an empirical study, we conduct a competitive analysis using upper bounds with perfect information. We also show that VFA-CF and VFA-2S greatly outperform alternative approaches that: 1)  do not rely on future information, or 2) are based on point estimation of future information, or 3) employ heuristics rather than exact methods, or 4) use exact evaluations of future rewards.  
}%


\HISTORY{}

\maketitle

%

\section{Introduction}
\label{sec:introduction}

E-commerce markets are booming at remarkable rates. Due to an unprecedented series of lockdowns, billions of people stayed at home to prevent the spread of COVID-19. It is reported that the e-commerce revenue saw a 10\% increase in Europe in 2020 due to the pandemic (\cite{Statista2021}).
At the same time, the expectations of customers in terms of service quality are also increasing. When questioned about the features they would like to obtain from the delivery of online purchases, almost half (48\%) of shoppers mentioned speed, 42\% reduction of the cost of delivery, and 41.6\% more reliable information about delivery time (\cite{StatistaShopping}). 
By providing evidence of the clash between SF Express and Alibaba, \cite{cui2020value} show that delivering   expensive or less-discounted products in a reliable and timely manner brings better customer experience and thus higher profits.
When it comes to the specific challenges in improving customers' satisfaction with delivery, one of the crucial features is related to 
the last-mile delivery leg (\cite{archetti2021recent}).  An important feature of last-mile distribution is that its operations start as soon as parcels are delivered to the final logistic center, typically a city distribution center (CDC). Given the short delivery times requested by the customers, delivery operations typically need to start before all parcels expected to arrive during the day are available at CDC. This raises a question: should the dispatcher wait for more or all parcels to be delivered at CDC, or should he/she start the delivery as soon as there is any available parcel and vehicle?

In this paper, we focus on the sequential decision-making problem related to when to deliver parcels and which parcels to deliver under the assumption that the time at which parcels become available at CDC (called \textit{release dates}) is stochastic and dynamic. Release dates are the moments when parcels become available
for delivery. On the other hand, the release time is like the beginning of a time window, i.e., the earliest feasible time at which a customer can be visited or a cargo can be picked up. In problems with release time, a vehicle can start the route without waiting for all parcels to be ready. We introduce a new problem called the \emph{Orienteering Problem with Stochastic and Dynamic Release Dates} (DOP-rd).
The problem finds applications in same-day delivery (SDD) systems (\cite{li2023emerging}). As stated in \cite{stroh2022tactical}, SDD services face tight delivery deadlines and relatively low order volumes, thus time instead of vehicle capacity tends to be the limiting resource. We assume that a single and uncapacitated vehicle is serving customer requests with no time window so that the vehicle should finish its service before the deadline $T_E$, which might correspond to the duration of the driver's working shift.  
The release dates are considered to be uncertain and their distributions are dynamically updated during the sequential decision-making process.  The objective is to maximize the expected number of requests served within $T_E$. This corresponds to maximizing the expected number of customers who get their packages within the desired deadline (end of the day), thus in turn, this implies the maximization of customer satisfaction \citep{joerss2016customer,voccia2019same}. 


This paper presents the following main contributions to the literature:  
\begin{enumerate}
    \item We introduce and study DOP-rd. The problem shares similarities (with respect to the input setting) with the one studied in \cite{archetti2020dynamic}. The main differences are in the objective function and the presence of a deadline. While \cite{archetti2020dynamic} aim to minimize the expected traveling and waiting time while ensuring that all parcels are delivered, the DOP-rd seeks to maximize the expected number of parcels delivered within a deadline $T_E$. The shift in objectives and constraints brings fundamental changes to the problem setting, necessitating unique solution approaches.  
    \item We model the problem as a Markov Decision Process (MDP) and propose a batch approach to approximate future routes (actions), based on Monte-Carlo simulation for scenario generation and continuous approximation from \cite{daganzo1984length} for estimating the duration of routes. To the best of our knowledge, the method we propose to approximate future routes has never been applied before to routing problems. At each decision epoch, we learn the value function based on the updated information, which is embedded both in scenario generation and in the approximation approach. Thus it leads to a reinforcement learning approach. 
    
    \item 
    We propose two novel hybrid approaches to derive a decision-making policy 
    in the (approximated) decision space. Both belong to the class of value function approximation (VFA) methods and are hybridized with a one-step look-ahead policy. The first is combined with a consensus function (VFA-CF) and the second with a two-stage stochastic ILP model (VFA-2S).  
The former approach relies on exact solutions of a series of deterministic integer linear programs (ILPs); the latter one uses a two-stage stochastic ILP model instead.  These ILPs are solved using state-of-the-art branch-and-cut techniques. The obtained solutions are then used to decide when to leave the depot and which customers to serve while maximizing the expected number of parcels delivered. In addition, we propose a partial characterization of optimal policy, the conditions of which are checked before applying VFA-CF/VFA-2S. 
\item 
Most of the existing literature on stochastic and dynamic vehicle routing problems dealing with MDP relies on heuristic policies (see \cite{ulmer2020modeling}). In this work, we approximate future routes in each decision epoch and employ exact methods to determine the best action associated with the current state. 
Thanks to this advantage of embedding exact methods 
inside a reinforcement learning framework, 
VFA-CF/VFA-2S can be run online without sophisticated learning procedures that require extensive training, such as those based on neural networks. Specifically, learning the value function requires only minimal tuning of a few hyper-parameters.
    \item 
    In an extensive computational study, we show VFA-CF/VFA-2S excel other benchmark approaches thanks to these following distinct features: (1) embedding of the future information, (2) using sampling rather than point estimations, (3) using exact methods rather than simple heuristics, and (4) using approximations of future routes. Specifically, concerning point (2), we compare the VFA approaches with a fully deterministic approach, which builds exact solutions based on point estimations of future release dates. \\
    We additionally provide results of a competitive analysis, in which the upper bounds are calculated by solving to optimality the problem under the assumption of having perfect information.
    Besides, we conduct experiments with two vehicles to test the potential and adaptability of proposed VFA approaches when extended to the multi-vehicle case. 

\end{enumerate}

The paper is organized as follows. Section \ref{sec:literature} presents a literature review by focusing on the differences with the current work. Section \ref{sec:problem_formu} contains  the problem description and formulation as a  Markov Decision Process. Section \ref{sec:ub_pc} presents upper bounds to solution values and a partial characterization of the optimal policy. The solution approaches are introduced in Section \ref{sec:sol_approach}, while  Section \ref{sec:computational_result} is devoted to computational experiments and results. Finally, we conclude in Section \ref{sec:conclusion}. 

\section{Literature Review}
\label{sec:literature}

We first focus on the literature on routing problems with stochastic and dynamic features and orienteering problems in particular, and then on the common approaches of reinforcement learning used for these routing problems. Finally, we move to problems with release dates and highlight differences in problem settings and solution approaches compared to the current work.

The literature on stochastic and dynamic routing problems is wide. We focus on pioneering papers and surveys. \cite{psaraftis1988dynamic} introduce the traveling salesman problem with the dynamic customer demands.  \cite{bertsimas1991stochastic} generalize the former results and propose a  mathematical model for the stochastic and dynamic vehicle routing problems (VRPs). \cite{bent2004scenario} propose using the multi-scenario sampling in combination with a consensus function for the dynamic vehicle routing problem with time windows. Other interesting extensions  of stochastic or dynamic settings appear in \cite{savelsbergh1998drive},  \cite{secomandi2001rollout}, \cite{secomandi2009reoptimization}, \cite{goodson2013rollout}, \cite{subramanyam2021robust}. 
As for surveys on stochastic and dynamic VRPs, we refer to \cite{ritzinger2016survey}, \cite{oyola2018stochastic}, \cite{soeffker2021stochastic}. 

The problem we study belongs to the class of orienteering problems. An Orienteering Problem (OP) involves a set of nodes, each with a corresponding score. The goal is to determine a path (or a tour) that visits a subset of these nodes, maximizing the total accumulated score subject to a limited time budget. For the details on the variants and applications, we refer readers to surveys by \cite{souffriau2011orienteering}, \cite{archetti2014chapter}, and \cite{gunawan2016orienteering}. Specifically, \cite{souffriau2011orienteering} present extensions, variants, solution approaches, and applications. \cite{archetti2014chapter} present a survey of the class of vehicle routing problems with profits, where the orienteering problems are considered as the basic problems. \cite{gunawan2016orienteering} extend the summaries of \cite{souffriau2011orienteering} and \cite{archetti2014chapter} by including more recent papers and new variants. \cite{gunawan2016orienteering} also study OP with stochastic aspects such as uncertainties in the collected scores, travel and service times, waiting time, etc. In addition, \cite{zhang2018dynamic} present a variant of OP in which the traveler may experience stochastic waiting time at locations. The objective is to maximize the expected rewards collected by determining the visiting sequence while respecting each customer's time window. \cite{song2020building} address a problem involving regular service for subscription customers. The goal is to minimize the expected costs minus the revenue from serving on-demand customers. The problem is modeled as a two-stage stochastic decision problem, where the first stage assigns drivers to customers and the second stage is a team-orienteering problem with time windows. Another paper by \cite{angelelli2021dynamic} studies a dynamic and stochastic variant of OP spanning two continuous days. On the first day, the company accepts or rejects real-time requests that arrive randomly. On the second day, the accepted requests require a vehicle to visit customers for goods pickup.  The goal is to maximize the expected profit, defined as the difference between the expected total prize and travel cost.


\subsection{Reinforcement Learning for Stochastic and Dynamic Routing Problems}
Reinforcement learning (RL) has emerged as a successful approach for solving various decision-making problems (\cite{silver2016mastering}). It is based on the fundamental concept that an agent interacts with an environment and learns what to do based on the feedback it receives in the form of rewards or penalties. The detailed techniques and algorithms used in RL have been extensively discussed in books, with notable references including \cite{sutton2018reinforcement} who describe the topic from the point of view of computer science, \cite{bertsekas2019reinforcement} who adopts the language of control theory, and \cite{Powell2022reinforcement} who names it approximate dynamic programming (ADP) using the language of operations research. The application of RL in solving stochastic and dynamic routing problems has gained significant attention. Next, we present the related research works.

Q-learning is an algorithm that learns the value of taking an action in a state.  It is one of the most popularly used RL algorithms and has been applied to address various routing problems. Here are some examples. \cite{mao2018reinforcement} study the adaptive routing problem in stochastic and time-dependent traffic networks. The goal is to find a routing strategy to minimize the expected total travel time. To solve the problem, they introduce two variants of Q-learning approaches. The first variant is an online Q-learning method that requires a discrete state space. The second is an offline fitted Q iteration technique with a continuous state space that approximates the Q function to overcome the curse of dimensionality. \cite{basso2022dynamic} consider a VRP with a single electric commercial vehicle for reliable charge planning. The problem considers the uncertainties in energy consumption and customer requests, aiming to minimize the expected energy consumption. 
To address the problem, the authors propose a combination of Q-learning and VFA techniques. Specifically, they employ Q-learning with expected energy and battery depletion and a VFA based on a reduced state representation. 
\cite{revadekar2020qoral} study a pickup and delivery problem in the context of online medicine delivery. In the problem, customers order medicines through an online system, which finds the nearest pharmacy and assigns a delivery person to serve the delivery request. The goal is to minimize the delivery time. To tackle the problem, the authors propose a double Q-learning, which uses a double estimator to update the Q tables storing distance and time reward matrices separately. 
Similarly, \cite{bozanta2022courier} study a pickup and delivery problem but in a meal delivery service. The problem involves picking up meals from customer-requested restaurants and delivering them to customers located on an $m\times m$ grid, with a limited number of available couriers and a restricted time frame. In this case, the objective is to maximize the revenue derived from the requests served. The authors propose three different Q-learning approaches. 
\cite{chen2022same} work on addressing the unfairness issue in SDD service availability. They model the problem with a weighted sum of two objectives: maximizing the overall service level (utility) and maximizing the minimum service rate across all regions (fairness). To address this problem, the authors first use a routing heuristic to reduce the action space and then propose a deep Q-learning approach. \cite{chen2022deep} also utilize a deep Q-learning approach but in the context of a decision problem involving the assignment of new customers to either drones or vehicles or the rejection of service.
For more references on the topic, we refer readers to the survey by \cite{hildebrandt2022opportunities}.

Contrary to the methods described above, we do not rely on classical $Q$-learning mechanisms, and employ exact methods combined with a one-step look-ahead and approximation of future routes instead. To the best of our knowledge, closest to our work in the routing literature is the study by \cite{anuar2021multi} where the authors study a multi-depot dynamic VRP with stochastic road capacity caused by earthquake damages. In their problem setting, the routes computed in advance may become infeasible during real-time execution due to unforeseen disasters such as earthquakes. The authors propose a solution approach using two ILP models, one for replenishment and the other for delivering supplies. These two models are used interchangeably based on the vehicle's capacity. 
Like the current work, \cite{anuar2021multi} use exact methods inside the RL/ADP framework. 
However, the two works make different contributions, which are summarized in Online Appendix B.

\subsection{Routing Problems with Release Dates}
\label{lite:RD}
The previous studies mainly focus on demands and travel time uncertainties. The works on other stochastic and dynamic aspects, such as release dates, appeared recently.
The term ``routing problem with release date" was first introduced in \cite{cattaruzza2016multi}, where the authors define it as a problem where each parcel is associated with a release date indicating the time at which the parcel is available at depot. Release dates are supposed to be known beforehand. The problem is a multi-route vehicle routing problem with release dates, and the authors solve it through a population-based algorithm.
In  \cite{archetti2015complexity}, the Traveling Salesman Problem with release dates (TSP-rd) is considered, i.e., the problem with a single uncapacitated vehicle. The authors analyze the complexity of the problem in two variants - minimizing the total distance traveled with a constraint on the maximum deadline and minimizing the total time needed to complete the distribution without a deadline. Later, \cite{reyes2018complexity} extend the results found in \cite{archetti2015complexity} by considering service guarantees on customer-specific delivery deadlines. \cite{shelbourne2017vehicle} consider VRP with release and due dates where a due date indicates  the latest time the order should be delivered to the customer. The objective function is defined as a convex combination of the operational costs and customer service level. The authors propose a path-relinking algorithm to address the problem. More recently, \cite{archetti2018iterated} propose a mathematical programming formulation and present a heuristic approach based on an iterated local search for solving TSP-rd. 

All contributions mentioned above study routing problems with deterministic release dates. Recently, there have been several studies related to applications in SDD services, where release dates are stochastic. Among them,  \cite{voccia2019same} study a multi-vehicle dynamic pickup and delivery problem with time constraints where requests arrive dynamically from online customers. Each request is associated with a deadline or a time window. The objective is to maximize the expected number of requests that can be delivered on time. The authors first model the problem as an MDP, then, based on sample-scenario planning, they propose a heuristic solution approach. Rather than defining delivery windows at the receiver side, \cite{van2019delivery} consider dispatch windows at the consolidation center. 
As in \cite{voccia2019same}, an MDP is defined. To overcome ``the three curses of dimensionality'',  the authors propose an ADP algorithm using a VFA. In contrast, \cite{archetti2020dynamic} study a single-vehicle problem with no deadline and where the objective is to deliver all parcels as fast as possible. They propose a reoptimization approach that heuristically reduces the number of decision epochs. Furthermore, two models are defined to select an action at each reoptimization epoch. While both models proposed in \cite{archetti2020dynamic} aim to minimize the total completion time, the main difference between the two is that the first model is a stochastic mixed-integer program model that considers the full stochastic information on the release dates, while the second one is a deterministic model that uses a point estimation (PE) for the stochastic release dates. For a detailed survey on routing problems in SDD, see \cite{li2023emerging}.

Based on \cite{archetti2020dynamic}, our paper studies the DOP-rd with a single, uncapacitated vehicle serving customer requests with no time window. Unlike \cite{archetti2020dynamic}, which does not consider drivers' working hours or strict deadlines, we aim at maximizing the expected number of parcels delivered within a deadline $T_E$. Recent studies highlight the practical importance of focusing on the maximum number of delivered parcels (\cite{voccia2019same} and \cite{behrend2018integration}, \cite{ulmer2018same}), as vehicles often need to return to the depot during working hours, and some parcels may be deferred for distribution the next day or by a third-party carrier. In \cite{archetti2020dynamic}  every solution in which all parcels are served is feasible. In their setting, the vehicle can wait until the last parcel arrives at the depot and subsequently serve all the parcels in a single route. On the contrary, in our model, we address strict deadline constraints, ensuring timely parcel delivery. This distinguishes our approach from \cite{archetti2020dynamic}. 

As for solution approaches, the decision-making policy we consider adds to existing reoptimization frameworks behind the VFA approaches in routing literature. Indeed, instead of proposing a heuristic approach for solving problems in each scenario as in \cite{voccia2019same} or at each reoptimization epoch as in \cite{archetti2020dynamic}, we solve ILPs with the approximation of future routes to balance the computational efficiency and policy performance. Our extensive numerical study shows the value of our proposed reoptimization scheme in routing applications. The ones proposed in the current work are innovative not only in the way future routes are approximated but also in how policies are derived through VFA-2S and VFA-CF and in the way the ILP models are constructed.

Another contribution, which is very close to the current study, is given by \cite{van2019delivery}. Again, the problem is formulated as an MDP. However, decision epochs are predefined time instants separated by equidistant intervals, as in \cite{klapp2018one}.
The objective function in \cite{van2019delivery} is to minimize the expected total dispatching costs. 
The paper focuses on dispatching decisions, so the routing decisions are not considered in the optimization problem, and routing costs are estimated by Daganzo's formula. In our work, decision epochs are set dynamically according to the release dates. Also, differently from \cite{van2019delivery},  we consider a single vehicle. Finally, we use Daganzo's formula with a different scope, i.e., for inferring how many requests can be fulfilled in future routes at each decision epoch. Instead, we get the exact routing sequence for the current epoch. 

Finally, another related work is given by \cite{schrotenboer2021fighting}, where the authors study how to balance the consolidation potential and delivery urgency of orders. Contrary to our setting, it is a pickup and delivery problem with a fleet of vehicles and customer deadlines, aiming to maximize the expected customer satisfaction. Their approach is based on a cost function approximation.



We summarize the major characteristics of the contributions on routing problems with release dates in Table 3 in Online Appendix B. 

\section{Problem Description}
\label{sec:problem_formu}

The DOP-rd is defined as follows. A company distributes parcels to its customers from a CDC. At CDC the company may hold some level of inventory, so some parcels may be available right away, but most of them are delivered to the CDC by suppliers throughout the day. Hence, the distribution of the latter parcels can start only after their arrival at CDC. The moment at which parcels are delivered to the CDC is unknown. In case suppliers' vehicles are equipped with Global Positioning System (GPS), the information about their arrivals to the CDC (also called depot in the following) is constantly updated. In SDD, customer locations can be partitioned into districts with dedicated vehicles, where each district can be served by a vehicle multiple times during working hours (\cite{banerjee2022fleet}, \cite{stroh2022tactical}). We consider a single uncapacitated vehicle to perform the service within a district. The goal is to serve as many parcels as possible within a given deadline. 

More formally, let  $G = (V, A)$ be a complete graph. The set of vertices $V$ includes vertex $0$, representing the depot, and the set $N$ of customers. Each customer is associated with an arrival time of its parcel to the depot, called release date, and a delivery location. Packages arrive throughout the day until a predefined deadline $T_E$, corresponding to when the vehicle has to be back at the depot and finish the service. Release dates are stochastic, and their distributions are dynamically updated throughout the day. Specifically, the stochastic aspect simulates the situation where suppliers' vehicles carrying the parcels to the CDC might encounter unpredicted events. The dynamic aspect is related to the fact that the information about the arrival of the suppliers' vehicles to the CDC evolves over time, as in the case of GPS-equipped vehicles. Each arc $(i, j) \in A$ is associated with a traveling time $d_{ij}> 0$, and we assume that the triangle inequality holds, i.e., $d_{ij}+d_{ji'} \ge d_{ii'}$, $i,j,i' \in V$. 
A single uncapacitated vehicle is available to perform the deliveries. We define a route as a tour starting and ending at the depot and visiting customers in between. The set of routes is denoted as $\mathcal{K}=\{0, ..., |\mathcal{K}|-1\}$, where $\mathcal{K}$ contains the first route leaving for delivery (also called route 0), followed by the set of future routes denoted by $K$. 
The goal is to maximize the expected number of requests served within $T_E$. 

In this paper, we present the case in which the orders are placed, but the arrival time of the parcels ordered at the depot (CDC) is unknown. However, the solution approach we propose is also valid for the case where goods are at the depot but the request time is unknown, as studied in \cite{voccia2019same}. In both cases, the customer locations are assumed to be known. We note that the problem studied in \cite{voccia2019same} represents a particular case of our problem as it does not consider dynamic release dates.

We now formalize the problem as a Markov Decision Process (MDP). Note that the detailed notation used is summarized in Table 1 in Online Appendix A.

\subsection{The Problem as a Markov Decision Process} 

At each decision epoch $e$, we distinguish between customers that have already been served (i.e., their parcels have been delivered), $N_e^{served}$, and those that are still unserved, $N_e^{unserved}$, where $N_e^{served} \bigcup N_e^{unserved} = N$.
There are two categories of unserved customers based on the information available for the release date at decision epoch $e$. If the parcel has arrived at the depot, the release date is known, and the customer can be served, thus we call this set of customers $N_e^{known}$. On the contrary, $N_e^{unknown}$ denotes the set of customers whose parcels are still to be delivered to the depot. In turn, there are two subsets of $N_e^{unknown}$:  customers whose information about the distribution of the release date is not updated until the parcel arrives at the depot are denoted as $N_e^{static}$; customers whose information about the distribution of the release date is updated, i.e., a new estimation of release date distribution is given in real-time, thanks to GPS equipped vehicles, are denoted as $N_e^{dynamic}$. Note that $N_e^{static} \bigcup N_e^{dynamic} = N_e^{unknown}$. 
The parcels are delivered individually or in batches from suppliers to the depot.
Let $\widetilde{r}^e_i$ be the random variable associated with the release date of customer $i \in N_e^{unserved}$ at decision epoch $e$. This variable is defined as follows:
\begin{itemize}
    \item for $i \in N_e^{known}$, $\widetilde{r}^e_i = r_i$, where $r_i$ is the actual arrival time of the parcel of customer $i$.
        \item for $i \in N_e^{static}$, we have $\widetilde{r}^e_i = \widetilde{r}^0_i$, which means that the distribution of the random variable associated with the release date is estimated at time 0 and never updated till the parcel arrives at the depot. In practice, the mean and variance of the distribution can be empirically obtained from historical data.
        \item for $i \in N_e^{dynamic}$, $\widetilde{r}^e_i$ is a random variable representing release date at decision epoch $e$. Its distribution is dynamically updated throughout the day.
\end{itemize}

The distinction between $ N_e^{dynamic}$ and $N_e^{static}$ does not relate to policy design (which is solely based on $N_e^{unknown}$) but rather pertains to the utilization of instances generated for numerical tests (see Online Appendix F).

At each decision epoch $e$, the company has to decide whether to wait for more parcels to arrive or to dispatch the vehicle delivering a subset of the parcels available at the depot.  No preemptive return is allowed, i.e., the vehicle comes back to the depot only once all parcels on board are delivered to the corresponding customers. \cite{ulmer2019preemptive} investigate the benefit of allowing preemptive depot returns to integrate dynamic requests into delivery routes. The results showed that benefits are marginal.

The MDP components are the following:
\begin{itemize}
    \item  \textit{Decision epochs:} 
    A decision epoch is denoted as $e \in \{0, ..., E\}$, and the time associated with it is $t_e$. Note that $T_E$ represents the deadline, i.e., when the vehicle has to be back to the depot, and no further deliveries are allowed. The first decision epoch $0$ corresponds to the beginning of the delivery operations. At each decision epoch, an action is made based on the available deterministic and updated stochastic information. A decision epoch corresponds to either the arrival of a package at the depot or to the time when the vehicle is back to the depot at the end of a route.  
    To avoid waiting for too long at the depot without having any new arrival of packages, additional decision epochs are defined at equally distant time intervals $\phi$, when the action is evaluated according to the updated stochastic information. 
    \item  \textit{States:} 
    The state $S_e = (t_e, N^{known}_e, \{\widetilde{r}^e_i\}_{i\in N_e^{unknown}})$ includes all information for evaluating an action at decision epoch $e$, i.e.,  the time of the decision epoch $t_e$,  the set $N^{known}_e$ of customers with known release dates (whose parcels have arrived), 
    and the stochastic information (i.e., distributions) related to random variables associated with future release dates $\{\widetilde{r}^e_i\}_{i\in N_e^{unknown}}$. 
    At decision epoch 0, no customer has been served. 
    \item  \textit{Actions:} 
    At each decision epoch $e$, the company has to either dispatch a subset of accumulated orders or wait at the depot for new parcels to arrive. The decision involves action $\mathcal{X}_e$, to which we also refer as route 0 for decision epoch $e$. When the vehicle is not at the depot or if the vehicle is at the depot, but there is no parcel available, the decision is to wait, i.e., the route 0 is empty. 
    Similarly to what has been done in \cite{archetti2020dynamic},  if dispatching is decided, an action consists of a new route serving a subset of the unserved customers. However, while in \cite{archetti2020dynamic} customers served in the new route could belong to either $N^{known}_e$ or $N^{unknown}_e$, we instead reduce the set of customers to be potentially included in the new route to $N^{known}_e$. This implies that the new route, in case at least one customer is included, will leave immediately from the depot, contrary to \cite{archetti2020dynamic} where the vehicle might have to wait at the depot in case the route includes at least one customer in $N^{unknown}_e$.
    Formally, the action space at decision epoch $e$, denoted by $X(S_e)$, includes all possible subsets of $N^{known}_e$, each one associated with the TSP route serving all parcels in the subset.
    The set of customer locations visited in the route associated with the action $\mathcal{X}_e$ is denoted as $l(\mathcal{X}_e)$, which is empty if the route is empty, i.e., the vehicle is not dispatched and the action is waiting at the depot. The route 0 associated with the action in decision epoch $e$ consists of a set of consecutive arcs traversed by the vehicle (this set could also be empty). The traversal of an arc is represented by the binary variable denoted as $x^e_{ij}, i, j\in N_e^{known}\cup \{0\}$. Thus $\mathcal{X}_e$ can be expressed in terms of these decision variables as $\mathcal{X}_e =(x^e_{ij})$.
    
    
    \item  \textit{Transitions:} 
    After action $\mathcal{X}_e$, a transition is made from state $S_e$ to a new state $S_{e+1}$. The state $S_{e+1}$ is described as $S_{e+1} = (t_{e+1}, N^{known}_{e+1}, \{\widetilde{r}^{e+1}_i\}_{i\in N_{e+1}^{unknown}})$. 
    We have $t_{e+1} = t_e + t_{route}(\mathcal{X}_e)$ if route is not empty,  $t_{e+1}$ = min$(t_p, t_e+\phi)$ otherwise, 
    where $t_{route}(\mathcal{X}_e)$ is the time associated with the route corresponding to action $\mathcal{X}_e$, $t_p$ is the earliest time when a new parcel arrives while the vehicle is at the depot, 
    and $\phi$ is the maximum waiting time between two decision epochs as defined earlier. Let $N^{new}_{e+1}$ be customers whose parcels arrive at the depot between decision epochs $e$ and $e+1$. 
    Correspondingly, the set of customers is updated as $
        N_{e+1}^{unknown} := N_{e}^{unknown}\backslash N^{new}_{e+1} \ \text{and}\ 
    N^{known}_{e+1} :=N^{known}_e\symbol{92}l(\mathcal{X}_e)\cup N^{new}_{e+1}.$
    \item  \textit{Objective:} 
    The objective is to find a policy maximizing the expected number of requests served within $T_E$. Each action $\mathcal{X}_e$ chosen when the system is in state $S_e$, is guided by a decision rule and brings an \textit{immediate reward}, expressing the number of requests served by the route dispatched at time $t_e$, and a future reward. A sequence of decision rules is defined as a \textit{policy}, and the optimal policy is the one that maximizes the expected number of parcels delivered till epoch $E$.
    
        In decision epoch $e$, the goal is to maximize the expected number of requests served within the time interval $[t_e,T_E]$. We set the final value function $V_{E}(S_E)=0$, and for $e=0, ..., E-1$, we define the value function as the following recursive formula:
    \begin{equation}\label{eq:value_func_obj}
    V_e(S_e) = \max_{\mathcal{X}_e\in X(S_e)}\{C(S_e, \mathcal{X}_e) + \gamma\mathbb{E}[V_{e+1}(S_{e+1}\mid S_e, \mathcal{X}_e)]\}.
    \end{equation}

 Here, $C(S_e, \mathcal{X}_e)$ denotes the immediate reward, and $\gamma\mathbb{E}[V_{e+1}(S_{e+1}\mid S_e, \mathcal{X}_e)]$ is the expected value  associated with decision epoch $e+1$, while  $\gamma$ is a discount factor. In situations characterized by high information uncertainty and an inability to guarantee the precision of future reward estimations, a discount factor can be applied to discount the effect of
making mistakes in the future (\cite{burk2023environmental}). Considering the nature of the SDD problem, where time constraints are so tight, the more distant the future is, the less certain and accurate our solutions tend to be. Therefore, including a discount factor, as an algorithmic parameter, allows us to explicitly adapt to and manage the diminishing accuracy over time, providing better flexibility and performance of the algorithm in handling such a dynamic problem. 

The system suffers from the curse of dimensionality in the number of states and actions, which grows exponentially with the number of requests. In the recent literature,  reinforcement learning is frequently used to deal with this issue (see, e.g., \cite{powell2009you} and \cite{van2019delivery}). In this work, we propose two reinforcement learning approaches for solving the DOP-rd, the first based on a value function approximation with a consensus function and the second on a value function approximation with a two-stage stochastic ILP model. Both methods are hybridized with a one-step look-ahead strategy, in which future release dates are sampled in a Monte-Carlo fashion.  

The two approaches, which make use of branch-and-cut-based exact methods to improve the quality of decisions, are described in detail in Section \ref{sec:sol_approach}.

\end{itemize}

\section{Upper Bounds and Partial Characterization of Optimal Policy}\label{sec:ub_pc}

In the following, we discuss valid upper bounds useful for competitive analysis. We then provide sufficient conditions under which it is optimal for the vehicle to leave the depot immediately in order to serve a subset of available parcels.

\subsection{Upper Bounds}

A first trivial upper bound for DOP-rd can be calculated by solving an OP, assuming that all unserved parcels are available at the depot (see Online Appendix C for further details). In the following, we instead describe a stronger upper bound obtained a-posteriori when all information is known.

\paragraph{Upper Bound with Perfect Information:} 
A tighter upper bound for DOP-rd can be computed  
in a \emph{wait-and-see} fashion, using the perfect information with respect to realized release dates. In this case, the optimal solution is the solution of the OP with Release Dates (OP-rd), 
a deterministic counterpart of the problem studied in this article. Inspired by the model in \cite{archetti2018iterated} for TSP-rd and completion time minimization, we propose a formulation for OP-rd as follows.

Let us assume that all release dates, $r_i\ge 0$, $i \in N$ are known. Then, OP-rd aims to find the maximum number of parcels that can be served in the interval $[0,T_E]$. A solution consists of up to $|\cK|$ routes performed by the vehicle, so that the end time of the last route does not exceed $T_E$. The problem is NP-hard since the OP can be reduced in polynomial time to it.
We can model the problem as an ILP, making use of the following decision variables:
\begin{itemize}
    \item $s_i^k$: binary variable indicating whether parcel $i$ is served in route $k$;
    \item $\xi^k_{ij}$: binary variable equal to 1 in case arc $(i,j)$ is traversed in route $k$;
    \item $d_k$: continuous variable representing the starting time of route $k$.
\end{itemize}
The \IL{ILP formulation for the OP-rd} is then:
\begin{subequations}
	\label{eq:base}
\begin{equation}\label{eq:obj_ILP_base}
\hbox{max} \sum_{i \in N}\sum_{k \in \cK} s^k_i 
\end{equation}

\begin{align}
\label{eq:inout_ILP_base}
 \text{subject to: } \sum_{(i, j)\in A}\xi^k_{ij} &= \sum_{(j, i)\in  A}\xi^k_{ji}  = s_i^k && \quad i \in V, k \in \cK\\
\label{eq:subtour_ILP_base}
 \sum_{i, j\in S}\xi_{ij}^k & \leq \sum_{i \in S\setminus{\ell}}s_i^k 
&& \quad S \subseteq  N, \IL{|S|\geq 2,} \ell \in S, k \in \cK \\
\label{eq:singleVisit_ILP_base}
\sum_{k \in \cK}s_i^k & \leq 1 &&  \quad i \in N \\
\label{eq:releaseDate_ILP_base}
d_k & \ge r_i s_i^k &&  \quad i \in N, k \in \cK \\
\label{eq:consecutiveroutes_ILP_base}
d_{k+1} & = d_k + \sum_{(i, j)\in A}d_{ij} \xi^k_{ij}  && \quad k \in \cK \\
\label{eq:finalroute_ILP_base}
d_{|\cK|} & = T_E &&  \\
\IL{\xi^k_{ij}, s^k_{i}} & \in \{0,1\} && i,j \in V, k \in \cK
\end{align}

where \eqref{eq:inout_ILP_base} are degree constraints, \eqref{eq:subtour_ILP_base} are the generalized subtour elimination constraints (GSECs), and \eqref{eq:singleVisit_ILP_base} 
fix the number of visits to a customer to be at most one. Constraints \eqref{eq:releaseDate_ILP_base}--\eqref{eq:finalroute_ILP_base} determine the starting time of each route. Specifically, \eqref{eq:releaseDate_ILP_base} establish that a route cannot depart prior to the release date of each customer to be visited, \eqref{eq:consecutiveroutes_ILP_base} fix the starting time of a route as the ending time of the preceding route, and \eqref{eq:finalroute_ILP_base} set the ending time of the last route equal to $T_E$. Note that constraint \eqref{eq:finalroute_ILP_base}, as well as the equality in constraint \eqref{eq:consecutiveroutes_ILP_base}, is because the waiting time can be shifted to the beginning of the distribution without affecting the structure and the value of the solution, thus avoiding any waiting time between consecutive routes and at the end, as shown in \cite{archetti2018iterated}. This way, the starting time of route $d_{|\mathcal{K}|}$ equals the ending time of the last route performed, $|\mathcal{K}|-1$, which is, in turn, equal to the deadline $T_E$.
We also add constraints \eqref{subt0} and symmetry-breaking constraints 
 \eqref{symmetry_break} to strengthen the formulation:
\begin{align} 
\sum_{i \in N} s_i^k & \leq |N| s_0^k && \quad k \in \cK && \label{subt0}
&& \\
s_0^k & \leq  s_0^{k+1} && \quad k \in \cK, k\neq|\cK|-1. && \label{symmetry_break}
\end{align}
\end{subequations}

The optimal solution of \eqref{eq:base} provides an upper bound to the DOP-rd. If the optimal solution is not found within a given computation time limit, the upper bound at the end of the computation (rounded down to the nearest integer) provides a valid upper bound for the DOP-rd. The model is solved using a branch-and-cut algorithm, in which the GSECs are separated on the fly.  For more details on the separation of these constraints, see e.g., 
\cite{lucena2004strong}, \cite{Ljubic21}.

Note that Formulation \eqref{eq:base} needs an upper bound on the number of routes $|\cK|$ as input. A trivial upper bound is $|\cK|=|N|$, which is the one used in \cite{archetti2018iterated}. However, a tighter value can be obtained by solving the following ILP problem. Let us order customers $i \in N$ in increasing order according to the values of release dates (ties are broken arbitrarily). We define $\alpha_i$ as a binary variable equal to 1 if a direct route, from the depot to customer $i$ and back, is performed, $i \in N$. A valid upper bound on the number of routes is given by the optimal solution of:
\begin{subequations}
	\label{eq:UB}
\begin{equation}\label{eq:obj_UB}
\hbox{max} \sum_{i \in N}\alpha_i 
\end{equation}
\begin{align}\label{eq:cons_UB} \text{subject to: }
 \alpha_i r_i+\sum_{j=i}^{j=|N|}(d_{0j}+d_{j0})\alpha_j &\le T_E
&& \quad i \in N \\
\IL{\alpha_{i}} & \in \{0,1\} && \quad i \in N 
\end{align}
\end{subequations}
Formulation \eqref{eq:UB} counts the maximum number of direct routes that can be performed within $T_E$. Let us call $UB_{route}$ the optimal solution of \eqref{eq:UB}.
\begin{prop}
\label{theo:UB}
Given a vector of deterministic release dates $r_i$ and deadline $T_E$:
\begin{enumerate}
\item $UB_{route}$ is a valid upper bound on the maximum number of routes  \IL{$|\cK|$ for the OP-rd};
\item 
Any optimal solution of \eqref{eq:UB} is a feasible solution for \eqref{eq:base}.
\end{enumerate}
\end{prop}
\proof{Proof.} See Online Appendix D.

\subsection{Partial Characterization of Optimal Policy} \label{sec:partial}
The following proposition provides sufficient conditions to partially characterize an optimal policy. 
\begin{prop} \label{prop:charact} 
\item Let $\ell$  be the largest number of parcels from $N^{unserved}_e$ that can be delivered within $[t_e,T_E]$. If there exists a subset $N_{sol} \subseteq N^{known}_e$ of parcels that can be delivered within $[t_e,T_E]$ such that $|N_{sol}|= \ell$, then the optimal policy is to leave the depot immediately and distribute the parcels from $N_{sol}$. 
\end{prop}
\proof{Proof.} Trivial.

Hence, Proposition \ref{prop:charact} provides sufficient conditions under which it is optimal to leave the depot and (partially) distribute the available parcels. Checking these conditions  requires solving an NP-hard optimization problem. Indeed, given a decision epoch $e$,  and the set of parcels from $N^{unserved}_e$, one has to find an optimal solution of the OP with respect to $N^{unserved}_e$ with a maximum route duration equal to $T_E - t_e$ and with all parcel profits set to one. In case the value of the optimal solution is $\ell > |N^{known}_e|$, then it immediately follows that the condition of Proposition \ref{prop:charact} is not satisfied. Otherwise, one has to solve a new OP on the set $N^{known}_e$, and the condition of Proposition \ref{prop:charact} holds if and only if the value $\ell$ of the optimal solution remains unchanged. Note that, in preliminary experiments, we also implemented a different approach for verifying Proposition \ref{prop:charact}, where a single OP was solved on $N^{unserved}_e$ where profits were set to 1 for parcels in $N^{unknown}_e$, and to $1-\epsilon$ for parcels in $N^{known}_e$, with a sufficiently small value of $\epsilon$. However, the \IL{former} approach 
turned out to be much more efficient than the latter one.
Finally, we note that, for small values of $|N^{known}_e|$ and large values of $[t_e,T_e]$, it is unlikely that Proposition \ref{prop:charact} holds, and in this case, we skip checking whether the underlying conditions hold (see Section \ref{sec:computational_result} for more implementation details). 



\section{Reinforcement Learning Solution Approaches}
\label{sec:sol_approach}
The solution approaches are based on Monte-Carlo simulation, where scenario generation is performed at each decision epoch $e$ according to the state $S_e$, and on an approximation of the future routes, described in Section \ref{sec:Batch}. Thus, an online 
approximation scheme is developed, which learns from updated information.  

According to formula \eqref{eq:value_func_obj}, in each decision epoch, the goal is to find an action $\mathcal{X}_e$,  i.e., the ``route 0'' (the route that starts now), which maximizes the sum of the immediate reward $C(S_e, \mathcal{X}_e)$, corresponding to the number of parcels delivered in the route 0,  and the expected number of parcels delivered by future routes. 
To anticipate future routes, 
at each decision epoch $e$, we generate a set of  scenarios $\Omega$ associated with unserved customers predicting their release dates $\tilde r_e$. For each scenario, we approximate future routes using the batch approach described in Section \ref{sec:Batch}. 
This approximation is used in both reinforcement learning approaches to determine the decision-making policy, which are:
\begin{itemize}
    \item \emph{Value Function Approximation with a Consensus Function (VFA-CF):} for each scenario,  we run a deterministic ILP model to determine the set of known requests that should be served with a route starting immediately and the set of requests included in future routes. Then, a consensus function based on the solutions obtained over all scenarios defines the set of known requests served immediately (if any), i.e., with a route leaving from the depot at time $t_e$.
    \item \emph{Value Function Approximation with a Two-Stage Stochastic ILP Model (VFA-2S):} instead of looking at each scenario independently, we build a two-stage stochastic ILP model in which the first stage 
    determines a set of requests to be delivered at time  $t_e$, while the second stage concerns the expected number of parcels delivered by future routes by integrating all scenarios with equal probability. 
\end{itemize}

As stated before, both methods are hybridized with a one-step look-ahead strategy. One-step look-ahead involves looking one step ahead and considering all possible actions that can be taken from the current state, and selecting the action that maximizes the immediate reward plus the expected reward of the next state. VFA involves estimating the value function for each state. We combine these two techniques to exert their advantages in the following way. We approximate the future impact of the current action rather than simulating all possible actions at the next state. At each decision epoch, we first approximate future routes using a batch approach. We then examine the available actions in the current state and evaluate each action's immediate reward (i.e., the number of parcels delivered by a route leaving now) and the next state's expected reward (i.e., the expected number of parcels delivered by future routes). Finally, we select the action that maximizes the sum of the immediate reward and the expected future reward.

As decisions for the DOP-rd must be taken quickly in practical contexts (e.g., in seconds/minutes), the considered approximation of future routes can be relevant to substantially speeding up computation and making the approach applicable in practice. Without the approximation, using the VFA-CF/VFA-2S with exact methods to determine the entire sequence of future routes would be intractable.

The two approaches are described in Sections \ref{policyFunc} and \ref{look-aheadFunc}, respectively (an overview of additional notation used is provided in Table 2 in Online Appendix A).
Both ILP methods are based on state-of-the-art branch-and-cut techniques. Note that, in both approaches, Proposition \ref{prop:charact} is checked at each decision epoch (before solving the corresponding ILP) and the procedure is stopped in case it is satisfied, as the best policy is identified by the OP solution defined by the proposition.

\subsection{The Batch Approach}
\label{sec:Batch}
At each decision epoch $e$, we are given the sets of known customers $N_e^{known}$ and unknown customers $N_e^{unknown}$. As for $N_e^{unknown}$, we generate a set of scenarios $\Omega$  where a scenario $\omega  \in \Omega$ represents a possible realization of the release dates for $i \in N_e^{unknown}$, sampled according to the distribution associated with $\tilde{r}^e_i$. Assume that we are given a scenario $\omega\in\Omega$, 
we propose to determine the route serving a subset of requests from $N_e^{known}$ (the route 0), while approximating the future routes and the respective rewards obtained from serving the remaining ones from $N_e^{unserved}$  until the deadline. We assume that there will be $|K|$ additional future routes serving customers in $N_e^{unserved}$, and for their approximation we use a batch approach. The term “batch” refers to a set of requests scheduled for future delivery, with each batch being served by a dedicated future route. On the other hand, the “batch approach” refers to the specific approach employed to create batches. 

First, we assume that each future route will serve at most $\rho$ requests. This will allow us to obtain an approximation of the duration of future routes easily. Also, it is reasonable to assume that in practical applications, the company will not wait too long (i.e., it will not wait for more than $\rho$ parcels to become available at the depot) before dispatching the vehicle. Thus, by properly tuning the value of $\rho$ (as shown in Online Appendix H.2), one might obtain a good approximation of the duration of future routes. We apply the formula of \cite{daganzo1984length} to estimate the duration of each future route. This formula is commonly used in literature thanks to the fact that it is easy to implement and provides good estimation ( \cite{jabali2012continuous} and \cite{van2019delivery}). According to the formula, given the Euclidean area $\mathcal{A}$ containing the locations of unserved customers and $\rho$, under the assumption that the locations are scattered uniformly and independently, the expected tour duration $T_D$ is $0.75\sqrt{\mathcal{A}\rho}$. 
This approach particularly fits the deliveries in densely populated areas, where dispatchers are assigned a restricted number of parcels in each delivery, thus route duration is rather stable and easy to estimate. 

The input to the batch approach is given by the duration of a batch route $T_D$ (calculated through Daganzo's formula mentioned above), the distribution of release dates, and $\rho$, the maximum number of requests served in each future route (represented by a batch). All these parameters, together with $|\Omega|$ (the number of scenarios sampled to approximate the future release dates), are referred to as $\theta$ in the following. Then, for each scenario $\omega \in \Omega$, the batch algorithm works as described in the following. For ease of reading, we omit the index $\omega$ in the notations.

The batch approach (whose pseudo code is given in Algorithm \ref{simu:Pro}) starts by ordering the unserved requests according to their release dates associated with the realization of scenario $\omega$ (Step 3).  Recall that $\widetilde{r}^e_i$ are random variables representing the release dates for $i \in N_e^{unserved}$. For a given scenario $\omega$ we will denote by $r^e_i$ the realization drawn from the release date distribution for $i \in N_e^{unknown}$ and actual release dates for $i \in N_e^{known}$. The algorithm then divides requests into batches. To serve as many requests as possible, the last route (or batch) is supposed to serve the requests arriving at the latest that can still be feasibly served before the deadline. Thus in Step 5 the selection is made backwards, i.e., starting from the request with the latest release date, and checking whether it can be included in the last route, i.e., if its release date is not larger than the earliest possible starting time of the last route, which is $T_E-T_D$. If true, the number of requests $q$ in the assigned batch is increased by one. Subsequently, if the request is known, the batch $k$ is added to $K_0$; otherwise, request $i$ is allocated to the batch $k$ and the counter $\rho_k$ is updated. After that, in Step 12, if the batch size limit is reached ($q=\rho$) or all requests are processed ($i=1$), the start and end time of the batch are updated, and counters are reset (Steps 13-15).
The algorithm iteratively follows the procedure, checking the next sorted requests. Then a new route is created (the second last) associated with a starting time equal to $T_E-2T_D$, and so on. In this manner, we continue checking the next requests until all are examined or the start time of a new batch $t_k$ is less than or equal to the time at decision epoch $e$ (Step 6).
To be consistent with the earlier notation, route 0 refers to the route supposed to leave the depot immediately.
We introduce the new notation for a detailed description of the algorithm:
\begin{itemize}
    \item $K$: the set of indices for all the batches created (we use the same notation as the set of future routes since each batch will be served by a future route. However, the IDs are annotated reversely, i.e., batch 1 is served by the last route, batch 2 is served by the second last route, etc);
    \item $K_0$: the set of indices for batches with positive spare capacities ($K_0 \subseteq K$). We say that a batch $k$ has a positive spare capacity if the number of requests in $N_e^{unknown}$ assigned to $k$ is less than $\rho$;
    \item $k(i)$: the index of the batch  request $i$ is assigned to,  $i \in N_e^{unknown}$; 
    \item $\tau_{start}^k$: the starting time of the route serving batch $k\in K$;
    \item $\tau_{end}^k$: the ending time of the route serving batch $k\in K$.
    \item  $\rho_k$: the number of requests in $N_e^{unknown}$ assigned to the route serving batch $k\in K$.
\end{itemize}
The scheme of the batch algorithm is presented in Algorithm \ref{simu:Pro}.  
\begin{algorithm}[tb!]
    \caption{The batch algorithm
    \label{simu:Pro}}
     \small
    \linespread{0.88}\selectfont
 \begin{algorithmic}[1]

    \Require{Time $t_e$; 
    Unserved requests $N_e^{unserved}=N_e^{known}\cup N_e^{unknown}$; Release dates \{$r^e_i$, $i \in N_e^{unserved}$\};  Deadline
    $T_E$; Maximum number of requests allowed in a batch $\rho$; 
     Duration of a batch route $T_D$; 
    \Ensure
    Set $K_0$; Number of batches $|K|$;
    Batch indices \{$k(i)$, $i \in N_e^{unknown}$\};
    \{($\tau_{start}^k, \tau_{end}^k,\rho_k)$, $k \in K$\};
    }
    \State 
    $n \gets |N_e^{unserved}|$; 
    $K \gets \emptyset$; $K_0 \gets \emptyset$; 
    $k \gets 1$; 
    $t_k \gets T_E$ - $T_D$; \IL{$\rho_1 \gets 0$}
    \State  $k(i) \gets 0$,  for all $i \in N_e^{unknown}$;
    \State Sort the requests from $N_e^{unserved}$ in increasing order of 
    release dates $r_i^e$;
    \State Initialize the number of requests assigned to the $k$-th batch $q \gets 0$;
    \For
    {$i \gets n$ to $1$}
    
    \IfThen{$t_k \leq t_e$}
        \textbf{break} 
    \If{release date $r^e_i \leq  t_k$}
            \State $q \gets q+1$
            \If{$i \in N_e^{known}$}
                \IfThen{$k \not \in K_0$}
                $ K_0 \gets K_0 \cup \{k\}$
            
            \Else 
                \quad $k(i) \gets k$; $\rho_k \gets \rho_k+1$
            \EndIf

        \If{$q = \rho$ or $i = 1$}
            \State $K \gets K \cup$ $\{k\}$
            \State $\tau_{start}^k \gets t_k$; $\tau_{end}^k \gets t_k+T_D$
            \State $k \gets k + 1$;
            $t_k \gets t_k - T_D$; $q \gets 0$; \IL{$\rho_k \gets 0$}
            
        \EndIf
    \EndIf
    \EndFor
    \State
    \Return {$K_0, K$, \{$k(i)$, $i \in N_e^{unknown}$\},  and \{$(\tau_{start}^k$, $\tau_{end}^k, \rho_k)\}_{k\in K}$}

 \end{algorithmic}
\end{algorithm}
Note that in Algorithm \ref{simu:Pro}, customers in $N_e^{known}$ are treated differently than those in $N_e^{unknown}$. For each customer $i \in N_e^{unknown}$ such that $r_i^e \leq T_E - T_D$ the algorithm determines the unique batch $k(i)\in K$ to which $i$ is assigned. However, the customers in $N_e^{known}$ can be assigned to any potential batch $k \in K$ with some spare capacities, i.e., which contains less than $\rho$ unknown requests. The subsets of the batches with spare capacities are denoted by $K_0$, and a unique assignment of a request $i \in N_e^{known}$ to some $k \in K_0$ is obtained through the optimization models proposed in Sections  \ref{policyFunc} and \ref{look-aheadFunc}.

\begin{table}[tb!]
\caption{\IL{Input instance for the batch approach (\textit{release date (RD) 
0 means the parcel $i$ is at the depot, i.e., $i \in N_e^{known}$})}}
\label{table:batchTitle}
\centering
\scalebox{0.68}{
\begin{tabular}{|p{50mm}|p{7.20mm}|p{5.50mm}|p{5.50mm}|p{5.50mm}|p{5.50mm}|p{5.50mm}|p{5.50mm}| p{5.50mm}|l|}
\hline
\multirow{2}{*}{\textbf{Total number of  requests}} & \multirow{2}{*}{9} & \multicolumn{7}{l|}{\multirow{2}{*}{\textbf{\begin{tabular}[c]{@{}l@{}}Max number of requests\\  in a batch $\rho$\end{tabular}}}} & \multirow{2}{*}{3} \\  &    & \multicolumn{7}{l|}{}   &                    \\ \hline
\textbf{Number of batches}                         & 3                  & \multicolumn{7}{l|}{\textbf{Duration of a batch route $T_D$}}  & 1                  \\ \hline
\textbf{Returning time of route 0}                 & \IL{TBD}                  & \multicolumn{7}{l|}{\textbf{Deadline}}            & 16                 \\ \hline
\textbf{ID of the requests}                        & 1                  & 2                    & 3                    & 4                     & 5                     & 6                    & 7                    & 8& 9                  \\ \hline
\textbf{RD of the requests}                       & 0                  & 0                    & 0                    & 14                    & 14                    & 15                   & 15                   & 15 & 15                \\ \hline
\end{tabular}}
\end{table}
\begin{figure}[tb!]
    \centering    \includegraphics[width=0.43\textwidth]{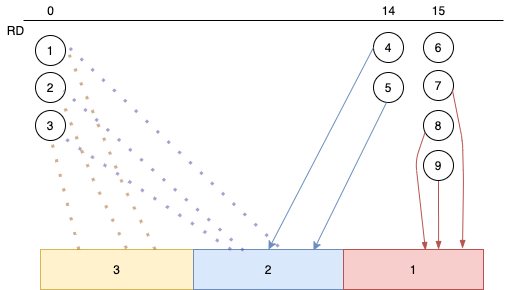}
    \caption{The assignment of requests to batches (\textit{a dashed line with no arrow connects known requests with batches from $K_0$, indicating potential assignments of each $i \in N_e^{known}$ to some $k \in K_0$)} }
    \label{fig:batch}
\end{figure}
\begin{table}[tb!]
\centering
\caption{The batches}
\label{table:batchContent}
\scalebox{0.7}{
\begin{tabular}{|l|c|c|c|c|c|c|c|c|}
\hline
\textbf{Batch   number}         & \textbf{3} & \multicolumn{3}{c|}{\textbf{2}} & \multicolumn{3}{c|}{\textbf{1}} \\ \hline
\textbf{Unknown requests assigned}      & {-}  & {-} & 4        & 5                & 7      & 8   & 9   \\ \hline
\textbf{Max RD served}         & 0          & \multicolumn{3}{c|}{14}         & \multicolumn{3}{c|}{15}         \\ \hline
\textbf{Start time of the tour} & 13         & \multicolumn{3}{c|}{14}         & \multicolumn{3}{c|}{15}         \\ \hline
\textbf{End time of the tour}   & 14         & \multicolumn{3}{c|}{15}         & \multicolumn{3}{c|}{16}         \\ \hline
\textbf{k(i)}                   & -          & -         & $k(4)=2$ & $k(5)=2$ & $k(7)=1$  & $k(8)=1$ & $k(9)=1$ \\ \hline
\textbf{$K_0$}                   & \checkmark          & \multicolumn{3}{c|}{\checkmark}          & \multicolumn{3}{c|}{\xmark}          \\ \hline
\end{tabular}}
\end{table}

To explain how the batch approach works, a toy example with 9 requests is displayed in Table \ref{table:batchTitle}, and the output is shown in Table \ref{table:batchContent} and Figure \ref{fig:batch}.
Note that in Table \ref{table:batchTitle} we 
also consider the ending time of route 0 (which is a decision variable defined in the ILP model, where it is denoted as  $\tau_{end}^0$, see Section \ref{ILP: model}), that is the route leaving immediately from the depot. Its ending time will allow scheduling only those batches that can start after route 0 is completed. The role of this parameter will become clearer in the following sections, where we explain how the output of the batch algorithm is integrated with the two solution approaches. As the deadline is 16, the first batch route should be completed no later than 16. Each route lasts one time unit (according to Daganzo's formula), so the tour's start time is 15. A batch can serve three requests at most. As the value of the solution does not depend on the subset of requests served but just on their number, we arbitrarily assign to batch 1 (which is mapped as the last future route) the three requests with the highest index. For batch 2, its ending time is 15, so its starting time is 14. Thus the two requests with sampled release dates equal to 14 are assigned to the batch. As batch 2 has a residual capacity of 1, any known requests (with IDs 1, 2, and 3) can be assigned to it. In addition, requests 1, 2, and 3 can be assigned to batch 3. Table \ref{table:batchContent} describes the content of batches and displays the values of $k(i)$. It also shows that batches 2 and 3 are the ones that can be used to serve known requests (thus, they belong to set $K_0$) with the check mark on the last row. If a request $i \in N_e^{unknown}$ is not assigned to any batch (for example, because its release date is larger than the starting time of batch 1 or because there are no sufficient batches to include all requests), then $k(i)=0$. 

The optimality conditions for the batch approach under some simplifying conditions are shown in Online Appendix E. Hence, by combining scenario generation and the batch approach, we approximate the value of future routes as a one-step look-ahead policy in which future events are condensed in a single step, and future actions are evaluated through batch approach. 
 
Finally, to evaluate if each future route is executable, we need to check whether the starting time is greater than the ending time of the route leaving immediately from the depot, that is route 0. This is done through the optimization models presented in the following section.

\subsection{Value Function Approximation with a Consensus Function (VFA-CF)}
\label{policyFunc}

There are four fundamental ways of deriving policies in MDPs (\cite{powell2014clearing}), one of which is based on value function approximations (VFAs) where the whole term representing the value function of the future states in (\ref{eq:value_func_obj}) is replaced by an approximation. Our solution approaches fall into this category. VFA-CF relies on the batch approach to approximate future routes, and uses a consensus function working according to the following rule:
    \begin{equation}\label{eq:value_func}
    \tilde X^{VFA-CF}_e(S_e \mid \theta) = \Phi(\tilde{X}^\omega_e)_{\omega \in \Omega} 
    \end{equation}
\begin{equation} \label{eq:determinisic}\text{where}\qquad 
    \tilde{X}^\omega_e = \arg \max_{\mathcal{X}_e \in X(S_e)}\{C(S_e, \mathcal{X}_e) + \gamma \tilde{V}_{e+1}(\tilde{S}_{e+1} \mid S_e, \mathcal{X}_e, \omega) \}  
       \end{equation}
        is the optimal action for the given value function approximation $\tilde{V}_{e+1}(\tilde{S}_{e+1}\mid S_e, \mathcal{X}_e, \omega)$ ($\tilde{S}_{e+1}$ is the approximation of state at $e+1$ and $\tilde{V}$ is an estimate of the value of being in state $\tilde{S}_{e+1}$), under the realization of scenario $\omega$.
       As defined in Section \ref{sec:Batch}, $\theta$ captures all tunable parameters of the model. 
       Function $\Phi$ is the consensus function {for determining the action, considering the best actions over all scenarios (see Section \ref{consensus}). Thus, the main idea of the proposed  VFA-CF is that for each scenario $\omega$, the best action $\tilde{X}^\omega_e$ is determined, according to the method described in Section \ref{ILP: model}. This consists in solving a deterministic ILP where the goal is to determine the route that should leave immediately from the depot (which might be empty) by optimizing the function given by the sum of the immediate reward associated with this route plus the estimated number of requests to be served in future routes (determined through the batch approach described in Section \ref{sec:Batch}). Then, a consensus function is applied to determine the best action according to the solution obtained on each of the scenarios. Note that, as we solve an ILP for each scenario independently, there is no uncertainty in the input data to the ILP: indeed, the release dates correspond to the realizations associated with the given scenario, and we refer to this model as \textit{a deterministic ILP model} below}.
       
       \subsubsection{Consensus Function}
\label{consensus}
We define the consensus function, denoted by $\Phi$ in formula (\ref{eq:value_func}), as the policy to determine which requests are served in route 0 while considering all solutions associated with different scenarios.
The scheme is reported in Algorithm \ref{simu:Consensus}. The function checks solutions $\tilde{X}_e^\omega$ over all scenarios $\omega \in \Omega$, by focusing on the requests in $N_e^{known}$ served in route 0 in each scenario. The idea is to determine which subset of these requests should be served in route 0 at the current decision epoch. Specifically, a request $i \in N_e^{known}$ is served if it is included in route 0 in at least $\lceil\lambda|\Omega|\rceil$ scenarios, with $0 < \lambda < 1$. If at least one request is included in the route, then a Traveling Salesman Problem (TSP) is solved over all requests included to determine the best sequence, thus determining the values of the arc variables $x^e_{ij}$, $i, j \in N_e^{known}\cup\{0\}$ associated with the action $\mathcal{X}_e$. If the duration of the TSP exceeds $T_E$, an OP model is solved to determine the requests to be served, followed by the solution of TSP over the selected requests to get the best visit sequence.  Otherwise, the decision is to wait at the depot until the next decision epoch $e+1$. 
\begin{algorithm}[tb!]
\caption{Consensus Function \label{simu:Consensus}}
\linespread{0.85}\selectfont
\begin{algorithmic}[1]
\Require{  
Time $t_e$; $N_e^{known}$;
a set of scenarios $\Omega$; solution $\tilde{X}^\omega_e$ of scenario $\omega$;}
\Ensure a route $R$; 

\State Initialize the frequency of all requests: $F_e(i) \gets 0\  i \in N_e^{known}$; Initialize the set of requests to be served as $I_{sol} \gets \emptyset$;
\For{$\omega \in \Omega$}
    {
        \State $I_\omega \leftarrow$  the set of requests served in route $0$ of the solution $\tilde{X}^\omega_e$ of scenario $\omega$. 
        \For{$i \in I_\omega$}  $F_e(i) \gets F_e(i) + 1$
        \EndFor
    }
\EndFor
\For{$i \in N_e^{known}$}
    \IfThen{$F_e(i) \geq \lambda|\Omega|$} $I_{sol} \gets I_{sol} \cup \{i\}$
\EndFor
\State $R \leftarrow$ TSP over $I_{sol}$
\If{length of R $> T_E-t_e$} 
    \State $I_{sol} \gets$ OP over $I_{sol}$
    \State \Return $R \leftarrow$ TSP over $I_{sol}$
\Else
    \State \Return R
\EndIf
\end{algorithmic}
\end{algorithm}

     
\subsubsection{Deterministic ILP Model}
\label{ILP: model}

We now describe the optimization model used to find the decision $\tilde X^\omega_e$ given in \eqref{eq:determinisic}.
The idea is that only the reward of route $0$  (which corresponds to $C(S_e, \mathcal{X}_e)$) is calculated with the detailed routing information (since the route is to be executed immediately in the case at least one request is assigned to it), while the value and the duration of future routes are estimated through the batch approach (Section \ref{sec:Batch}). 
We define as $K = \{1, ..., \lvert K \rvert\}$  the set of indices for all the batches created where each batch is served by a future route. 
The parameters considered are the following:
\begin{itemize}
    \item $\tau_{start}^0$, a parameter indicating the starting time of the immediate route, so its value equals  $t_e$. 
	
	\item $A^{known}$, it includes all arcs $(i, j)$ where $i, j \in N_e^{known}\cup \{0\}$ linking customers with known requests and depot. 
\end{itemize}

In addition, we have the parameters associated with future batches corresponding to the output of the batch approach (see Section \ref{sec:Batch}). 
The decision variables are the following:
\begin{itemize}
    \item $x_{ij}^0$: binary arc variable for route $0$, equal to 1 if the route traverses arc $(i, j) \in A^{known}$, 0 otherwise.
    \item $y_i^0$: binary variable associated with known request $ i \in N_e^{known}$, equal to 1 if customer $i$ is served in route $0$, 0 otherwise.
    \item $z_k$: binary variable equal to 1 if future batch $k \in K$ is  executed, 0 otherwise.
    \item $\tau_{end}^0$: continuous 
    variable denoting the ending time of route 0. 
	\item $\pi_i^k$: binary variable equal to 1 if $i \in N_e^{known}$ is served in batch $k \in K_0$, 0 otherwise.
\end{itemize}





The deterministic ILP model is given as follows.
In line with (\ref{eq:determinisic}), \IL{the first term in the objective function \eqref{eq:obj_ILP}}, namely $\sum_{i\in N^{known}_e}y_i^0$, maps $C(S_e, \mathcal{X}_e)$ representing the number of requests served by action $\mathcal{X}_e$. The sum of the two remaining terms in \eqref{eq:obj_ILP}, namely 
$\sum_{k \in K} \rho_k z_k$ and  
$\sum_{k \in K_0} \sum_{i \in N^{known}_e} \pi_i^k$ maps $\tilde{V}_{e+1}(\tilde{S}_{e+1} \mid S_e, \mathcal{X}_e, \omega)$, representing the approximated number of parcels delivered by future routes. Note that in case $\sum_{i\in N^{known}_e}y_i^0=0$, it means that route 0 does not serve any request and the decision is to wait at the depot until the next decision epoch $e+1$. Then for each scenario $\omega$, the corresponding ILP formulation is given as follows: 

{  
\begin{subequations}
	\label{eq:deterministic}
\begin{equation}\label{eq:obj_ILP}
\hbox{max}\ 
\sum_{i\in N^{known}_e}y_i^0 + \gamma \left[
\sum_{k \in K} \rho_k z_k +  \sum_{k \in K_0} \sum_{i \in N^{known}_e}\pi_i^k\right]
\end{equation}

\begin{align}\label{eq:inout_ILP}
\text{subject to:} \sum_{(i, j)\in A^{known}}x_{ij}^0 &= \sum_{(j, i)\in A^{known}}x_{ji}^0 = y_i^0
&& \quad i \in N_e^{known}\cup \{0\} \\
\label{eq:subtour_ILP}
 \sum_{i, j\in S}x_{ij}^0 & \leq \sum_{i \in S\setminus{\ell}}y_i^0 
&& \quad S \subseteq  \IL{N^{known}_e, |S|\geq 2,} \ell \in S \\
\label{eq:endStartRoute_ILP}
 \tau_{end}^0 &= \tau_{start}^0 + \sum_{(i, j) \in A^{known}}d_{ij}x_{ij}^0 \\
\label{eq:deadline_ILP}
 \tau_{end}^0 &\leq T_E \\
\label{eq:requestLimit_ILP}
 \sum_{k \in K_0}\pi_i^k+y_i^0 &\leq 1 && \quad  i \in N^{known}_e \\
\label{eq:wholeBatch_ILP}
\sum_{i \in N^{known}_e} \pi_i^k &\leq (\rho - \rho_k) z_k && \quad  k \in K_0 \\
\label{eq:routeK_ILP}
 z_{k+1} &\leq z_k && \quad  k = 1, \ldots |K|-1 \\
\label{eq:startK_ILP}
 \tau_{end}^0 - \tau_{start}^k &\leq   (T_E-\tau_{start}^k)(1-z_k) && \quad  k\in K \\
\label{eq:decivars}
 y_i^0, \pi_i^k, x_{ij}^0, z_k &\in \{0, 1\}  && \IL{i,j \in N^{known}_e, k \in K}
\end{align}

Besides, we initialize the model with GSECs  of size two 
and the constraints (\ref{eq:depot_ILP}) to strengthen the formulation:
\begin{align}
\label{eq:size2_ILP}
x_{ij}^0 + x_{ji}^0 &\le y_{i}^0 && \quad i,j \in N^{known}_e 
\\
\label{eq:depot_ILP}
\sum_{i \in \IL{N^{known}_e}} y_i^0 & \leq |N| y_0^0 &&
\end{align}
\end{subequations}

}
The objective function (\ref{eq:obj_ILP}) maximizes the total number of requests served by route $0$ plus the approximated number of requests served by future routes, distinguished in unknown and known requests, multiplied by the discount factor $\gamma$. Constraints (\ref{eq:inout_ILP}), (\ref{eq:subtour_ILP}) guarantee that route $0$ is a circuit connected to the depot and prevent generating subtours. Constraint (\ref{eq:endStartRoute_ILP}) determines the ending time of route 0, which is computed as its starting time plus the traveling time. 
Constraints (\ref{eq:deadline_ILP}) ensure that  route $0$ finishes before the deadline $T_E$.  Constraints (\ref{eq:requestLimit_ILP}) state that the same request cannot be served by both route $0$ and any future route serving a batch in $K_0$, which guarantees the unique assignment of each known request as mentioned in Section \ref{sec:Batch}. 
Constraints (\ref{eq:wholeBatch_ILP}) guarantee that each future route serving batches in $K_0$ 
does not serve more than $\rho$ requests. Note that these constraints are needed for batches in $K_0$ only, as any request in $N_e^{known}$ can potentially be assigned to each of these batches. Instead, for the batches in $K\backslash K_0$, the constraint is satisfied by construction from the batch approach. 
Constraints (\ref{eq:routeK_ILP}) state that if batch $k+1$ is executed, batch $k$ must be executed too. Constraints (\ref{eq:startK_ILP}) ensure that the future route serving batch $k$ can start only after finishing route $0$. Finally, \eqref{eq:decivars} define the variables domain.

 To solve this model, we implemented a branch-and-cut procedure in which GSECs \eqref{eq:subtour_ILP} are separated in each node of the branching tree (see Section \ref{sec:computational_result} for further details).

\subsection{Value Function Approximation with a Two-Stage Stochastic ILP Model (VFA-2S)}
\label{look-aheadFunc}
The second approach we propose is a two-stage stochastic ILP model using a VFA:
\begin{equation}\label{eq:value_function}
    \tilde{X}^{VFA-2S}_e(S_e \mid \theta) = \argmax_{\mathcal{X}_e\in X(S_e) 
    }\{C(S_e, \mathcal{X}_e) + \gamma\mathbb{E}[\widetilde{V}_{e+1}(\widetilde{S}_{e+1} \mid S_e, \mathcal{X}_e)]\} 
\end{equation}
where, as before, $\theta$ captures all tuning parameters, 
$\widetilde{S}_{e+1}$ is the approximation of state at the epoch $e+1$, and $\widetilde{V}_{e+1}(\widetilde{S}_{e+1})$  represents the approximation of the value function of being in state $\widetilde{S}_{e+1}$. As VFA-CF, it belongs to VFA since it also approximates the future impact (cost-to-go, or value function) of the current action (route 0) using the future routes returned by the batch approach, which will be explained in the following.

The action $\tilde{X}^{VFA-2S}_e(S_e \mid \theta)$ is determined through the solution of a two-stage stochastic ILP model as presented in the following.

\subsubsection{Two-Stage Stochastic ILP Model}
\label{2-stage: model}
The modeling idea is similar to the one presented in Section \ref{ILP: model}, i.e., routing decisions are taken only for route 0 while future routes are approximated with the batch approach. The main difference with respect to the approach presented in Section \ref{policyFunc} is that a single Two-Stage Stochastic Integer Linear Programming Model (2-SIP) is solved using all scenarios in $\Omega$, instead of solving $|\Omega|$ deterministic ILPs (separately per scenario). Specifically, each scenario $\omega \in \Omega$ is associated with the realization of the release dates of the unknown requests $\{\widetilde{r}^e_i\}_{i \in N_e^{unknown}}$, denoted as  $\widetilde{r}^{e1}, ...,  \widetilde{r}^{e\omega}$. Also, we denote as $p_\omega$ ($0 < p_\omega < 1$) the probability associated with scenario $\omega \in \Omega$ ($\sum_{\omega \in \Omega} p_\omega=1$). We first run the batch algorithm on the realization of each scenario to pre-calculate the following parameters:

\begin{itemize}
    \item $K^\omega$ is the set of indices for all the batches created in scenario $\omega$, $\omega \in \Omega$,
    \item $K_0^\omega$ is the subset of batches with spare capacities in scenario $\omega \in \Omega$,
    \item $\tau^{k\omega}_{start}$ is  the starting time of batch $k\in K^\omega$ in scenario $\omega\in \Omega$,
    \item  $\rho_k^\omega$ is the number of unknown requests assigned to batch $k\in K^\omega$ in scenario $\omega\in \Omega$.
\end{itemize}


As for the decision variables, we keep variables $x_{ij}^0$, $y_i^0$ and  $\tau_{end}^0$ as defined in the deterministic model (\ref{eq:deterministic}).
In addition, we have
the scenario-related variables as follows:
\begin{itemize}
    \item $\pi_i^{k\omega}$: $i \in N_e^{known}, k \in K_0^\omega$, binary variable equal to 1 if known request $i$ is served in batch  $k$ in scenario $\omega \in \Omega$, and 0 otherwise.
    \item $z_k^\omega$: Binary variable equal to 1 if batch $k\in K^\omega$ in scenario $\omega \in \Omega$ is used, and 0 otherwise. For ease of reading, we will simply write $z_{k}^\omega$ instead of $z_{k^\omega}^\omega$.
\end{itemize}

The deterministic equivalent of the 2-SIP is given by model (\ref{eq:model_STO_DEP}). The objective function is: 

\begin{subequations}\label{eq:model_STO_DEP}\begin{equation}
\label{eq:obj_STO_DEP}
 \hbox{max} \sum_{i\in N^{known}_e}y_i^0 + \gamma\sum_{\omega\in \Omega}p_\omega(\sum_{k \in K^\omega}\rho_k^\omega z_k^\omega + \sum_{k \in K_0^\omega}\sum_{i \in N^{known}_e}\pi_i^{k\omega}) 
\end{equation}
The first-stage decision corresponds to determining route 0. In line with the function (\ref{eq:value_function}),  the number of requests served by route 0 is determined by $\sum_{i\in N^{known}_e}y_i^0$ and matches the immediate reward $C(S_e, X_e)$.
The second-stage decisions indexed by $\omega$ are associated with determining the requests served in future routes of scenario $\omega$. Hence the second term in the objective function represents the expected value of future states with respect to the set of scenarios $\Omega$. The constraints are:
\begin{align}
\text{ s.t. \qquad $ (x_{ij}^0,y_i^0,\tau_{end}^0)$  satisfies \eqref{eq:inout_ILP}-\eqref{eq:deadline_ILP}, \eqref{eq:size2_ILP}-\eqref{eq:depot_ILP}
}\notag  \end{align}
\begin{align}
\sum_{k \in K_0^\omega}\pi^{k\omega}_i  &\leq 1 - y_i^0  && \quad \omega \in \Omega, i \in N^{known}_e \label{eq:requestLimit_STO_DEP}\\
 \sum_{i \in N^{known}_e} \pi^{k\omega}_i &\leq (\rho-\rho_k^\omega) z_k^\omega && \quad  \omega \in \Omega,  k \in K_0^\omega \label{eq:wholeBatch_STO_DEP}\\
 z_k^\omega &\geq z^\omega_{k+1} && \quad  \omega \in \Omega,  k= 1, \ldots |K^\omega|-1 \label{eq:routek_STO_DEP}\\
  \tau_{start}^{k\omega} + (T_E-\tau_{start}^{k\omega})(1-z_k^\omega) &\geq \tau_{end}^0 && \quad  \omega\in \Omega,  k\in K^\omega \label{eq:startK_STO_DEP} \\
 \pi_i^{k\omega}, z_k^\omega &\in \{0, 1\} && \omega \in \Omega, i\in N^{known}_e, k \in K^\omega \label{eq:decivars11_DEP} \\
x^0_{ij}, y^0_i &\in \{0,1\} && i,j \in N_e^{known} \label{eq:decivars12_DEP}
\end{align}
\end{subequations}

The objective function \eqref{eq:obj_STO_DEP} maximizes the total number of requests served by route 0 plus the expected number of requests served by future routes. The constraints  \eqref{eq:requestLimit_STO_DEP} - \eqref{eq:startK_STO_DEP}
have a similar meaning as constraints \eqref{eq:requestLimit_ILP}-\eqref{eq:startK_ILP} of Formulation (\ref{eq:deterministic}), with the difference that they are replicated over all scenarios. Finally, \eqref{eq:decivars11_DEP} and \eqref{eq:decivars12_DEP} define the variables domain.

In order to solve this 2-SIP model, we implement a branch-and-cut procedure, similar to the one of the deterministic ILP model, with a dynamic separation of GSECs \eqref{eq:subtour_ILP}. If at least one request is included in such obtained route 0, then a TSP is solved over all requests included to determine the best sequence formed by arcs $x^e_{ij}$, $i, j \in N_e^{known}\cup\{0\}$. 

\section{Computational Experiments}
\label{sec:computational_result}
In this section, we describe the computational experiments we carry out to validate the performance of the two proposed approaches. The code is developed in Python 3.7 and CPLEX version 20.1.0.0 is used to solve all ILPs, run on a single thread with a memory usage limit set to 2GB. 
 We run the simulations on HP Z4 G4 Workstation with Intel(R)
Xeon(R) W-2255 CPU @ 3.70GHz. 

The maximum time interval between consecutive decision epochs $\phi$ is set to 10, which is a sufficiently small value to avoid long waiting times at the depot according to the input instances used in the computational campaign (described in Section \ref{sec:instances}).  

As for implementation details, we remark on the following. 
Seven branch-and-cut algorithms are implemented: a) to calculate upper bounds with perfect information, b) to check whether the conditions of the partially optimal policy are satisfied, c) to calculate the consensus function of the VFA-CF approach, d) to calculate the value function approximation in the VFA approaches (VFA-CF and VFA-2S), e) to get the route for the ME approach {described in Section \ref{sec:myopic}} and f) to calculate the heuristic solution based on the modified OP-rd {described in Section \ref{sec:OPrd_PE}}. All these algorithms rely on the separations of GSECs, which are implemented using lazy cutcallback of the CPLEX solver. To benefit from general-purpose cuts generated by the CPLEX solver, we collect the cuts separated at the root node of the branching tree and then restart the ILP. We set a time limit for the solution of each branch-and-cut as follows: VFA-CF: 5 minutes, VFA-2S: 10 minutes, ME: 10 minutes, and OPrd\_PE: 2 minutes. Note that the deterministic model is allocated with a shorter computing time as it is solved for each scenario. As for the OP-rd heuristic model, the maximum time assigned is set to 2 minutes to be consistent with a practical application setting where a short time is available to take decisions in every decision epoch. All other CPLEX parameters are left at their default values. 
The call to the partial characterization of the optimal policy is skipped in each decision epoch of VFA-CF and VFA-2S when $|N^{known}_e|<0.25|N^{unserved}_e|$ and remaining time $[t_e,T_E] > 0.75T_E$. We consider that for such settings there are small chances that conditions of Proposition \ref{prop:charact} are satisfied.
%
%
Concerning the calculation of upper bounds using model \eqref{eq:base}, we provide CPLEX with a starting feasible solution corresponding to the solution of model \eqref{eq:UB}, and the time limit is set to one hour.

The remainder of this section is organized as follows. We first describe the input instances in Section \ref{sec:instances} and in Section \ref{sec:benchmarks} we present the benchmark approaches we use for comparisons including two myopic approaches (Section \ref{sec:myopic}) and an approach considering point estimation of future information (Section \ref{sec:OPrd_PE}). Finally, Section \ref{sec:compute} compares the two proposed approaches with the benchmark approaches (Section \ref{sec:final_results}) and the experiments for two vehicles (Section \ref{2-vehicle}). 

Additional computational results can be found in Online Appendices including the results related to the tuning of hyper-parameters (Appendix H), which are: the number of scenarios $|\Omega|$, the batch size $\rho$, i.e., the maximum number of parcels included in each future route as used in the batch approach, the discount factor $\gamma$, and parameter $\lambda$ used in the consensus function to determine which requests should be served in route 0 (see Section \ref{consensus}). The results associated with the three values of $|\Omega|\in\{10, 30, 50\}$ show that on average the policy quality is improved when more scenarios are used, but the running time is increased. We opt for $|\Omega| = 30$ to balance both performance aspects. The tuning on four different values of $\rho$ was conducted for $\rho \in \{ 5, 10, 15, 20\}$. The results show that for both VFA-CF and VFA-2S, on average, the performance improves when increasing $\rho$ from 5 to 15, but it deteriorates when increasing to 20. So we set  $\rho$ to $15$. When it comes to discount factor $\gamma$, we do not see a significant difference among the values of $\gamma \in \{ 0.7, 0.8, 0.9, 1.0\}$, and we choose $\gamma=0.8$ for both policies as it gives slightly better results. Finally, the value of parameter $\lambda$ is set to 0.5. We also perform experiments with $\lambda = 0.4$ and 0.6, and the results show no remarkable difference with respect to $\lambda = 0.5$. 
Hyper-parameter tuning is performed on a separate benchmark set of instances that have been excluded from the testing phase (see Section \ref{sec:instances}). Online Appendix G illustrates the benefit of integrating partial characterization of optimal policy within VFA-CF/VFA-2S. Given the benefits of PC and the fact that computing time is reduced or remains reasonable, we
show the results of VFA-CF/VFA-2S with PC in section \ref{sec:compute}.

\subsection{Input Instances}\label{sec:instances}
We use the set of benchmark instances described in \cite{archetti2020dynamic} that are derived from Solomon's instances (\cite{solomon1987algorithms}). These instances are publicly available at \cite{BenchmarkInstances}. Specifically, \IL{instance classes} $C1$, $C2$, $R1$ and $RC1$ are considered. Each instance contains 50 customers and is associated with three parameters $\beta$, $\delta$ and $T_E$. The first two parameters, $\beta$, $\delta$, define the dispersion of release dates and the percentage of customers with dynamic release dates, respectively. We refer to Online Appendix F for more details.
Parameter $T_E$ is the deadline of the delivery service (not defined in the problem studied in \cite{archetti2020dynamic}), which is determined as follows. We first determine the latest actual arrival time of all requests and denote it as $T_{standard}$. Then we generate four instances with $T_E$ equal to $c\, T_{standard}$, where $c \in \{0.6, 0.8, 1.0, 1.2\}$.
For each customer input data and each choice of $\beta$, $\delta$, and $c$, five different instances (with different seeds) are generated in which two are used for hyper-parameter tuning (288 instances), and three are used for the final tests (432 instances). In our instances, we have some parcels
available in the beginning, and they could be dispatched immediately (if necessary). To test the scalability of the proposed approaches, we created 36 additional instances containing 100 customers, where the customers' locations are generated by combining instances C101 and RC101, and dynamic release dates are generated by combining the ones of instances C101 and RC101 accordingly. The instances, along with a detailed description of the instance generation procedure, can be found at \cite{DataFor100Customers}. 

\subsection{Benchmark approaches}
\label{sec:benchmarks}

In the following, we present three benchmark approaches used for comparing with VFA-CF/2S.

\subsubsection{Two myopic approaches (no future information)}\label{sec:myopic}

In contrast to the two proposed VFA approaches that make use of sampling, future information, and exact solutions of ILPs, we propose two myopic approaches as benchmarks that use neither sampling nor future information. One relies on the solution of an ILP while the other is based on a heuristic approach to build the route to compare with. Myopic approaches are commonly used as benchmarks in papers mentioned in the former sections. Examples of such benchmarks include approaches that do not employ sampling or delay, as shown in \cite{voccia2019same}, as well as the explicitly labeled “myopic approach" discussed in \cite{archetti2020dynamic} and \cite{ulmer2019offline}.  Besides, the two myopic methods we propose mimic real-world cases where companies do not use sophisticated optimization tools. The main idea of both approaches is to dispatch the vehicle immediately in every decision epoch by serving all known requests (if any) that can be feasibly served. The difference between the two lies in how the vehicle route (and eventually the subset of known requests to serve) is determined in each decision epoch. 

The first approach denoted as \textit{ME} in the following, works as follows: every time the vehicle is back to the depot, in case there is at least one parcel available, the vehicle is immediately dispatched. All parcels available are served if the deadline is not violated. In case the deadline is violated, there is a need to select a subset of parcels to deliver. This is done through exactly solving an orienteering problem (via a branch-and-cut method) to maximize the number of parcels delivered while satisfying the limited route duration. 


The second myopic approach, denoted as \textit{MH}, differs from the first in the way the vehicle route is constructed, which follows the nearest neighbor idea. Specifically, every time the vehicle is at the depot, and there is at least one parcel available, the dispatching decision is made by selecting first the closest customer location to the depot, then the nearest customer to the one just selected, and so on until the remaining time is not enough to serve any new customer or all parcels available are included on the route. 

\subsubsection{An approach considering point estimation of future information}
\label{sec:OPrd_PE}
The comparison with the myopic approaches described in the former section aims to show the benefit of incorporating information about future events into the decision policy, as made in VFA-CF and VFA-2S. However, VFA-CF/VFA-2S also incorporate the use of sampling along with the approximation of future routes (through the batch approach). Thus, we propose a further comparison to show the benefit of these two components. Specifically, 
we compare VFA-CF and VFA-2S with an ILP-based approach considering point estimation (PE) of future information and no approximation. 
The approach, to which we refer as OPrd\_PE, is based on the OP-rd formulation (\ref{eq:obj_ILP_base}), which uses expected release dates as point estimations and performs exact evaluations of future rewards. At each decision epoch, we set the release dates of each unknown request to their expected values, and then we solve the OP-rd (\ref{eq:obj_ILP_base}) where we impose that route 0 cannot start before the time of the current decision epoch. We then take the first route in the obtained solution. If all requests served in it are available, we execute the route immediately. Otherwise, we wait until the requests arrive according to the actual release dates, and execute the route. In case the returning time of the route is after the deadline, which may happen because the actual release dates can be later than the corresponding expected value, the route is not executed. Note that this approach embeds the optimal solution of an ILP to determine the action associated with each decision epoch, similar to VFA-2S and VFA-CF. However, it calculates the exact value associated with future routes (contrary to what is done in VFA-2S and VFA-CF, where this value is approximated due to the batch approach), which clearly might require a longer computing time.

\subsubsection{Summary of the Characteristics of the Approaches}
\label{sec:chara_summary}
Table \ref{cha_summary} summarizes the key characteristics of the different approaches compared in this section. The myopic methods MH and ME do not employ sampling nor consider future information. In addition, at each decision epoch, the decision in ME is based on the solution of ILPs, while MH is based on the nearest neighbor heuristic. As for OPrd\_PE, it does not utilize a sampling technique; instead, it relies on point estimation and exact evaluation of future rewards. Notably, the decision in OPrd\_PE is based on the solution of ILPs. In contrast, the two VFA approaches, VFA-CF and VFA-2S, employ sampling and embed future information within ILPs. Besides, they leverage the power of approximation through the batch approach to speed up decision-making. 

By comparing VFA approaches to MH, ME and OPrd\_PE, we can demonstrate the benefits of embedding future information and approximation within ILPs for better decision-making. 
\begin{table}[!htp]
\caption{Summary of major characteristics of the tested approaches
}
\label{cha_summary}
\centering
\scalebox{0.8}{\color{black}
\begin{tabular}{l|ccccc}
\hline&                                               \textbf{MH} & \textbf{ME} & \textbf{OPrd\_PE} & \textbf{VFA-CF} & \textbf{VFA-2S} \\ \hline 
\textbf{Sampling}             & \xmark        & \xmark        & \xmark               & \checkmark        & \checkmark        \\
\textbf{Future Information}   & \xmark        & \xmark        & \checkmark(PE)       & \checkmark        & \checkmark        \\
\textbf{Future Approximation} & \xmark        & \xmark        & \xmark               & \checkmark        & \checkmark        \\
\textbf{ILP solution}         & \xmark        & \checkmark    & \checkmark           & \checkmark        & \checkmark   \\\hline   
\end{tabular}}
\end{table}

\subsection{Computational results}
\label{sec:compute}
In the following subsections, 
we present the performance of the VFA approaches by comparing them with the three benchmark approaches, then we show experiments for two vehicles.
\subsubsection{Performance of VFA-CF and VFA-2S}\label{sec:final_results}

In this section, we compare the results of VFA-CF and VFA-2S (with PC) against ME, MH, and OPrd\_PE. The goal is to show the benefit of incorporating the approximation of future routes and future information into exact solution methods to improve sequential decision making. 
The Key Performance Indicator (KPI) used for the comparison is the gap from the best policy found over all approaches. In addition, we also report, for each of the five approaches, the gap with respect to the upper bound with perfect information obtained by solving Formulation \eqref{eq:base}. This latter measure is related to the competitive analysis where the focus is on measuring the discrepancy in policy quality due to lack of information. 

Results are provided in Table \ref{4approaches_betadelta}, condensed by values of $\beta$ and $\delta$,  and graphical representations are presented in Figure 1 (displaying values over $\beta$) and Figure 2 in Online Appendix I. Detailed results for all values of $\beta$, $\delta$ and $c$ can be found in Online Appendix M. 
In addition, summarized results over the value of $c$ can be found in Online Appendix J, while Appendix K shows an example of solutions obtained through the different approaches. The running times reported below include all the preprocessing steps, such as reading input data, implementing the MDP structure, and the runs of Algorithm \ref{simu:Pro}. As running Algorithm \ref{simu:Pro} usually takes less than one second, we do not report it separately.

For each approach, we report the average gap with respect to the best policy found over all approaches (denoted as $gap\%$), the gap with respect to the upper bound with perfect information ($gap_{ub}\%$), the computational time in seconds ($t[s]$) and the number of times the best policy was found, including ties ($freq$). As for the $gap\%$, we calculated it as follows. Given an input instance, let $N_{p}$ denote the value of the policy found by approach $p$. The gap with respect to the best policy found over all approaches is calculated as $ 1 - \frac{N_{p}}{\max_{m \in \{VFA-CF, VFA-2S, ME, MH, OPrd\_PE\}}  N_{m}}$. Instead, the gap with respect to the upper bound is calculated as  $ 1 - \frac{N_{p}}{UB}$, where $UB$ is the value of the upper bound with perfect information. 

Besides, we summarize the detailed statistics of the 432 instances used over different values of $\beta$ and $\delta$ in Tables \ref{detailed_timeHorizons}-\ref{detailed_waiting}. Table \ref{detailed_timeHorizons} lists the average lengths of the instances' planning horizons in minutes. For each solution approach, Table \ref{detailed_routes} presents the average numbers of routes executed; Table \ref{detailed_knownRequests} shows the average numbers of known orders that can be delivered at each decision epoch over the instances, which reflect the size of the ILP models; Table \ref{detailed_decisionEpochs} shows the average numbers of decision epochs and the average running time per decision epoch over the instances; Table \ref{detailed_routeDuration} and Table \ref{detailed_waiting} display the average vehicle traveling and waiting time over the instances in minutes, where the waiting time is calculated by subtracting traveling time from the corresponding length of the planning horizon.

\begin{table}[tb!]
\caption{Average performance of the five approaches (50 customers)}
\label{4approaches_betadelta}
\centering
\begin{subtable}{\linewidth}\centering
\scalebox{0.6}{\color{black}
\begin{tabular}{l|l|l|cccc|cccc}
\toprule
\multicolumn{1}{l}{\textbf{}}     & \textbf{}      & \textbf{}                                & \multicolumn{4}{c}{\textbf{VFA-CF}}                               & \multicolumn{4}{c}{\textbf{VFA-2S}}                               \\\hline
\multicolumn{1}{l}{$\beta$} & $\delta$ & \multicolumn{1}{c|}{\textbf{\#instances}} & gap{[}\%{]}    & gap\_ub{[}\%{]} & t{[}s{]}       & freq          & gap{[}\%{]}    & gap\_ub{[}\%{]} & t{[}s{]}       & freq          \\\hline\hline
\multirow{3}{*}{0.5}              & 0              & 48                                       & 10.00          & 24.24           & 41.03          & 25            & 9.45           & 23.72           & 29.93          & 25            \\
                                  & 0.5            & 48                                       & 10.93          & 23.41           & 48.07          & 20            & 12.28          & 24.51           & 42.48          & 14            \\
                                  & 1              & 48                                       & 12.84          & 25.45           & 105.61         & 15            & 12.67          & 25.24           & 64.10          & 13            \\\hline
\multicolumn{1}{l}{\textbf{}}     & \textbf{AVG}   & \textbf{48}                              & \textbf{11.26} & \textbf{24.36}  & \textbf{64.90} & \textbf{20.0} & \textbf{11.47} & \textbf{24.49}  & \textbf{45.50} & \textbf{17.3} \\\hline\hline
\multirow{3}{*}{1}                & 0              & 48                                       & 3.92           & 19.36           & 54.81          & 21            & 3.39           & 18.85           & 34.64          & 26            \\
                                  & 0.5            & 48                                       & 4.71           & 21.95           & 48.13          & 24            & 3.98           & 21.33           & 33.02          & 22            \\
                                  & 1              & 48                                       & 3.58           & 15.60           & 39.69          & 20            & 4.28           & 16.11           & 31.40          & 17            \\\hline
\multicolumn{1}{l}{\textbf{}}     & \textbf{AVG}   & \textbf{48}                              & \textbf{4.07}  & \textbf{18.97}  & \textbf{47.54} & \textbf{21.7} & \textbf{3.88}  & \textbf{18.76}  & \textbf{33.02} & \textbf{21.7} \\\hline\hline
\multirow{3}{*}{1.5}              & 0              & 48                                       & 1.07           & 7.75            & 35.26          & 36            & 0.64           & 7.40            & 31.68          & 40            \\
                                  & 0.5            & 48                                       & 3.09           & 13.83           & 39.82          & 28            & 2.28           & 13.25           & 33.11          & 26            \\
                                  & 1              & 48                                       & 2.18           & 10.73           & 36.04          & 29            & 2.76           & 11.15           & 33.04          & 26            \\\hline
\textbf{}                         & \textbf{AVG}   & \textbf{48}                              & \textbf{2.11}  & \textbf{10.77}  & \textbf{37.04} & \textbf{31.0} & \textbf{1.89}  & \textbf{10.60}  & \textbf{32.61} & \textbf{30.7} \\\hline\hline
\multicolumn{2}{c}{\textbf{AVG}}                   & \textbf{48}                              & \textbf{5.81}  & \textbf{18.03}  & \textbf{49.83} & \textbf{24.2} & \textbf{5.75}  & \textbf{17.95}  & \textbf{37.04} & \textbf{23.2}
\\\hline
\end{tabular}
}
\end{subtable}
\newline
\vspace*{0.01\linewidth}
\newline
\begin{subtable}{\linewidth}
\centering
\scalebox{0.6}{\color{black}
\begin{tabular}{l|l|l|cccc|cccc|cccc}
\toprule
\multicolumn{1}{l}{\textbf{}}     & \textbf{}      & \textbf{}                                & \multicolumn{4}{c}{\textbf{MH}}                                  & \multicolumn{4}{c}{\textbf{ME}}                                   & \multicolumn{4}{c}{\textbf{OPrd\_PE}}                             \\\hline
\multicolumn{1}{l}{$\beta$} & $\delta$ & \multicolumn{1}{c|}{\textbf{\#instances}} & gap{[}\%{]}    & gap\_ub{[}\%{]} & t{[}s{]}      & freq          & gap{[}\%{]}    & gap\_ub{[}\%{]} & t{[}s{]}       & freq          & gap{[}\%{]}    & gap\_ub{[}\%{]} & t{[}s{]}        & freq          \\\hline\hline
\multirow{3}{*}{0.5}              & 0              & 48                                       & 36.20          & 45.97           & 0.19          & 0             & 23.78          & 34.88           & 82.27          & 6             & 15.38          & 26.73           & 154.42          & 30            \\
                                  & 0.5            & 48                                       & 30.21          & 38.81           & 0.21          & 2             & 20.60          & 30.45           & 30.81          & 10            & 16.67          & 26.41           & 171.74          & 33            \\
                                  & 1              & 48                                       & 34.79          & 43.24           & 0.20          & 0             & 24.58          & 34.09           & 67.25          & 0             & 4.31           & 16.44           & 199.98          & 34            \\\hline
\multicolumn{1}{l}{\textbf{}}     & \textbf{AVG}   & \textbf{48}                              & \textbf{33.73} & \textbf{42.67}  & \textbf{0.20} & \textbf{0.7}  & \textbf{22.99} & \textbf{33.14}  & \textbf{60.11} & \textbf{5.3}  & \textbf{12.12} & \textbf{23.19}  & \textbf{175.38} & \textbf{32.3} \\\hline\hline
\multirow{3}{*}{1}                & 0              & 48                                       & 17.98          & 30.10           & 0.22          & 0             & 9.20           & 23.11           & 119.77         & 13            & 11.68          & 25.88           & 623.72          & 21            \\
                                  & 0.5            & 48                                       & 19.22          & 32.65           & 0.23          & 3             & 11.45          & 26.90           & 86.28          & 7             & 19.82          & 32.15           & 508.52          & 24            \\
                                  & 1              & 48                                       & 15.67          & 25.11           & 0.24          & 6             & 8.30           & 19.26           & 73.11          & 19            & 4.36           & 15.81           & 557.29          & 33            \\\hline
\multicolumn{1}{l}{\textbf{}}     & \textbf{AVG}   & \textbf{48}                              & \textbf{17.62} & \textbf{29.29}  & \textbf{0.23} & \textbf{3.0}  & \textbf{9.65}  & \textbf{23.09}  & \textbf{93.05} & \textbf{13.0} & \textbf{11.95} & \textbf{24.61}  & \textbf{563.18} & \textbf{26.0} \\\hline\hline
\multirow{3}{*}{1.5}              & 0              & 48                                       & 7.27           & 12.79           & 0.21          & 18            & 2.20           & 8.60            & 38.54          & 28            & 13.74          & 17.44           & 788.30          & 32            \\
                                  & 0.5            & 48                                       & 8.32           & 18.28           & 0.24          & 12            & 4.07           & 14.71           & 8.46           & 29            & 10.16          & 19.18           & 636.08          & 31            \\
                                  & 1              & 48                                       & 7.12           & 14.72           & 0.23          & 16            & 3.09           & 11.42           & 38.71          & 29            & 3.62           & 11.60           & 698.63          & 35            \\\hline
\textbf{}                         & \textbf{AVG}   & \textbf{48}                              & \textbf{7.57}  & \textbf{15.26}  & \textbf{0.23} & \textbf{15.3} & \textbf{3.12}  & \textbf{11.58}  & \textbf{28.57} & \textbf{28.7} & \textbf{9.17}  & \textbf{16.07}  & \textbf{707.67} & \textbf{32.7} \\\hline\hline
\multicolumn{2}{c}{\textbf{AVG}}                   & \textbf{48}                              & \textbf{19.64} & \textbf{29.07}  & \textbf{0.22} & \textbf{6.3}  & \textbf{11.92} & \textbf{22.60}  & \textbf{60.58} & \textbf{15.7} & \textbf{11.08} & \textbf{21.29}  & \textbf{482.08} & \textbf{30.3}
\\\bottomrule
\end{tabular}
}
\end{subtable}
\end{table}

\begin{table}[tb!]
\begin{minipage}{0.25\linewidth}
    \caption{Length of time horizons.}
    \label{detailed_timeHorizons}
    \centering
    \scalebox{0.56}{\color{black}
    \begin{tabular}{l|l||c}
    \hline
    $\beta$                       & $\delta$                      & \textbf{timeHorizon} \\ \hline\hline
    0.5                           & 0                             & 277.94                                \\
                                  & 0.5                           & 276.63                                \\
                                  & 1                             & 308.04                                \\
                                  & \textbf{AVG} & \textbf{287.53}                                \\ \hline
    1                             & 0                             & 550.63                                \\
                                  & 0.5                           & 535.08                                \\
                                  & 1                             & 614.94                                \\
                                  & \textbf{AVG} & \textbf{566.88}                                \\ \hline
    1.5                           & 0                             & 865.63                                \\
                                  & 0.5                           & 771.88                                \\
                                  & 1                             & 805.90                                \\
                                  & \textbf{AVG} & \textbf{814.47}                                \\ \hline
    \textbf{AVG} &                               & \textbf{556.29}      \\ \hline
    \end{tabular}}
\end{minipage}\hfill
\begin{minipage}{.35\linewidth}
    \caption{Number of routes.}
    \label{detailed_routes}
    \centering
    \scalebox{0.56}{\color{black}
    \begin{tabular}{l|l||c|c|c|c|c}
    \hline
    $\beta$                       & $\delta$                      & \textbf{VFA-CF} & \textbf{VFA-2S} & \textbf{MH}   & \textbf{ME}   & \textbf{OPrd\_PE} \\ \hline\hline
    0.5                           & 0                             & 2.38                             & 2.38                             & 2.38                           & 2.31                           & 2.27                               \\
                                  & 0.5                           & 2.44                             & 2.42                             & 2.35                           & 2.38                           & 2.54                               \\
                                  & 1                             & 2.58                             & 2.58                             & 2.44                           & 2.42                           & 2.85                               \\
                                  & \textbf{AVG} & \textbf{2.47}                             & \textbf{2.46}                             & \textbf{2.39}                           & \textbf{2.37}                           & \textbf{2.56}                               \\ \hline
    1                             & 0                             & 4.44                             & 4.38                             & 3.38                           & 3.38                           & 6.06                               \\
                                  & 0.5                           & 4.13                             & 4.04                             & 3.21                           & 3.40                           & 5.19                               \\
                                  & 1                             & 4.23                             & 4.29                             & 3.69                           & 3.98                           & 6.46                               \\
                                  & \textbf{AVG} & \textbf{4.26}                             & \textbf{4.24}                             & \textbf{3.42}                           & \textbf{3.58}                           & \textbf{5.90}                               \\ \hline
    1.5                           & 0                             & 5.54                             & 5.60                             & 4.81                           & 5.19                           & 9.56                               \\
                                  & 0.5                           & 5.48                             & 5.50                             & 4.56                           & 4.96                           & 7.58                               \\
                                  & 1                             & 5.56                             & 5.65                             & 4.83                           & 5.23                           & 8.48                               \\
                                  & \textbf{AVG} & \textbf{5.53}                             & \textbf{5.58}                             & \textbf{4.74}                           & \textbf{5.13}                           & \textbf{8.54}                               \\ \hline
    \textbf{AVG} &                               & \textbf{4.09}   & \textbf{4.09}   & \textbf{3.52} & \textbf{3.69} & \textbf{5.67}     \\ \hline
    \end{tabular}}
\end{minipage} \hfill
\begin{minipage}{.35\linewidth}
\caption{Number of known requests at each decision epoch.}
\label{detailed_knownRequests}
\centering
\scalebox{0.56}{\color{black}
\begin{tabular}{l|l||c|c|c|c|c}
\toprule
\textbf{$\beta$} & \textbf{$\delta$} & \textbf{VFA-CF} & \textbf{VFA-2S} & \textbf{MH}   & \textbf{ME}   & \textbf{OPrd\_PE} \\\hline\hline
0.5           & 0              & 15.27           & 15.11           & 17.02          & 17.05          & 19.88              \\
              & 0.5            & 14.99           & 15.08           & 16.06          & 16.88          & 19.31              \\
              & 1              & 14.31           & 14.43           & 16.35          & 15.41          & 18.86              \\
              & \textbf{AVG}            & \textbf{14.86}           & \textbf{14.87}           & \textbf{16.48}          & \textbf{16.45}          & \textbf{19.35}              \\\hline
1             & 0              & 13.10           & 13.19           & 13.44          & 12.50          & 16.62              \\
              & 0.5            & 12.49           & 12.58           & 13.02          & 12.70          & 15.54              \\
              & 1              & 10.91           & 11.49           & 12.63          & 11.66          & 16.55              \\
              & \textbf{AVG}            & \textbf{12.17}           & \textbf{12.42}           & \textbf{13.03}          & \textbf{12.29}          & \textbf{16.24}              \\\hline
1.5           & 0              & 8.13            & 9.71            & 9.05           & 8.52           & 14.79              \\
              & 0.5            & 9.53            &10.16            & 9.48           & 8.79           & 14.77              \\
              & 1              & 8.23            & 8.91            & 9.29           & 8.86           & 13.37              \\
              & \textbf{AVG}            & \textbf{8.63}            & \textbf{9.59}            & \textbf{9.27}           & \textbf{8.72}           & \textbf{14.31}              \\\hline
\textbf{AVG}  & \textbf{}      & \textbf{11.88}  & \textbf{12.30}  & \textbf{12.93} & \textbf{12.48} & \textbf{16.63}    
\\\hline
\end{tabular}
}\end{minipage}
\end{table}

\begin{table}[h]
\caption{Number of decision epochs and running time per decision epoch.}
\label{detailed_decisionEpochs}
\centering
\scalebox{0.56}{\color{black}
\begin{tabular}{l|l||c|c|c|c|c|c|c|c|c|c}
\toprule
\multicolumn{1}{l}{\textbf{}}     & \textbf{}      & \multicolumn{2}{c}{\textbf{VFA-CF}} & \multicolumn{2}{|c}{\textbf{VFA-2S}} & \multicolumn{2}{|c}{\textbf{MH}} & \multicolumn{2}{|c}{\textbf{ME}} & \multicolumn{2}{|c}{\textbf{OPrd\_PE}} \\\hline\hline
\multicolumn{1}{l}{\textbf{$\beta$}} & \textbf{$\delta$} & \#epochs         & t[s] per epoch   & \#epochs         & t[s] per epoch   & \#epochs       & t[s] per epoch & \#epochs       & t[s] per epoch & \#epochs          & t[s] per epoch    \\\hline
\multirow{3}{*}{0.5}              & 0              & 6.19             & 6.63             & 5.94             & 5.04             & 3.19           & 0.06           & 3.13           & 26.33          & 3.02              & 51.12             \\
                                  & 0.5            & 6.56             & 7.32             & 6.67             & 6.37             & 3.00           & 0.07           & 3.21           & 9.60           & 3.33              & 51.52             \\
                                  & 1              & 7.46             & 14.16            & 7.44             & 8.62             & 3.31           & 0.06           & 3.15           & 21.38          & 3.42              & 58.53             \\
\multicolumn{1}{l|}{\textbf{}}     & \textbf{AVG}   & \textbf{6.74}    & \textbf{9.37}    & \textbf{6.68}    & \textbf{6.68}    & \textbf{3.17}  & \textbf{0.06}  & \textbf{3.16}  & \textbf{19.10} & \textbf{3.26}     & \textbf{53.72}    \\\hline
\multirow{3}{*}{1}                & 0              & 6.85             & 8.00             & 6.83             & 5.07             & 4.00           & 0.05           & 4.25           & 28.18          & 6.83              & 91.28             \\
                                  & 0.5            & 8.10             & 5.94             & 7.79             & 4.24             & 3.98           & 0.06           & 4.25           & 20.30          & 6.00              & 84.75             \\
                                  & 1              & 7.63             & 5.21             & 7.04             & 4.46             & 4.54           & 0.05           & 4.85           & 15.06          & 7.06              & 78.91             \\
\multicolumn{1}{l|}{\textbf{}}     & \textbf{AVG}   & \textbf{7.53}    & \textbf{6.38}    & \textbf{7.22}    & \textbf{4.59}    & \textbf{4.17}  & \textbf{0.05}  & \textbf{4.45}  & \textbf{21.18} & \textbf{6.63}     & \textbf{84.98}    \\\hline
\multirow{3}{*}{1.5}              & 0              & 8.98             & 3.93             & 6.46             & 4.90             & 5.94           & 0.03           & 6.25           & 6.17           & 10.10             & 78.02             \\
                                  & 0.5            & 9.31             & 4.28             & 7.73             & 4.28             & 5.48           & 0.04           & 5.96           & 1.42           & 8.13              & 78.29             \\
                                  & 1              & 8.56             & 4.21             & 7.48             & 4.42             & 5.58           & 0.04           & 6.13           & 6.32           & 9.13              & 76.56             \\
\textbf{}                         & \textbf{AVG}   & \textbf{8.95}    & \textbf{4.14}    & \textbf{7.22}    & \textbf{4.54}    & \textbf{5.67}  & \textbf{0.04}  & \textbf{6.11}  & \textbf{4.63}  & \textbf{9.12}     & \textbf{77.62}    \\\hline
\multicolumn{2}{l||}{\textbf{AVG}}                   & \textbf{7.74}    & \textbf{6.63}    & \textbf{7.04}    & \textbf{5.27}    & \textbf{4.34}  & \textbf{0.05}  & \textbf{4.57}  & \textbf{14.97} & \textbf{6.34}     & \textbf{72.11}   \\\bottomrule
\end{tabular}
}\end{table}

\begin{table}[h]
\begin{minipage}{0.5\linewidth}
\caption{Traveling time.}
\label{detailed_routeDuration}
\centering
\scalebox{0.56}{\color{black}
\begin{tabular}{l|l||c|c|c|c|c}
\hline
\textbf{$\beta$} & \textbf{$\delta$} & \textbf{VFA-CF} & \textbf{VFA-2S} & \textbf{MH}     & \textbf{ME}     & \textbf{OPrd\_PE} \\ \hline\hline
0.5                               & 0                                  & 263.27                           & 264.00                           & 272.06                           & 271.56                           & 209.65                              \\
                                  & 0.5                                & 262.73                           & 261.15                           & 273.90                           & 271.17                           & 201.06                              \\
                                  & 1                                  & 290.46                           & 290.69                           & 300.98                           & 303.83                           & 256.08                              \\
                                  & \textbf{AVG}      & \textbf{272.15}                           & \textbf{271.94}                           & \textbf{282.31}                           & \textbf{282.19}                           & \textbf{222.26}                              \\ \hline
1                                 & 0                                  & 529.44                           & 529.15                           & 546.13                           & 531.50                           & 451.10                              \\
                                  & 0.5                                & 507.04                           & 507.79                           & 525.81                           & 523.21                           & 415.58                              \\
                                  & 1                                  & 576.29                           & 576.50                           & 593.54                           & 586.85                           & 535.19                              \\
                                  & \textbf{AVG}      & \textbf{537.59}                           & \textbf{537.81}                           & \textbf{555.16}                           & \textbf{547.19}                           & \textbf{467.29}                              \\ \hline
1.5                               & 0                                  & 789.79                           & 791.33                           & 799.94                           & 797.44                           & 756.19                              \\
                                  & 0.5                                & 717.54                           & 716.46                           & 738.00                           & 725.27                           & 640.31                              \\
                                  & 1                                  & 755.19                           & 754.42                           & 778.21                           & 762.83                           & 719.31                              \\
                                  & \textbf{AVG}      & \textbf{754.17}                           & \textbf{754.07}                           & \textbf{772.05}                           & \textbf{761.85}                           & \textbf{705.27}                              \\ \hline
\textbf{AVG}     &        & \textbf{521.31} & \textbf{521.28} & \textbf{536.51} & \textbf{530.41} & \textbf{464.94}    \\ \hline
\end{tabular}
}
\end{minipage}
\begin{minipage}{0.5\linewidth}
\caption{Waiting time.}
\label{detailed_waiting}
\centering
\scalebox{0.56}{\color{black}
\begin{tabular}{l|l||c|c|c|c|c}
\hline
\textbf{$\beta$} & \textbf{$\delta$} & \textbf{VFA-CF} & \textbf{VFA-2S} & \textbf{MH}    & \textbf{ME}    & \textbf{OPrd\_PE} \\ \hline
0.5                               & 0                                  & 14.67                            & 13.94                            & 5.88                            & 6.38                            & 68.29                               \\
                                  & 0.5    & 13.90      & 15.48  & 2.73 & 5.46 & 75.56 \\
                                  & 1 & 17.58 & 17.35 & 7.06& 4.21& 51.96 \\
                                  & \textbf{AVG}      & \textbf{15.38}& \textbf{15.59} & \textbf{5.22}& \textbf{5.35}& \textbf{65.27}  \\ \hline
1                                 & 0                                  & 21.19                            & 21.48                            & 4.50                            & 19.13                           & 99.52                               \\
                                  & 0.5  & 28.04 & 27.29& 9.27& 11.88& 119.50\\
                                  & 1  & 38.65 & 38.44& 21.40& 28.08& 79.75\\
                                  & \textbf{AVG}      & \textbf{29.29} & \textbf{29.07}& \textbf{11.72}& \textbf{19.69}& \textbf{99.59}\\ \hline
1.5                               & 0                                  & 75.83                            & 74.29                            & 65.69                           & 68.19                           & 109.44                              \\
                                  & 0.5& 54.33  & 55.42& 33.88& 46.60 & 131.56\\
                                  & 1 & 50.71& 51.48& 27.69& 43.06& 86.58\\
                                  & \textbf{AVG}      & \textbf{60.29}& \textbf{60.40}& 42.42 & \textbf{52.62}& \textbf{109.19}\\ \hline
\textbf{AVG}     & \textbf{}         & \textbf{34.99}  & \textbf{35.02}  & \textbf{19.79} & \textbf{25.89} & \textbf{91.35}     \\ \hline
\end{tabular}
}\end{minipage}\end{table}


\textbf{Comparison between VFA-CF and VFA-2S} Focusing first on the KPIs related to the best-known policy (gap[\%] and freq) and on the comparison between VFA-2S and VFA-CF in Table \ref{4approaches_betadelta}, we notice that, on average, the results obtained with VFA-2S are associated with the smallest gaps, being anyway very similar to the ones obtained with VFA-CF. When considering the number of best policies found, we see VFA-CF behaves better than VFA-2S. However,  differences are negligible. We conduct a t-test on the objective values obtained by the two approaches and get a p-value of 0.98, which is significantly greater than 0.05. Thus there is not enough evidence to suggest a significant difference in objective values achieved by the two algorithms, or the chances of the two having different results are low. We instead observe a more remarkable difference in terms of computing time, with VFA-CF being slower than VFA-2S. Specifically, 
the critical case is the one with  $\beta = 0.5$ and $\delta = 1$, where the computing time of VFA-CF  is almost two times the one of VFA-2S. This might stem from the fact that VFA-CF Formulation \eqref{eq:deterministic} is solved for each scenario while VFA-2S involves solving the deterministic equivalent problem (DEP). Solving DEP can be computationally faster than solving multiple scenario problems on a low number of scenarios and considering I/O overheads. 
Moving to the competitive analysis, we see that VFA-CF and VFA-2S produce results between 7\% and 25\% from the upper bound with perfect information.

\textbf{Comparison with the benchmark approaches} Focusing now on myopic approaches, they are largely outperformed in terms of policy quality by VFA approaches. In terms of computing time, while the one of MH is negligible, for ME, it is comparable to the one of VFA-2S and VFA-CF (despite achieving much worse policies). From Tables \ref{detailed_routes} and \ref{detailed_routeDuration}, we find that MH and ME execute a smaller number of routes, but these routes are longer in duration. As they end up serving fewer requests, we can infer that they potentially overlook the benefits of incorporating future requests to optimize and consolidate deliveries. This observation underscores the importance of considering future information when making sequential decisions. 

Regarding OPrd\_PE, while it often produces best policies, its performance is inconsistent across instances, and the average running time is significantly longer compared to other approaches. For the effect on $\delta$, we see that when $\delta$ is 0, VFA approaches have the smallest gaps. When $\delta$ is 1, OPrd\_PE improves significantly, but this is not true for others. This can be attributed to the decreasing uncertainties as the vehicle approaches the depot, leading to more precise estimations of release dates (see Online Appendix F for details). When $\delta$ equals 1, indicating dynamic updates of all customers' release dates, the uncertainties associated with the release dates of unknown parcels decrease as they approach the depot, so point estimations of the release dates of parcels become more precise. OPrd\_PE builds the first route by taking some known and soon-to-arrive unknown requests. Proximity of estimations to actual release dates ensures more on-time scheduling, resulting in a high number of served requests. Conversely, without dynamic updates of release dates, the estimations are less precise, so the routes proposed by OPrd\_PE have a higher chance of being delayed. Notably, OPrd\_PE's performance is not stable across settings. In contrast, VFA-CF and VFA-2S exhibit more robust performance as they generally outperform all others in gaps, time, and frequency of achieving the best. From Tables \ref{detailed_routes} and \ref{detailed_routeDuration}, we observe that OPrd\_PE executes more but shorter routes than others, which may not be cost-effective. Additionally, OPrd\_PE has a long runtime (8 minutes on average), which is attributed to its higher model complexity that includes both known and unknown requests. This shows the benefit of the batch approach proposed in this paper: it reduces ILP complexity without affecting policy quality.

In addition, we observe that, in most approaches, gap$_{ub}$[\%] decreases when $\beta$ increases. This might be due to the fact that, when the release dates are more dispersed, the best policy is to dispatch the vehicle immediately with a subset of known requests, thus avoiding waiting too long at the depot.
Instead, when the release dates are less dispersed, determining whether to wait or to start dispatching some parcels becomes more challenging. This is also witnessed by the improved performance, for larger $\beta$ values, of myopic approaches,  which are based on the idea of dispatching the vehicle immediately with the largest possible subset of known requests. 

To examine the necessity of considering both stochastic and dynamic release dates, we benchmark our VFA approaches against a purely stochastic setting in which no dynamic updates of release dates are used. We therefore compare VFA approaches with an alternative approach, called  OPrd\_Sto, that solves the problem at time 0 (using the point estimates sampled from the stochastic information) and does not update the decisions afterward. The results from these experiments reveal a significant performance gap:  
 OPrd\_Sto takes 30 times more runtime compared to our VFA approaches and serves 28\% fewer requests on average. Further details can be found in Online Appendix L.

\textbf{Scalability test} We also conduct additional scalability tests on 36 instances containing 100 customers and present the results in Table \ref{4approaches_100}. Focusing first on VFA approaches, we observe that their gaps are similar. While the running time increases for both approaches, the one of VFA-CF is much longer compared to VFA-2S. In addition to the explanation provided for the test involving 50 customers, we identify two potential reasons contributing to the extended running time of VFA-CF: (1) As described in Section \ref{consensus}, whenever the trip duration for serving the selected customers from the consensus function exceeds the remaining time, an orienteering problem model is solved to determine the final requests to be served. In contrast, VFA-2S does not involve such a step. (2) The time limit for running VFA-CF Formulation \eqref{eq:deterministic} is set at 5 minutes, and we run that formulation 30 times (as there are 30 scenarios) at each decision epoch. Conversely, VFA-2S has a time limit of 10 minutes and is solved only once.

Next, in comparison with benchmark approaches, VFA approaches consistently exhibit minimum gaps, approximately 2\%, which is much smaller than those of the other approaches (over 30\%). While the myopic approaches are still fast, especially the simple MH approach, they achieve much worse policies compared to VFA approaches in most cases. Scaling from 50 to 100 customers, the performance of OPrd\_PE greatly deteriorates and the running time increases significantly. This outcome is anticipated, given the explosion in model size and search space, making it a challenge to find an optimal solution within the constrained running time.

\begin{table}[tb!]
\caption{Average performance of the five approaches (100 customers)}
\label{4approaches_100}
\centering
\scalebox{0.6}{\color{black}
\begin{tabular}{l|l|l|ccc|ccc|ccc|ccc|ccc}
\toprule
\multicolumn{1}{l}{\textbf{}}     & \textbf{}      & \textbf{}                                & \multicolumn{3}{c|}{\textbf{VFA-CF}}            & \multicolumn{3}{c|}{\textbf{VFA-2S}}            & \multicolumn{3}{c|}{\textbf{MH}}               & \multicolumn{3}{c|}{\textbf{ME}}                 & \multicolumn{3}{c}{\textbf{OPrd\_PE}}            \\\hline
\multicolumn{1}{l|}{\textbf{$\beta$}} & \textbf{$\delta$} & \multicolumn{1}{c|}{\textbf{\#instances}} & gap{[}\%{]}   & t{[}s{]}        & freq         & gap{[}\%{]}   & t{[}s{]}        & freq         & gap{[}\%{]}    & t{[}s{]}      & freq         & gap{[}\%{]}    & t{[}s{]}        & freq         & gap{[}\%{]}    & t{[}s{]}          & freq         \\\hline\hline
\multirow{3}{*}{0.5}              & 0              & 4                                        & 1.29          & 117.77          & 3            & 1.29          & 52.17           & 3            & 32.81          & 0.35          & 0            & 23.03          & 170.86          & 0            & 33.30          & 1552.86           & 1            \\
                                  & 0.5            & 4                                        & 1.23          & 515.22          & 3            & 1.23          & 63.55           & 3            & 41.97          & 0.34          & 0            & 26.15          & 176.88          & 0            & 23.40          & 1394.72           & 1            \\
                                  & 1              & 4                                        & 10.16         & 649.99          & 3            & 10.16         & 67.29           & 3            & 38.96          & 0.33          & 0            & 31.17          & 66.39           & 0            & 14.38          & 6243.51           & 1            \\\hline
\multicolumn{1}{l}{\textbf{}}     & \textbf{AVG}   & \textbf{4}                               & \textbf{4.23} & \textbf{427.66} & \textbf{3.0} & \textbf{4.23} & \textbf{61.00}  & \textbf{3.0} & \textbf{37.91} & \textbf{0.34} & \textbf{0.0} & \textbf{26.78} & \textbf{138.04} & \textbf{0.0} & \textbf{23.69} & \textbf{3063.70}  & \textbf{1.0} \\\hline\hline
\multirow{3}{*}{1}                & 0              & 4                                        & 3.56          & 203.42          & 2            & 3.60          & 63.50           & 2            & 31.52          & 0.33          & 0            & 27.57          & 151.25          & 1            & 64.35          & 12160.15          & 0            \\
                                  & 0.5            & 4                                        & 3.07          & 327.95          & 3            & 0.00          & 62.23           & 4            & 37.55          & 0.34          & 0            & 37.47          & 0.73            & 0            & 26.10          & 14439.80          & 0            \\
                                  & 1              & 4                                        & 0.00          & 1649.04         & 4            & 0.00          & 209.57          & 4            & 41.34          & 0.36          & 0            & 40.56          & 0.75            & 0            & 31.18          & 14870.07          & 0            \\\hline
\multicolumn{1}{l}{\textbf{}}     & \textbf{AVG}   & \textbf{4}                               & \textbf{2.21} & \textbf{726.80} & \textbf{3.0} & \textbf{1.20} & \textbf{111.76} & \textbf{3.3} & \textbf{36.81} & \textbf{0.34} & \textbf{0.0} & \textbf{35.20} & \textbf{50.91}  & \textbf{0.3} & \textbf{40.54} & \textbf{13823.34} & \textbf{0.0} \\\hline\hline
\multirow{3}{*}{1.5}              & 0              & 4                                        & 0.00          & 336.78          & 4            & 1.75          & 167.87          & 2            & 43.42          & 0.35          & 0            & 42.69          & 0.78            & 0            & 27.47          & 16658.04          & 0            \\
                                  & 0.5            & 4                                        & 0.45          & 99.89           & 3            & 0.00          & 47.46           & 4            & 33.73          & 0.37          & 0            & 29.42          & 0.80            & 0            & 46.07          & 9948.25           & 0            \\
                                  & 1              & 4                                        & 0.00          & 102.53          & 4            & 0.00          & 55.34           & 4            & 42.68          & 0.37          & 0            & 42.32          & 0.69            & 0            & 51.63          & 14281.81          & 0            \\\hline
\textbf{}                         & \textbf{AVG}   & \textbf{4}                               & \textbf{0.15} & \textbf{179.73} & \textbf{3.7} & \textbf{0.58} & \textbf{90.22}  & \textbf{3.3} & \textbf{39.95} & \textbf{0.36} & \textbf{0.0} & \textbf{38.14} & \textbf{0.75}   & \textbf{0.0} & \textbf{41.72} & \textbf{13629.37} & \textbf{0.0} \\\hline\hline
\multicolumn{2}{c}{\textbf{AVG}}                   & \textbf{4}                               & \textbf{2.20} & \textbf{444.73} & \textbf{3.2} & \textbf{2.00} & \textbf{87.66}  & \textbf{3.2} & \textbf{38.22} & \textbf{0.35} & \textbf{0.0} & \textbf{33.38} & \textbf{63.24}  & \textbf{0.1} & \textbf{35.32} & \textbf{10172.14} & \textbf{0.3}
\\\hline
\end{tabular}
}
\end{table}

\textbf{Summary} Overall, both VFA-2S and VFA-CF largely outperform the other approaches, among them, the myopic ones being considered as proxies of practical policies. The difference in policy quality between VFA-2S and VFA-CF is marginal. However, as VFA-2S is faster, it is a good compromise between policy quality and computational burden.
To summarize, we showed the benefits of:
\begin{itemize}
    \item approximating future routes through a batch approach that helped largely reduce the computing time (with respect to OPrd\_PE) without affecting policy quality;
    \item taking into account future information in the decision process by comparing the VFA-CF/VFA-2S with two myopic approaches, which provide vastly worse results.
\end{itemize}

\subsubsection{Multiple vehicles}
\label{2-vehicle}
In this section, we present experiments where two vehicles are available for multiple dispatches in an area during the day. These tests are aimed at showing how VFA-CF  and VFA-2S can be adapted to this setting in terms of scalability and policy quality.
The approaches are modified as follows: at each decision epoch, each vehicle is treated individually and independently of the status of the other vehicle, i.e., without considering the existence of the other vehicle (e.g., without leaving some parcels for the other to deliver), according to the approaches described above. We note that the two vehicles still share the same depot and the same set of customer demands. 

This two-vehicle approach offers a first attempt at tackling the multi-vehicle case to investigate the viability of our models, providing a feasible approach. The MDP model (definitions of decision epochs, states, decisions, etc) for the multi-vehicle case may need to be modified. However, we believe that similar building blocks based on the batch approach and ILP models can be adapted for the multiple-vehicle case.

In Table \ref{tab-2vehicle}, we show the average running time with one vehicle and two vehicles under the column named t[s], the average number of requests served with two vehicles, the average and maximum improvement of the policy with two vehicles with respect to the policy with one vehicle only, and the number of times the policy with two vehicles improved over the one with one vehicle only. Compared to the single-vehicle version, the two-vehicle version improves over 106 instances out of 130 on average, serving more than $31\%$ customers over the improved instances.
Notably, the number of served requests doubled for certain instances, resulting in $100\%$ improvement. It is worth mentioning that the simulation running time increases in the two-vehicle version, and it nearly doubles for the approach with consensus function (but still very fast), compared to the single-vehicle version. This increase in running time is expected as the simulation is executed more frequently due to the availability of two vehicles. In addition, we find that the average number of requests served by VFA-CF and VFA-2S are similar, with a slight advantage for VFA-2S.
\begin{table}[hpt]
\caption{Performance of Two Vehicles Compared to Single Vehicle (50 customers)}
\label{tab-2vehicle}
\centering
\scalebox{0.59}{\color{black}
\begin{tabular}{l|c||cc|c|ccc|cc|c|ccc}
\toprule
              &  & \multicolumn{6}{c|}{\textbf{VFA-CF}} & \multicolumn{6}{c}{\textbf{VFA-2S}}\\\hline
              &                & \multicolumn{2}{c|}{t[s]}        & \multicolumn{1}{c}{\multirow{2}{*}{\#served}} & \multicolumn{3}{|c|}{Improvement}                   & \multicolumn{2}{c|}{t[s]}        & \multicolumn{1}{c}{\multirow{2}{*}{\#served}} & \multicolumn{3}{|c}{Improvement}                   \\
              & \#instances    & 1Vehicle       & 2Vehicles      & \multicolumn{1}{c}{}                          & \multicolumn{1}{|c}{AVG(\%)}         & MAX(\%)         & \#impr. & 1Vehicle       & 2Vehicles      & \multicolumn{1}{c}{}                          & \multicolumn{1}{|c}{AVG(\%)}         & MAX(\%)         & \#impr. \\\hline\hline
$\delta$=0       & 144            & 43.70          & 114.46         & 41.11                                         & 35.85          & 88.24           & 111            & 32.08          & 46.48          & 41.19                                         & 35.54          & 88.24           & 111            \\
$\delta$=0.5     & 144            & 45.34          & 61.57          & 39.13                                         & 33.07          & 80.00           & 124            & 36.20          & 46.56          & 39.03                                         & 32.40          & 100.00          & 125            \\
$\delta$=1       & 144            & 60.45          & 64.72          & 40.95                                         & 29.67          & 100.00          & 119            & 42.85          & 49.46          & 41.06                                         & 30.61          & 100.00          & 119            \\\hline
$\beta$=0.5      & 144            & 64.90          & 113.08         & 31.96                                         & 49.50          & 100.00          & 143            & 45.50          & 43.10          & 31.84                                         & 49.60          & 100.00          & 143            \\
$\beta$=1        & 144            & 47.54          & 60.56          & 43.32                                         & 24.53          & 75.00           & 128            & 33.02          & 45.78          & 43.46                                         & 24.63          & 76.92           & 129            \\
$\beta$=1.5      & 144            & 37.04          & 67.12          & 45.92                                         & 16.78          & 60.00           & 83             & 32.61          & 53.62          & 45.98                                         & 16.48          & 55.00           & 83             \\\hline
$c$=0.6         & 108            & 44.88          & 63.66          & 30.13                                         & 42.98          & 100.00          & 108            & 33.04          & 48.41          & 30.29                                         & 42.60          & 100.00          & 108            \\
$c$=0.8         & 108            & 47.34          & 62.46          & 39.58                                         & 32.75          & 93.75           & 101            & 33.18          & 46.53          & 39.51                                         & 32.84          & 88.24           & 101            \\
$c$=1           & 108            & 57.73          & 64.38          & 44.21                                         & 28.03          & 87.50           & 79             & 40.50          & 48.54          & 44.23                                         & 27.83          & 87.50           & 79             \\
$c$=1.2         & 108            & 49.36          & 130.50         & 47.67                                         & 21.93          & 80.95           & 66             & 41.45          & 46.53          & 47.68                                         & 22.72          & 84.00           & 67             \\\hline
\textbf{AVG/MAX} & \textbf{129.6} & \textbf{49.83} & \textbf{80.25} & \textbf{40.40}                                & \textbf{31.51} & \textbf{100.00} & \textbf{106.2} & \textbf{37.04} & \textbf{47.50} & \textbf{40.43}                                & \textbf{31.52} & \textbf{100.00} & \textbf{106.5}  
\\\hline
\end{tabular}
}
\end{table}

\section{Conclusion}\label{sec:conclusion}

To the best of our knowledge, this article is the first to study the orienteering problem with stochastic and dynamic release dates.   
For this challenging sequential decision-making problem, we propose two RL approaches.  Both rely on MDP with: discrete sets of scenarios simulating future realizations of release dates in each decision epoch, approximation of future routes, and exact solutions based on ILP models. However, they employ the approximation of future routes in two different ways: VFA-CF handles each scenario independently and uses a consensus function to decide on the action in each decision epoch. VFA-2S considers all scenarios simultaneously and uses a two-stage stochastic ILP model instead. They are compared with two myopic alternatives  (no future information used) that we consider as proxies of practical policies, and {an alternative ILP-based approach} considering PE of future information. 
Our empirical study confirms the benefits of incorporating approximation of future routes and future information into exact solution methods to improve sequential decision making. 
Both VFA-CF and VFA-2S significantly outperform myopic approaches, not only in terms of best-obtained gaps with respect to the perfect-information upper bounds but also in terms of the frequency with which the best-quality policies are obtained. Compared to the one that exploits PE of future information, {VFA-CF} and VFA{-2S} are less sensitive to the changes in the dynamics of release dates and provide more stable results along with much lower running time. 
The difference in policy quality between VFA-CF and VFA-2S is marginal. However, as VFA-2S is faster, we believe it is a good compromise between policy quality and computational burden.

This paper shows that exact methods provide benefits in solving complex sequential stochastic and dynamic optimization problems, specifically routing problems, which are known to be extremely difficult even in their deterministic counterpart. Also, accounting for future uncertainty as done in VFA-2S/VFA-CF, i.e., by coupling exact methods (sections \ref{policyFunc}-\ref{look-aheadFunc}) with an approximation of the future routes (Section \ref{sec:Batch}), brings substantial value compared to the approaches that do not use future information (ME/MH), without losing tractability (which is instead lost in the OPrd\_PE). Moreover, the overall methodology does not require extensive training (as many of $Q$-learning approaches do), and relies on a tuning of very few hyper-parameters instead. The current paper provides a solid base for building approaches for other more complicated problem settings, involving, for example, vehicle capacity, time windows, or multiple vehicles. Additionally, the methodology can be potentially better tailored to the situations in which parcels arrive in batches.
For the multi-vehicle case: OP-rd and ME can be adapted by integrating multiple vehicles into the formulations, and MH can be modified to assign requests to idle vehicles at each decision epoch. In Section \ref{2-vehicle}, we show how the VFA approaches are used for two vehicles. An interesting research direction could be to design tailored approaches for the multi-vehicle case. Adapting VFA-CF and VFA-2S, combined with the approximation of future routes based on the batch approach, is more involved as the assignment of batches to vehicles should be included in the batch approach. However, the methodology proposed in this paper represents a solid starting point and given that it proved effective in single-vehicle and two-vehicle cases, it holds promise for the multi-vehicle case. 

\ACKNOWLEDGMENT{%
This work was partially funded by CY Initiative of Excellence (``Investissements d'Avenir" ANR-16-IDEX-0008''). This support is greatly appreciated. The authors are grateful for the reviews of the associate editor and two anonymous referees; their comments helped substantially in improving a former version of the paper.
}
\bibliographystyle{informs2014trsc}
\bibliography{IISE-Trans}

\clearpage
\begin{appendices}
We report overview of notation in Appendix \ref{sec:appendix_notation}, supplement to the literature review in Appendix \ref{sec:liter}, a first trivial upper bound for DOP-rd in Appendix \ref{sec:ub_appendix}, proof of Theorem 1 in Appendix \ref{sec:proof}, optimality conditions for the batch approach in Appendix \ref{optimality_batch}, description of input instances in Appendix \ref{sec:appendix_instances}, benefit of partial characterization of optimal policy in Appendix \ref{sec:partial_benefit}, hyper-parameter tuning in Appendix \ref{sec:training}, graphical representation of performance of the approaches on $\beta$ and $\delta$ in Appendix \ref{sec:4approaches_beta_delta}, performance of the approaches on time scale in Appendix \ref{sec:4approaches_timescale}, illustrative example in Appendix \ref{sec:route_example_plot}, comparison with a deterministic and static approach in Appendix \ref{sec:deter}, and detailed results for all values of $\beta$, $\delta$, and $c$ in Appendix \ref{detail}.

\section{Overview of notation used} \label{sec:appendix_notation}
See Table \ref{table:notation} and Table \ref{table:notation_sol}.
\begin{table}[H]
\caption{Overview of notation used in problem introduction}
\label{table:notation}
\centering
\scalebox{0.75}{\begin{tabular}{@{}ll@{}}
\toprule
\multicolumn{2}{c}{Problem Description}              \\ \midrule

\multicolumn{1}{l|}{$N$}                                                    & \multicolumn{1}{l}{Set of customers}                         \\
\multicolumn{1}{l|}{$V$}                                                    & \multicolumn{1}{l}{Set of vertices, $V=\{0\}\bigcup N$}      \\
\multicolumn{1}{l|}{$A$}                                                    & \multicolumn{1}{l}{Set of arcs connecting nodes $i$, $j$, where $i, j \in V$}      \\
\multicolumn{1}{l|}{$G$}                                                    & \multicolumn{1}{l}{A complete graph, $G = (V, A)$}      \\

\multicolumn{1}{l|}{$T_E$}                                                 & \multicolumn{1}{l}{Deadline, i.e., when the vehicle has to be back to the depot, and no further deliveries are allowed}                                                                  \\
\multicolumn{1}{l|}{$d_{ij}$}                                                 & \multicolumn{1}{l}{Traveling time from $i$ to $j$, where $i, j \in V$}                                    \\
\multicolumn{1}{l|}{$\mathcal{K}$}                                                    & \multicolumn{1}{l}{Set of routes, $\mathcal{K}=\{0, ..., |\mathcal{K}|-1\}$, including route 0 and the set of future routes denoted by K}                              \\
\multicolumn{1}{l|}{$e$}                                                    & \multicolumn{1}{l}{Decision epoch, $e\in \{0, ..., E\}$}                           \\
\multicolumn{1}{l|}{$t_e$}                                                 & \multicolumn{1}{l}{Time at decision epoch $e$}                                 \\
\multicolumn{1}{l|}{$N_e^{served}$}                     & \multicolumn{1}{l}{Set of customers already served before decision epoch $e$}                                 \\
\multicolumn{1}{l|}{$N_e^{unserved}$}                   & \multicolumn{1}{l}{Set of customers unserved at decision epoch $e$}                      
\\
\multicolumn{1}{l|}{$N_e^{known}$}                      & \multicolumn{1}{l}{Set of unserved customers whose parcels are available at the depot at decision epoch $e$} \\
\multicolumn{1}{l|}{$N_e^{unknown}$}                    & \multicolumn{1}{l}{Set of unserved customers whose parcels are still to be delivered to the depot}      \\
\multicolumn{1}{l|}{$N_e^{static}$}                     & \multicolumn{1}{l}{Set of customers whose information about release date is not updated until their parcels arrive at the depot} \\
\multicolumn{1}{l|}{$N_e^{dynamic}$}                    & \multicolumn{1}{l}{Set of customers whose information about release date is dynamically updated}                    \\
\multicolumn{1}{l|}{$\widetilde{r}^e_i$} & \multicolumn{1}{l}{Random variable associated with the release date of customer $i \in N_e^{unserved}$ at decision epoch $e$}                               \\
\multicolumn{1}{l|}{$r_i$}                                                 & \multicolumn{1}{l}{Actual arrival time of the parcel of customer $i \in N_e^{known}$}      \\
\midrule
\multicolumn{2}{c}{Components of MDP}                               \\
\midrule    
\multicolumn{1}{l|}{$\phi$}                                  & \multicolumn{1}{l}{Time interval between two additional decision epochs}                             \\
\multicolumn{1}{l|}{$S_e$}                                                 & \multicolumn{1}{l}{State at decision epoch $e$, $S_e = (t_e, N_e^{known}, \{\widetilde{r_i}^e\}_{i\in N_e^{unknown}}$)}                \\
\multicolumn{1}{l|}{$X(S_e)$}                                              & \multicolumn{1}{l}{Action space at decision epoch $e$}        \\
\multicolumn{1}{l|}{$\mathcal{X}_e$}                      & \multicolumn{1}{l}{Decision/action at decision epoch $e$, $\mathcal{X}_e \in X(S_e)$}              \\
\multicolumn{1}{l|}{$x^e_{ij}$}                      & \multicolumn{1}{l}{Binary arc variable at epoch $e$, taking value 1 if the route traverses arc $(i,j)$, 0 otherwise, where $i, j\in N_e^{known}\cup \{0\}$}              \\
\multicolumn{1}{l|}{$l(\mathcal{X}_e)$}                   & \multicolumn{1}{l}{Set of customer locations visited in the route associated with the action $\mathcal{X}_e$}                    \\

\multicolumn{1}{l|}{$t_{route}(\mathcal{X}_e)$}        & \multicolumn{1}{l}{Time required to perform the route associated with action $\mathcal{X}_e$}            \\
\multicolumn{1}{l|}{$t_p$}                                                 & \multicolumn{1}{l}{Earliest time when a new parcel arrives while the vehicle is at the depot} \\
\multicolumn{1}{l|}{$N_{e+1}^{new}$}                  & \multicolumn{1}{l}{Customers whose parcels arrived at the depot between decision epochs $e$ and $e+1$}                             \\
\multicolumn{1}{l|}{$V_e(S_e)$}            & \multicolumn{1}{l}{The value at decision epoch $e$ with state $S_e$}                   \\
\multicolumn{1}{l|}{$C(S_e,  \mathcal{X}_e)$}            & \multicolumn{1}{l}{Number of requests served by action $\mathcal{X}_e$}                   \\
\multicolumn{1}{l|}{$\gamma$}                                & \multicolumn{1}{l}{Discount factor} \\
\midrule
\multicolumn{2}{c}{Upper Bounds}                               \\
\midrule    
\multicolumn{1}{l|}{$|\mathcal{K}|$}                                  & \multicolumn{1}{l}{Number of routes performed by the vehicle}                             \\
\multicolumn{1}{l|}{$s_i^k$}                                  & \multicolumn{1}{l}{Binary variable for route $k \in\mathcal{K}$, equal to 1, if customer or depot $i \in V$ visited, 0 otherwise}                             \\
\multicolumn{1}{l|}{$\xi_{ij}^k$}                                  & \multicolumn{1}{l}{Binary arc variable for route $k \in\mathcal{K}$, equal to 1, if customer $i \in N$ visited, 0 otherwise}                             \\
\multicolumn{1}{l|}{$d_k$}                                  & \multicolumn{1}{l}{Starting time of route $k \in\mathcal{K}$}                             \\
\multicolumn{1}{l|}{$\alpha_i$}                                  & \multicolumn{1}{l}{Binary variable  equal to 1, if a direct route from the depot to customer $i \in N$ and back, is performed, 0 otherwise}                             \\
\bottomrule
\end{tabular}
}
\end{table}

\begin{table}[]
\caption{Overview of notation used in solution approaches}
\label{table:notation_sol}
\centering
\scalebox{0.65}{\begin{tabular}{@{}ll@{}}
\toprule
\multicolumn{2}{c}{Solution Approaches}              \\ \midrule
\multicolumn{1}{l|}{$\Omega$}                                                 & \multicolumn{1}{l}{Set of scenarios predicting their release dates  $\widetilde{r}^e_i$ at decision epoch e, associated with each customer $i \in N_e^{unserved}$}                                                                  \\
\multicolumn{1}{l|}{$\omega$}                                                    & \multicolumn{1}{l}{A realization of possible values of release dates for each customer $i \in N_e^{unserved}$, $\omega \in \Omega$}                         \\
\multicolumn{1}{l|}{$\rho$}                                                    & \multicolumn{1}{l}{Maximum number of requests to be served in each future route}      \\
\multicolumn{1}{l|}{$\mathcal{A}$}                                                    & \multicolumn{1}{l}{Euclidean area containing the locations of unserved customers}                           \\
\multicolumn{1}{l|}{$T_D$}                                                 & \multicolumn{1}{l}{Expected tour duration}                                 \\
\multicolumn{1}{l|}{$\theta$}                                                 & \multicolumn{1}{l}{Tuning parameters including $T_D$, the distribution of release dates, $\rho$, and $|\Omega|$.}                                 \\
\multicolumn{1}{l|}{$\tilde{S}_{e+1}$}                                  & \multicolumn{1}{l}{Approximation of state at $e+1$}                             \\
\multicolumn{1}{l|}{$\tilde{V}$}                                  & \multicolumn{1}{l}{The value of the value function estimated}                             \\
\multicolumn{1}{l|}{$\tau_{start}^0$}            & \multicolumn{1}{l}{A parameter indicating the starting time of the immediate route, so its value actually equals the time of the epoch $e$.}                   \\
\multicolumn{1}{l|}{$A^{known}$}                                & \multicolumn{1}{l}{Set of arcs $(i, j)$ where $i, j \in N_e^{known}\cup \{0\}$, linking customers with known requests and depot} \\
\multicolumn{1}{l|}{$p_\omega$}
             &
\multicolumn{1}{l}{Probability associated with each scenario $\omega \in \Omega$}
            \\
\multicolumn{1}{l|}{$route\ 0$}                               & \multicolumn{1}{l}{the first route leaving for delivery in $\mathcal{K}$, and for VFA approaches it starts immediately if non-empty in each decision epoch} \\
\midrule
\multicolumn{2}{c}{Batch Approach}                               \\
\midrule 
\multicolumn{1}{l|}{$K$}                      & \multicolumn{1}{l}{Set of indices for batches or future routes} \\
\multicolumn{1}{l|}{$K_0$}                    & \multicolumn{1}{l}{Set of indices for batches with spare capacities ($K_0\subseteq K$)}      \\
\multicolumn{1}{l|}{$k(i)$}                     & \multicolumn{1}{l}{Index of the batch in which the request $i\in N_e^{unknown}$ should be served} \\
\multicolumn{1}{l|}{$\tau_{start}^k$}                    & \multicolumn{1}{l}{Starting time of the route serving batch $k \in K$}           \\
\multicolumn{1}{l|}{$\tau_{end}^k$}                                                    & \multicolumn{1}{l}{Ending time of the route serving batch $k \in K$}         \\
\multicolumn{1}{l|}{$\rho_k$}                               & \multicolumn{1}{l}{Number of unknown requests assigned to the route serving batch $k\in K$} \\

\midrule
\multicolumn{2}{c}{Value Function Approximation with a Consensus Function}                               \\
\midrule    

\multicolumn{1}{l|}{$\tilde{X}^\omega_e$}                                                 & \multicolumn{1}{l}{Optimal action for the given value function approximation $\tilde{V}$ under scenario $\omega \in \Omega$ in decision epoch $e$}                \\
\multicolumn{1}{l|}{$\Phi$}                                              & \multicolumn{1}{l}{Consensus function for deciding the solution under the consideration of all the scenarios}        \\
\multicolumn{1}{l|}{$\tilde{X}^{PFA}_e$}                                                 & \multicolumn{1}{l}{Action decided by the consensus function $\Phi$}                \\
\multicolumn{1}{l|}{$x_{ij}^0$}                      & \multicolumn{1}{l}{Binary arc variable for route 0, equal to 1 if the route traverses arc $(i, j) \in A^{known}$, 0 otherwise}              \\
\multicolumn{1}{l|}{$y_i^0$}                   & \multicolumn{1}{l}{Binary variable associated with known requests $ i \in N_e^{known}$, equal to 1, if customer $i$ is visited in route $0$, and 0 otherwise}                    \\
\multicolumn{1}{l|}{$z_k$}                   & \multicolumn{1}{l}{Binary variable equal to 1   if future route $k \in K$ is  executed, 0 otherwise}      \\
\multicolumn{1}{l|}{$\tau_{end}^0$}        & \multicolumn{1}{l}{Continuous variable denoting the ending time of route 0}            \\
\multicolumn{1}{l|}{$\pi_i^k$}                  & \multicolumn{1}{l}{Binary variable equal to 1 if known request $i \in N_e^{known}$ is served in batch $k \in K_0$, 0 otherwise} 
\\
\multicolumn{1}{l|}{$\lambda$}                  & \multicolumn{1}{l}{the percentage of scenarios in which a location must appear for being included in route 0} \\
\midrule
\multicolumn{2}{c}{Value Function Approximation with a Two-Stage Stochastic ILP Model}                               \\
\midrule   
\multicolumn{1}{l|}{$K^\omega$}                                  & \multicolumn{1}{l}{Set of indices for all the batches created in scenario $\omega$, $\omega \in \Omega$}     
\\
\multicolumn{1}{l|}{$K_0^\omega$}                                  & \multicolumn{1}{l}{Subset of batches with spare capacities in scenario $\omega \in \Omega$}                             \\
\multicolumn{1}{l|}{$\tau^{k\omega}_{start}$}                                  & \multicolumn{1}{l}{Starting time of batch $k\in K^\omega$ in scenario $\omega\in \Omega$}                             \\
\multicolumn{1}{l|}{$\rho_k^\omega$}                                & \multicolumn{1}{l}{Number of unknown requests assigned to the route serving batch $k\in K^\omega$ in scenario $\omega\in \Omega$} \\

\multicolumn{1}{l|}{$\pi_i^{k\omega}$}                                  & \multicolumn{1}{l}{Binary decision variable equal to 1 if known request $i \in N_e^{known}$ is served in batch  $k$ in scenario $\omega \in \Omega$, $k \in K_0^\omega$, and 0 otherwise}                             \\
\multicolumn{1}{l|}{$z_k^\omega$}                                  & \multicolumn{1}{l}{Binary decision variable equal to 1 if batch $k\in K^\omega$ in scenario $\omega \in \Omega$ is executed, and 0 otherwise}                             \\
\multicolumn{1}{l|}{$\widetilde{r}^e$}                                  & \multicolumn{1}{l}{The distribution of uncertain release dates for unserved requests}                             \\
\multicolumn{1}{l|}{$\tilde{V}(y^0, \tau_{end}^0, \widetilde{r}^e)$}                                  & \multicolumn{1}{l}{The value of the recourse function associated with the second stage}                             \\
\bottomrule
\end{tabular}
}
\end{table}

\section{Supplement to the literature review}
\label{sec:liter}
\subsection{Paper difference}
\label{liter:diff}
 The following outlines some key differences between \cite{anuar2021multi} and our paper: 1) Problem domain and uncertainty factors: \cite{anuar2021multi} work on humanitarian applications, whereas our paper focuses on same-day deliveries. 
Furthermore, in the approach proposed by \cite{anuar2021multi},
the computation of a complete route at each decision epoch aims at incorporating the uncertain road capacity.
In contrast, we aim at maximizing the expected number of requests while incorporating the uncertainty of release dates, a prevalent aspect in SDD business practices;   2) Model complexities: The models proposed by \cite{anuar2021multi} primarily focus on determining the next destination for each vehicle at each decision epoch. In contrast, our models calculate multiple future routes, and the decision made pertains to building a complete route. This difference in the modeling approach leads to variations in the complexity and structure of the ILP models employed; 3) RL approaches: \cite{anuar2021multi} use a rollout approach, whereas we employ VFAs with a one-step look-ahead. Notably, the idea of approximating future routes is not employed by \cite{anuar2021multi}.

\subsection{Summary of literature review}
\label{sec:appendix_literature}
See Table \ref{table:literature}. We emphasize that we define release dates as \emph{dynamic} if the information regarding their distribution is updated throughout the day. The papers listed mostly study dynamic problems where ``dynamic'' means customer orders are not fully known in advance and arrive dynamically at the CDC, but few of them consider updated release dates.

\begin{sidewaystable} 
\caption{Literature Comparison}
\label{table:literature}
\centering
\scalebox{0.66}{\begin{tabular}{|c|c|cc|c|c|c|c|}
\hline
                    &                                                                      & \multicolumn{2}{c|}{\textbf{Release dates}}                                &                           &                                                                                                                                                   &                           &                                                                                                                                                                     \\ \hline
\textbf{Literature} & \textbf{Vehicle}                                                     & \multicolumn{1}{c|}{\textbf{Stochastic}}       & \textbf{Dynamic}          & \textbf{Deadline}         & \textbf{Objective}                                                                                                                                & \textbf{MDP}              & \textbf{Approach}                                                                                                                                                   \\ \hline
\cite{cattaruzza2016multi}           & Single                                                               & \multicolumn{1}{c|}{\xmark}     & \xmark     & \checkmark & Min(total travel distance)                                                                                                                        & \xmark     & Hybrid genetic algorithms                                                                                                                                           \\ \hline
\multirow{2}{*}{\cite{archetti2015complexity}}    & Single                                                               & \multicolumn{1}{c|}{\xmark}     & \xmark     & \checkmark & \begin{tabular}[c]{@{}c@{}}Min(total travel distance\\ within deadline) and \end{tabular}                                                         & \xmark     & -                                                                                                                                                                   \\
    & Unlimited                                                            & \multicolumn{1}{c|}{\xmark}     & \xmark     & \xmark     & Min(the completion time)                                                                                                                          & \xmark     & -                                                                                                                                                                   \\ \hline
\multirow{2}{*}{ \cite{reyes2018complexity}}    & Single                                                               & \multicolumn{1}{c|}{\xmark}     & \xmark     & \checkmark & 
\begin{tabular}[c]{@{}c@{}}Min(the completion time) and\\ Min(total travel distance within deadline) \end{tabular}& \xmark     & - \\ 
    & Unlimited                                                            & \multicolumn{1}{c|}{\xmark}     & \xmark     & \checkmark & Min(the travel distance)                                                                                                                          & \xmark     & -                                                                                                                                                                   \\ \hline
\cite{shelbourne2017vehicle}           & Multi                                                                & \multicolumn{1}{c|}{\xmark}     & \xmark     & \checkmark & \begin{tabular}[c]{@{}c@{}}Min(a convex combination of\\ the total distance traveled and\\ the total weighted tardiness of delivery)\end{tabular} & \xmark     & \begin{tabular}[c]{@{}c@{}}A path-relinking algorithm\\ based on\\ a hybrid evolutionary framework.\end{tabular}                                                    \\ \hline
\cite{archetti2018iterated}       & Single                                                               & \multicolumn{1}{c|}{\xmark}     & \xmark     & \xmark     & Min(the completion time)                                                                                                                          & \xmark     & \begin{tabular}[c]{@{}c@{}}A heuristic approach based on\\ an iterated local search\end{tabular}                                                                    \\ \hline
\cite{archetti2020dynamic}       & Single                                                               & \multicolumn{1}{c|}{\checkmark} & \checkmark & \xmark     & Min(the expected completion time)                                                                                                                          & \checkmark & \begin{tabular}[c]{@{}c@{}}Two models proposed\\ (stochastic and deterministic), \\ the solution is obtained \\ through an iterated local search\end{tabular}         \\ \hline
\cite{voccia2019same}              & \begin{tabular}[c]{@{}c@{}}Multi\\ (plus a third party)\end{tabular} & \multicolumn{1}{c|}{\checkmark} & \xmark & \checkmark & Max(the expected \#requests served)                                                                                                                            & \checkmark & \begin{tabular}[c]{@{}c@{}}Sample-scenario planning and\\ a variable neighborhood search \\ implemented for each scenario\end{tabular}                              \\ \hline
\cite{van2019delivery}                 & \begin{tabular}[c]{@{}c@{}}Multi\\ (plus a third party)\end{tabular} & \multicolumn{1}{c|}{\checkmark} & \xmark & \checkmark & Min(the expected total dispatching costs)                                                                                                                  & \checkmark & \begin{tabular}[c]{@{}c@{}}Approximate dynamic programming\\ using a linear value function approximation\\ to estimate the downstream costs\end{tabular}           \\ \hline
\cite{klapp2018one}               & Single                                                               & \multicolumn{1}{c|}{\checkmark} & \xmark & \checkmark & \begin{tabular}[c]{@{}c@{}}Min(the expected total travel costs\\ plus penalties\\  for unattended and realized orders)\end{tabular}                     & \checkmark & \begin{tabular}[c]{@{}c@{}}Three heuristic policies(a priori, direct rollout, \\ roll out a prize-collecting TSP guided\\ by an initial a priori solution)\end{tabular} \\ \hline

\cite{schrotenboer2021fighting}       & Multi                                                               & \multicolumn{1}{c|}{\checkmark}     &   \xmark   & \checkmark     & Max(the expected customer satisfaction)     & \checkmark     & \begin{tabular}[c]{@{}c@{}}A cost function approximation approach \\ that modifies a set-packing formulation\end{tabular}                                  \\\hline

This paper     & Single and two vehicles                                                               & \multicolumn{1}{c|}{\checkmark}     & \checkmark     & \checkmark     & Max(the expected \#requests served)     & \checkmark     & \begin{tabular}[c]{@{}c@{}}VFA-CF and VFA-2S approaches proposed\\ relying on the approximation of future routes\end{tabular}                                  \\\hline

\end{tabular}}
\end{sidewaystable}

\section{Upper Bound: the Orienteering Problem}\label{sec:ub_appendix}

A simple upper bound for the DOP-rd can be obtained by solving a special instance of the Orienteering Problem (OP). 
OP is defined on a  complete graph $\hat G=(\hat V, \hat A)$ in which we are given arc weights $\hat w_{ij} >0$, $(i,j) \in \hat A$, and node prizes $\hat p_i>0$, $i \in \hat V\setminus\{0\}$, where $0$ is the depot. The goal is to find a route whose total arc weight does not exceed a budget  $\hat B >0$, and the total collected prize from visited nodes is maximized.

In each decision epoch $e$, we can calculate a valid upper bound on the maximum number of parcels that can be delivered within the interval $[t_e,T_E]$ in the DOP-rd. This bound is given as a solution of the OP over the set of parcels in $N^{unserved}_e$, with all parcel prizes being set to one, and the maximum budget set to  $T_E - t_e$. 

Such obtained upper bound can be used in a competitive analysis. Moreover, it gives rise to the partial characterization of the optimal policy described in Section 4.2. The methodology used for solving the OP is based on a branch-and-cut (see, e.g., \cite{FisST98}), which is one of the most effective exact approaches for the OP.

\section{Proof of Theorem 1}\label{sec:proof}

\proof{Proof.} A solution of Formulation (2) consists of a sequence of non-empty feasible and compatible routes. A route is feasible if it does not start before any of the release dates of the customers to be visited and finishes before $T_E$. Routes are compatible if they visit each customer at most once. Considering a customer $i$, the shortest feasible route serving $i$ is the one that leaves the depot at time $r_i$, goes to $i$ and comes back to the depot. Constraints (3b) guarantee the feasibility of the direct routes selected: indeed, for each customer $i$, constraints (3b) ensure that the route to $i$, as well as the following direct routes can be performed before the deadline $T_E$. Thus, Formulation (3) determines the maximum number of direct routes performed within $T_E$.
Let us assume that there exists a solution $\tilde{s}$ with a higher number of routes than $UB_{routes}$. Given the argument above, at least one of the routes in $\tilde{s}$ is not a direct route. Thus, we can arbitrarily remove customers from this route up to when a single customer remains, making it a direct route. This procedure can be repeated on all routes containing more than one customer in $\tilde{s}$. In this way, we obtain a solution with the same number of routes as in $\tilde{s}$, and where all routes are direct routes. However, due to the argument above, the number of routes cannot be greater than $UB_{routes}$, and this proves point 1 of the theorem.\\
As for point 2, feasibility is ensured by constraints (3b): a route starts not earlier than the delivery date of the customer to be served, and its duration plus the duration of the following routes does not exceed $T_E$. $\qed$

\section{Optimality Conditions for the Batch Approach}
\label{optimality_batch}
In the following, we show that under some simplifying conditions, the batch approach given by Algorithm 1 provides a polynomial way to calculate an optimal solution.
Let us assume that we are at the decision epoch $e$ and that all the release dates are known with certainty. We denote as $r$ the vector of release dates. In the following proposition, we still use the notation $N_e^{unknown}$ referring to parcels whose release date is greater than $t_e$. 

\begin{prop}\label{prop:maxnumber}
If the release dates are deterministic, the duration of each route is $T_D$, independently of the location of the requests served, and the vehicle can deliver up to $\rho$ parcels per route, Algorithm 1 finds (in polynomial time):
\begin{enumerate}
    \item The maximum number of  requests from $N_e^{unknown}$ that can be served within $[t_e,T_E]$, given as $| N_{\mathcal O}|$, where
        $N_{\mathcal O} =  \{ i \in N_e^{unknown} : k(i) \neq 0 \}$. 
    \item The maximum number of all requests from $N_e^{unserved}$ that can be served within the interval $[t_e,T_E]$, given as:
    \[
    OPT_{\rho,T_D} = \begin{cases}
    |N^{known}_e| + |N_\mathcal{O}|, & \text{ if $ |N^{known}_e| < \sum_{k \in K_0} (\rho - \rho_k)$}  \\
    |K|* \rho, & \text{ otherwise}  
    \end{cases}
    \]
    \end{enumerate}
\end{prop}
\proof{Proof.}
We first observe that in case $\rho_k=\rho$ for all $k \in K$, then this is trivially the optimal solution. Thus, in the following, we discard this case.
\begin{enumerate}
    \item By construction of the solution $ N_{\mathcal O}$, there exists $\tilde{k}$, $1 \le \tilde{k} \le |K|$, such that $\rho_{\tilde{k}}< \rho$, $\rho_k=\rho$ for $k < \tilde{k}$ and $\rho_k =0$ for $k>\tilde{k}$. The batch $\tilde{k}$ is called \textit{critical batch} and corresponds to the only possible batch whose spare capacity is positive but strictly smaller than $\rho$. We call this property \textit{the critical batch property}. Suppose that there exists a solution $N_\mathcal{O'} \subseteq N_e^{unknown}$ such that $|N_\mathcal{O'}| > |N_\mathcal{O}|$   parcels from $N_e^{unknown}$ can be delivered within the interval $[t_e,T_E]$. Without loss of generality, we can shift the starting time of all batches to the latest possible and then reassign parcels to the latest possible batch so that the solution $N_\mathcal{O'}$ also satisfies the critical batch property. We argue that the number of parcels scheduled in batch $\tilde{k}$  of the reordered solution $N_\mathcal{O'}$ cannot be bigger than the one in $N_\mathcal{O}$. Indeed, by contradiction: let us consider a request $i$ served in batch $\tilde{k}$ in $N_\mathcal{O'}$ and not in $N_\mathcal{O}$. In case $i$ is not served in any batch $k<\tilde{k}$ in $N_\mathcal{O}$, then $i$ would have been inserted in batch $\tilde{k}$ in solution $N_\mathcal{O}$ as well. In case, instead, request $i$ is served in one batch $k<\tilde{k}$ in $N_\mathcal{O}$, then it means there exists at least one request $i'$ served in a batch $k<\tilde{k}$ in $N_\mathcal{O'}$ but not in $N_\mathcal{O}$. However, due to the way batches in $N_\mathcal{O}$ are constructed, this means that $r_i^e \geq r_{i'}^e$. Given that $i$ is inserted in batch $\tilde{k}$ in $N_\mathcal{O'}$, this means that $r_{i'}^e \leq \tau_{start}^{\tilde{k}}$, thus $i'$ and $i$ could be swapped in the corresponding batches in $N_\mathcal{O'}$. By repeating iteratively this procedure on all requests in batch $\tilde{k}$ in $N_\mathcal{O'}$, we obtain two identical sets of requests in the critical batch $\tilde{k}$.


    \item We notice that $\sum_{k \in K_0} (\rho - \rho_k)$ represents the sum of spare capacities over all routes that can be scheduled within $[t_e,T_E]$. Since the parcels from $N_e^{known}$ are already available at time $t_e$, if their number is smaller than the total spare capacity, all of them can be distributed among the routes determined by $K_0$, hence guaranteeing that $|N^{known}_e| + |N_\mathcal{O}|$ parcels are delivered in total. If, on the contrary,  $|N^{known}_e|$ is larger than the total spare capacity, only a subset of them will be delivered so that for each batch, its maximal capacity is used. We emphasize that due to the assumption that all routes have the same duration, namely $T_D$, and the same capacity $\rho$, customer locations are irrelevant in this case, and hence, the way how assignment of parcels from $N^{known}_e$ to batches from $K_0$ is done, does not affect the optimal solution. $\qed$
\end{enumerate}

\section{Input Instances} \label{sec:appendix_instances}
We summarize the procedure for generating the benchmark instances, similar to the one used in \cite{archetti2020dynamic}.

Solomon's instances contain six sets of instances: $C1$, $C2$, $R1$, $R2$, $RC1$, and $RC2$. The instances in the same set share the same vertex coordinates and vary in time windows only. Because we do not discuss time windows, we keep one instance from each set. Besides, we discard  $R2$ and $RC2$ because they share the same coordinates as $R1$ and $RC1$. Each of the instances is annotated as $I\_\beta$, $I$ is the instance name, including $C101$, $C201$, $R101$, and $RC101$, where the vertex coordinates are taken, and $\beta$ describes how much the release dates spread out, compared to the TSP travel time of the instance.

The release dates are generated using a Gaussian distribution with the expectation and the variance obtained by simulating the arrival of the vehicles transporting the packages to the depot. For customers in $N_e^{dynamic}$ (i.e., the set of customers whose information about the release date is dynamically updated along with the parcels' travel to the depot), both the expectation and the variance are updated by simulating the traveling of the vehicle for delivering the packages to the depot. The distance of each vehicle is updated by reducing the distance traveled in the previous time unit. If the updated distance is nonpositive at $e$, it means the vehicle has delivered the parcel to the depot, and the customer $i$ is added in the set $N_e^{known}$ with release date $r_i = t_e$. Otherwise, the vehicle speed is updated as a truncated random walk process with Gaussian steps, and the new expectations and variances of the distributions are computed as well. Specifically, as the vehicle travels to the depot, the uncertainties decrease, and the estimation of the release date gets more precise. We use $\delta\in \{0, 0.5, 1\}$ to indicate the rate of customers with a dynamically updated release date. When $\delta$ is 0, it means all distributions of customers' release dates are static. Besides, 
%
for each Solomon's instance, a parameter  $\beta \in \{0.5, 1, 1.5\}$ is generated, defining the dispersion of release dates. The greater the value of $\beta$, the broader the interval the release dates are sampled from. 

For each of Solomon's instances and value of parameters $\beta$,  $\delta$, and $c$, 
five instances have been created by varying the seed for random release date generation. In our paper, we report the results considering all five seeds, where the instances of two seeds are used for hyper-parameter tuning and the remaining ones are used for comparing with benchmarks. As for the pairwise Euclidean distances between customer coordinates, they are rounded up to the lowest integer values.

\section{Benefit of partial characterization of optimal policy}
\label{sec:partial_benefit} 
We illustrate the benefit of partial characterization (PC) of optimal policy (see Section 4.2) by comparing the performance of VFA-CF and VFA-2S with and without PC. Results are shown in Table \ref{partial_comparison2} where instances are classified according to the value on $\delta$ first, then $\beta$, and finally $c$. For each of the two approaches, the table reports the average computing time (in seconds) with and without PC, the average and maximum percentage improvement in policy value of the approach with PC with respect to the corresponding one without PC (over the set of improved instances), and the number of times the approach with PC improved on the policy  without PC.

In Table \ref{partial_comparison2}, 
we observe that when using PC, the running time decreases in VFA-CF. The reduction of time in VFA-CF is because the use of PC avoids some runs of the deterministic ILPs. On the other hand, in VFA-2S, the use of PC slightly increases the running time. Furthermore, we observe that policies are improved more significantly when release dates are spread at a moderate level ($\beta=1$) and when the time horizon is large ($c=0.8, 1, 1.2$).
Overall, we have an improvement of {3.44\%} for {VFA-CF} and {3.05\%} for VFA{-2S} over the set of instances improved, with maximum improvements being {11.76\% and 6.98\%}, respectively. Considering both the improvement in policy quality for the subset of improved instances and the reduction in running time, we see that the inclusion of PC of the optimal policy demonstrates an overall improvement in the performance of our proposed methods.


\begin{table}[h]
\captionsetup{labelfont={color=blue},font={color=blue}}
\caption{Comparison of VFA-CF and VFA-2S 
with and without partial characterization of optimal policy
}
\label{partial_comparison2}
\centering
\scalebox{0.68}{\color{black}
\begin{tabular}{lc||cc|ccc||cc|ccc}  \hline             &                & \multicolumn{5}{c||}{\textbf{VFA-CF}}                                                        & \multicolumn{5}{c}{\textbf{VFA-2S}}                                                       \\\hline
              &                & \multicolumn{2}{c|}{t[s]}        & \multicolumn{3}{c||}{Improvement}                 & \multicolumn{2}{c|}{t[s]}        & \multicolumn{3}{c}{Improvement}                \\\hline
              & \#instances    & noPC           & withPC         & AVG(\%)       & MAX(\%)        & \#impr. & noPC           & withPC         & AVG(\%)       & MAX(\%)       & \#impr. \\\hline\hline
$\delta=0$       & 144            & 82.80          & 43.70          & 3.10          & 5.00           & 9              & 34.46          & 32.08          & 2.81          & 5.00          & 11             \\
$\delta=0.5$     & 144            & 49.61          & 45.34          & 3.93          & 9.68           & 7              & 34.81          & 36.20          & 3.74          & 6.98          & 7              \\
$\delta=1$       & 144            & 69.36          & 60.45          & 4.67          & 11.76          & 9              & 39.03          & 42.85          & 4.07          & 6.82          & 8              \\\hline
$\beta=0.5$      & 144            & 101.51         & 64.90          & 0.00          & 0.00           & 0              & 45.87          & 45.50          & 0.00          & 0.00          & 0              \\
$\beta=1$        & 144            & 57.44          & 47.54          & 4.20          & 11.76          & 20             & 31.94          & 33.02          & 3.69          & 6.98          & 21             \\
$\beta=1.5$      & 144            & 42.82          & 37.04          & 2.70          & 4.44           & 5              & 30.49          & 32.61          & 2.44          & 2.86          & 5              \\\hline
$c=0.6$         & 108            & 47.91          & 44.88          & 2.50          & 2.50           & 1              & 30.80          & 33.04          & 2.82          & 2.86          & 2              \\
$c=0.8$         & 108            & 53.98          & 47.34          & 5.63          & 11.76          & 5              & 30.53          & 33.18          & 3.41          & 5.56          & 5              \\
$c=1$           & 108            & 107.56         & 57.73          & 5.11          & 6.82           & 7              & 44.28          & 40.50          & 4.74          & 6.82          & 7              \\
$c=1.2$         & 108            & 59.57          & 49.36          & 2.59          & 5.13           & 12             & 38.78          & 41.45          & 2.81          & 6.98          & 12             \\\hline
\textbf{AVG/MAX} & \textbf{129.6} & \textbf{67.26} & \textbf{49.83} & \textbf{3.44} & \textbf{11.76} & \textbf{7.5}   & \textbf{36.10} & \textbf{37.04} & \textbf{3.05} & \textbf{6.98} & \textbf{7.8}  
 \\\hline
\end{tabular}}
\end{table}

\section{Hyper-parameter tuning}\label{sec:training}

To justify the choice of our hyper-parameters, two out of the five seeds of instance generation have been used for hyper-parameter tuning. Specifically, we tune the number of scenarios $|\Omega|$, batch size $\rho$, discount factor $\gamma$, and $\lambda$ - the percentage of scenarios in which a location must appear to be included in route 0. We run the policies with PC with different parameter values and present the results in the following. We adopt a set of default parameter values: $|\Omega| = 30$, $\gamma = 0.7$, $\lambda = 0.5$. When we tune one of the parameters, we use the default values of the other parameters.

To compare the performance, we present the percentage gap from the best policy among the set of parameter values for the same approach. To better clarify, we provide an example of scenario numbers. For each of the two approaches $p \in \{ \text{VFA-CF, VFA-2S} \}$, and a given input instance, let $N_{p,|\Omega|}$ denote the number of total requests served, assuming approach $p$ is applied with a given value of $|\Omega|$.
Then 
\begin{equation}
\label{eq:gap_Nk}
    \Gamma_{p, |\Omega|} = 1 - \frac{N_{p,|\Omega| }}{\max_{|\Omega| \in \{10, 30, 50\}}  N_{p,|\Omega|}} 
\end{equation}
provides the gap between the policy obtained for a given value of $|\Omega|$ and the best policy among the three different values of $|\Omega|$. 

\subsection{Number of scenarios}
\label{sec:scenario}
We first choose three possible numbers of scenarios, namely 10, 30, and 50, and run the policies with each number. We display the average percentage gaps and the running time for each scenario number summarized by $\delta$, $\beta$, and time scale $c$. In Table \ref{table:scenario}\footnote{It is worth noting that we identified an outlier in the instance when $\delta=1$, $\beta=0.5$, and $c=1.2$, which has a significantly longer running time equal to 619.88s for VFA-2S with 30 scenarios. To provide a more generalized analysis, we removed this specific running time and obtained the following results: $\delta=1$-27.85, $\beta=0.5$-25.65, and $c=1.2$-27.16, with an average of 27.80, as shown in the table. The average results before the removal were $\delta=1$-34.02, $\beta=0.5$-31.84, and $c=1.2$-35.40, with an overall average of 33.04.}, each row corresponds to an average calculated over a subset of instances with the fixed values of $\delta$, $\beta$, and $c$ parameters on all instances, respectively. We observe that for both VFA-CF and VFA-2S, the gaps are similar across the three settings, and there is a slight advantage shown for a higher number of scenarios; the running time increases with the number of scenarios, which is reasonable as the more scenarios, the more runs in VFA-CF and the bigger model size in VFA-2S. To balance policy quality and running time, we set $|\Omega|$ equal to 30.

\begin{table}[htp]
\caption{Average percentage gap $\Gamma_{p, |\Omega|}$ from the best policy found for  $|\Omega| \in \{10,30,50\}$ and $p \in \{\text{VFA-CF}, \text{VFA-2S}\}$.}
\label{table:scenario}
\centering
\scalebox{0.7}{\color{black}
\begin{tabular}{l||cccccc||cccccc}
\toprule
                          & \multicolumn{6}{c}{\textbf{VFA-CF}}& \multicolumn{6}{c}{\textbf{VFA-2S}} \\\hline
\textbf{$|\Omega|$}       & \multicolumn{2}{c}{\textbf{10}} & \multicolumn{2}{c}{\textbf{30}} & \multicolumn{2}{c}{\textbf{50}} & \multicolumn{2}{c}{\textbf{10}} & \multicolumn{2}{c}{30}          & \multicolumn{2}{c}{50}          \\\hline
                          & gap{[}\%{]}     & t{[}s{]}       & gap{[}\%{]}     & t{[}s{]}       & gap{[}\%{]}     & t{[}s{]}       & gap{[}\%{]}     & t{[}s{]}       & gap{[}\%{]}     & t{[}s{]}       & gap{[}\%{]}     & t{[}s{]}       \\\hline
$\delta$=0   & 0.37           & 36.21          & 0.03           & 45.69          & 0.06           & 53.07          & 0.51           & 26.40          & 0.45          & 27.70          & 0.17          & 29.80          \\
$\delta$=0.5 & 0.24           & 35.59          & 0.40           & 45.91          & 0.33           & 55.75          & 0.54           & 26.54          & 0.87          & 27.85          & 0.42          & 29.78          \\
$\delta$=1   & 0.43           & 45.93          & 0.13           & 52.79          & 0.10           & 69.22          & 0.68           & 27.41          & 0.40          & 27.85          & 0.32          & 30.06          \\\hline
$\beta$=0.5  & 0.69           & 47.21          & 0.15           & 59.29          & 0.03           & 81.70          & 0.62           & 24.22          & 0.52          & 25.65          & 0.38          & 27.41          \\
$\beta$=1    & 0.14           & 36.35          & 0.20           & 47.25          & 0.31           & 53.85          & 0.52           & 27.48          & 0.71          & 28.82          & 0.31          & 31.19          \\
$\beta$=1.5  & 0.21           & 34.16          & 0.21           & 37.85          & 0.16           & 42.50          & 0.59           & 28.65          & 0.49          & 28.90          & 0.23          & 31.04          \\\hline
$c$=0.6                     & 0.45           & 36.94          & 0.32           & 44.82          & 0.32           & 49.63          & 0.58           & 28.73          & 0.57          & 29.57          & 0.22          & 32.39          \\
$c$=0.8                     & 0.38           & 52.39          & 0.07           & 62.36          & 0.11           & 83.32          & 0.55           & 26.87          & 0.80          & 28.41          & 0.51          & 31.17          \\
$c$=1                       & 0.53           & 34.60          & 0.26           & 41.51          & 0.09           & 46.97          & 0.88           & 25.89          & 0.46          & 26.04          & 0.35          & 27.53          \\
$c$=1.2                     & 0.04           & 33.03          & 0.10           & 43.82          & 0.14           & 57.48          & 0.30           & 25.64          & 0.47          & 27.16          & 0.14          & 28.42          \\\hline
\textbf{AVG}              & \textbf{0.35}  & \textbf{39.24} & \textbf{0.19}  & \textbf{48.13} & \textbf{0.16}  & \textbf{59.35} & \textbf{0.58}  & \textbf{26.78} & \textbf{0.57} & \textbf{27.80} & \textbf{0.30} & \textbf{29.88}\\\hline
\end{tabular}}
\end{table}

\subsection{Batch size}\label{sec:batchSize}

In this section, we present the results of the experiments we made to set a proper value of $\rho$, i.e., the maximum number of packages transported in future routes used in the batch approach. We test four values: $\rho \in \{ 5, 10, 15, 20\}$. 

Table \ref{table:nk2} shows the percentage gaps
for the VFA-CF and VFA-2S approaches, respectively. 
\begin{table}[tb!]
\begin{minipage}{0.6\linewidth}{
\caption{Average percentage gap $\Gamma_{p, \rho}$ from the best policy found for  $\rho \in \{5,10,15,20\}$ and $p \in \{\text{VFA-CF}, \text{VFA-2S}\}$.}
\label{table:nk2}
\centering
\scalebox{0.75}{\color{black}
\begin{tabular}{l||cccc||cccc}
\hline
\textbf{}    & \multicolumn{4}{c}{\textbf{VFA-CF}}                           & \multicolumn{4}{c}{\textbf{VFA-2S}}                           \\\hline
\textbf{$\rho$} & 5    & 10   & 15   & 20   & 5    & 10   & 15   & 20   \\\hline\hline 
$\delta=0$      & 7.11          & 5.25          & 2.83          & 4.49          & 6.94          & 5.85          & 3.17          & 4.50          \\
$\delta=0.5$    & 6.72          & 4.85          & 4.23          & 5.58          & 5.90          & 4.70          & 4.31          & 6.12          \\
$\delta=1$      & 5.18          & 5.74          & 3.61          & 4.00          & 5.05          & 5.13          & 3.49          & 4.19          \\\hline 
$\beta=0.5$     & 10.00         & 8.48          & 5.20          & 8.09          & 9.87          & 8.41          & 5.70          & 8.91          \\
$\beta=1$       & 6.37          & 4.79          & 3.62          & 4.09          & 5.90          & 4.94          & 3.47          & 4.11          \\
$\beta=1.5$     & 2.64          & 2.57          & 1.84          & 1.89          & 2.13          & 2.33          & 1.81          & 1.80          \\\hline 
$c=0.6$        & 9.12          & 7.36          & 6.32          & 8.37          & 8.63          & 7.32          & 6.18          & 9.64          \\
$c=0.8$        & 8.52          & 6.55          & 3.57          & 5.04          & 8.05          & 6.21          & 4.36          & 5.05          \\
$c=1$          & 3.85          & 4.70          & 2.78          & 3.38          & 3.89          & 3.82          & 2.50          & 2.83          \\
$c=1.2$        & 3.86          & 2.50          & 1.54          & 1.96          & 3.29          & 3.54          & 1.60          & 2.23          \\\hline \hline 
\textbf{AVG} & \textbf{6.34} & \textbf{5.28} & \textbf{3.55} & \textbf{4.69} & \textbf{5.97} & \textbf{5.23} & \textbf{3.66} & \textbf{4.94}
\\\hline 
\end{tabular}
}}\end{minipage}
\begin{minipage}{0.4\linewidth}{
\caption{Average percentage gap $\Gamma_{p, \gamma}$ from the best policy found for  $\gamma \in \{0.7,0.8,0.9,1.0\}$ and $p \in \{\text{VFA-CF}, \text{VFA-2S}\}$.}
\label{table:discount}
\centering
\scalebox{0.7}{\color{black}
\begin{tabular}{l||cccc||cccc}
\toprule
                          & \multicolumn{4}{c}{\textbf{VFA-CF}}              & \multicolumn{4}{c}{\textbf{VFA-2S}}              \\\hline
$\gamma$                  & 0.7            & 0.8            & 0.9       & 1.0     & 0.7            & 0.8            & 0.9     & 1.0       \\\hline\hline
$\delta$=0   & 1.27          & 2.00          & 1.99          & 3.24          & 2.05          & 1.98          & 2.33       & 4.04          \\
$\delta$=0.5 & 1.77          & 1.53          & 2.27          & 4.64          & 1.81          & 1.29          & 2.38          & 3.72          \\
$\delta$=1   & 2.19          & 1.97          & 2.20          & 3.16          & 2.20          & 2.37          & 2.19          & 3.13          \\\hline
$\beta$=0.5  & 2.77          & 2.70          & 2.81          & 2.67          & 3.37          & 2.72          & 3.18          & 2.93          \\
$\beta$=1    & 1.58          & 1.56          & 1.79          & 3.83          & 1.73          & 1.87          & 1.79          & 3.92          \\
$\beta$=1.5  & 0.88          & 1.24          & 1.86          & 4.55          & 0.96          & 1.04          & 1.92          & 4.04          \\\hline
$c$=0.6                     & 2.10          & 2.26          & 2.42          & 3.39          & 2.25          & 1.84          & 2.92          & 3.20          \\
$c$=0.8                     & 1.78          & 2.05          & 2.73          & 4.98          & 2.72          & 2.60          & 2.79          & 4.79          \\
$c$=1                       & 1.56          & 1.69          & 1.89          & 3.48          & 1.41          & 1.67          & 1.80          & 3.96          \\
$c$=1.2                     & 1.53          & 1.33          & 1.57          & 2.88          & 1.69          & 1.41          & 1.68          & 2.56          \\\hline\hline
\textbf{AVG}              & \textbf{1.74} & \textbf{1.83} & \textbf{2.15} & \textbf{3.68} & \textbf{2.02} & \textbf{1.88} & \textbf{2.30} & \textbf{3.63} \\\hline
\end{tabular}}}
\end{minipage}
\end{table}

For both VFA-CF and VFA-2S, we see that, on average, the performance improves when increasing $\rho$ from 5 to 15 but deteriorates when it increases to 20. The best results are obtained when $\rho$ equals $15$. Specifically, when $\rho=15$, the gaps are the smallest in most cases, while $\rho=5$ and $20$ show some exceptionally good results in a few cases, for example, when $\beta=1.5$ in VFA-2S. 
Based on the overall results,  we opt for a value of $\rho=15$ as it generally provides the best results. 

\subsection{Discount factor}
\label{sec:2approaches_discount}
We now analyze the impact of the discount factor on the performance. In Table \ref{table:discount}, we see that the results are similar among the three values of $\gamma \in \{0.7, 0.8, 0.9, 1.0\}$ for both policies, which means that the choice on $\gamma$  does not significantly influence the performance. To be consistent, we choose the value of $\gamma=0.8$ for both VFA-CF and VFA-2S.

\subsection{Performance of the VFA-CF on the basis of \(\lambda\)}
\label{sec:appendix_ps}
We test the performance of VFA-CF based on the percentage of scenarios in which a location must appear to be included in route 0, represented by $\lambda$ in Algorithm 1 in the paper. Similarly to other subsections of Section \ref{sec:training}, we report the average gap (\%) compared to the best policy obtained for VFA-CF over the three tested values of $\lambda$.  In Table \ref{table:ps}, we see the gaps are all within 2\%, which means that the choice on the $\lambda$  does not significantly influence the performance. We opt for $\lambda=0.5$ as this value provides the smallest gap.

\begin{table}[tb!]
\caption{Average percentage gap $\Gamma_{p, \lambda}$ from the best policy found for  $\lambda \in \{0.4,0.5,0.6\}$ and $p\ \text{is VFA-CF}$.}
\label{table:ps}
\centering
\scalebox{0.7}{\color{black}
\begin{tabular}{l|ccc|ccc|cccc|c}
\toprule
$\lambda$ & $\delta=0$ & $\delta=0.5$ & $\delta=1$ & $\beta=0.5$ & $\beta=1$ & $\beta=1.5$ & $c=0.6$ & $c=0.8$ & $c=1$  & $c=1.2$ & \textbf{AVG}  \\\hline\hline
0.4    & 0.32                    & 0.05                      & 0.58                    & 0.66                     & 0.18                   & 0.11                     & 0.21  & 0.33  & 0.66 & 0.06  & \textbf{0.32} \\
0.5    & 0.03                    & 0.23                      & 0.58                    & 0.24                     & 0.13                   & 0.47                     & 0.41  & 0.09  & 0.43 & 0.19  & \textbf{0.28} \\
0.6    & 0.49                    & 0.74                      & 0.48                    & 0.42                     & 0.42                   & 0.87                     & 1.25  & 0.38  & 0.41 & 0.23  & \textbf{0.57}
\\ \hline
\end{tabular}}
\end{table}

\section{Graphical Representation of Performance of the Approaches on $\beta$ and $\delta$}
\label{sec:4approaches_beta_delta}
\begin{figure*}
    \begin{subfigure}[b]{0.5\columnwidth}
        \includegraphics[width=\textwidth]{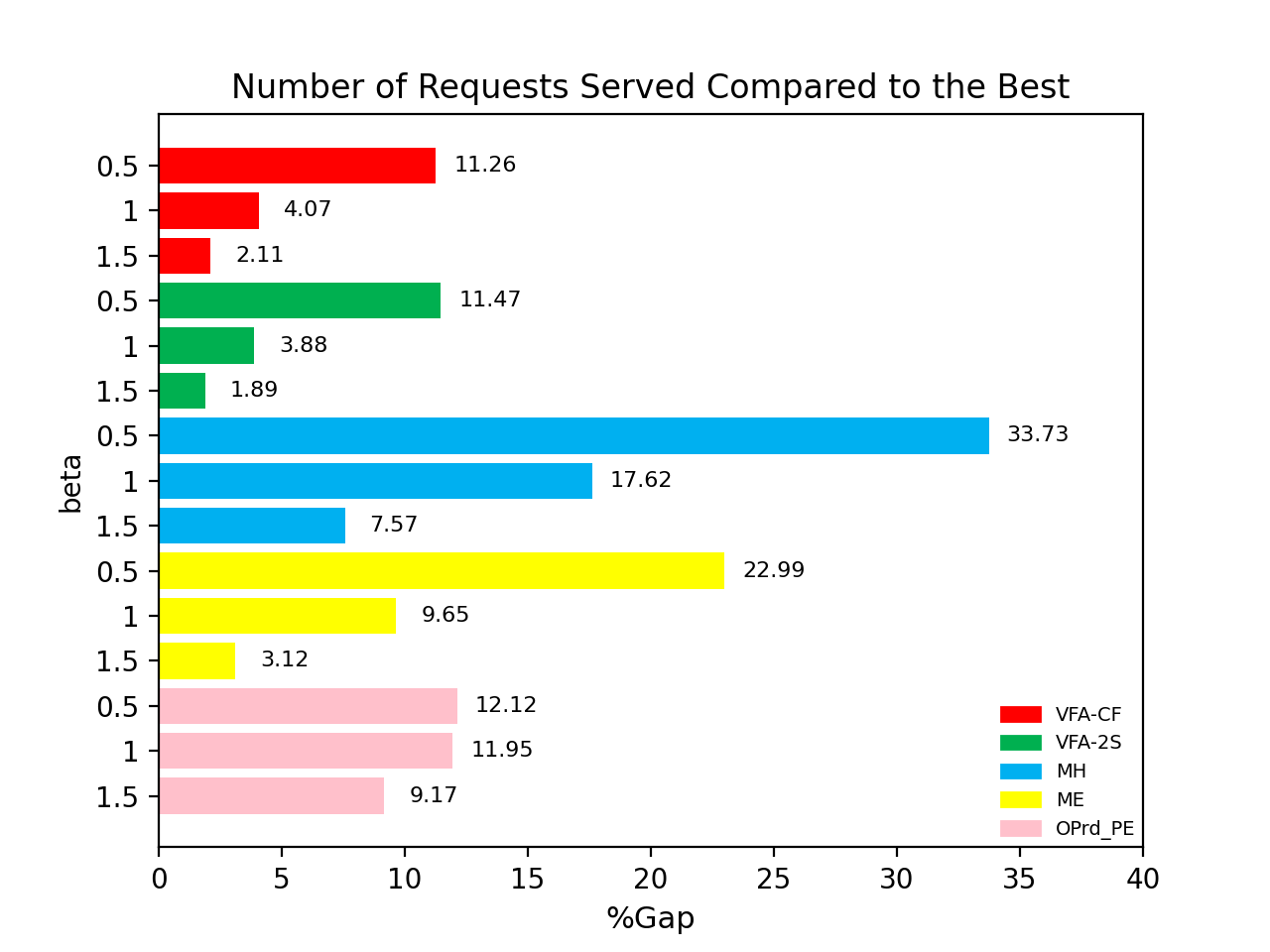}
    \end{subfigure}
    \hfill
    \begin{subfigure}[b]{0.5\columnwidth}
        \includegraphics[width=\textwidth]{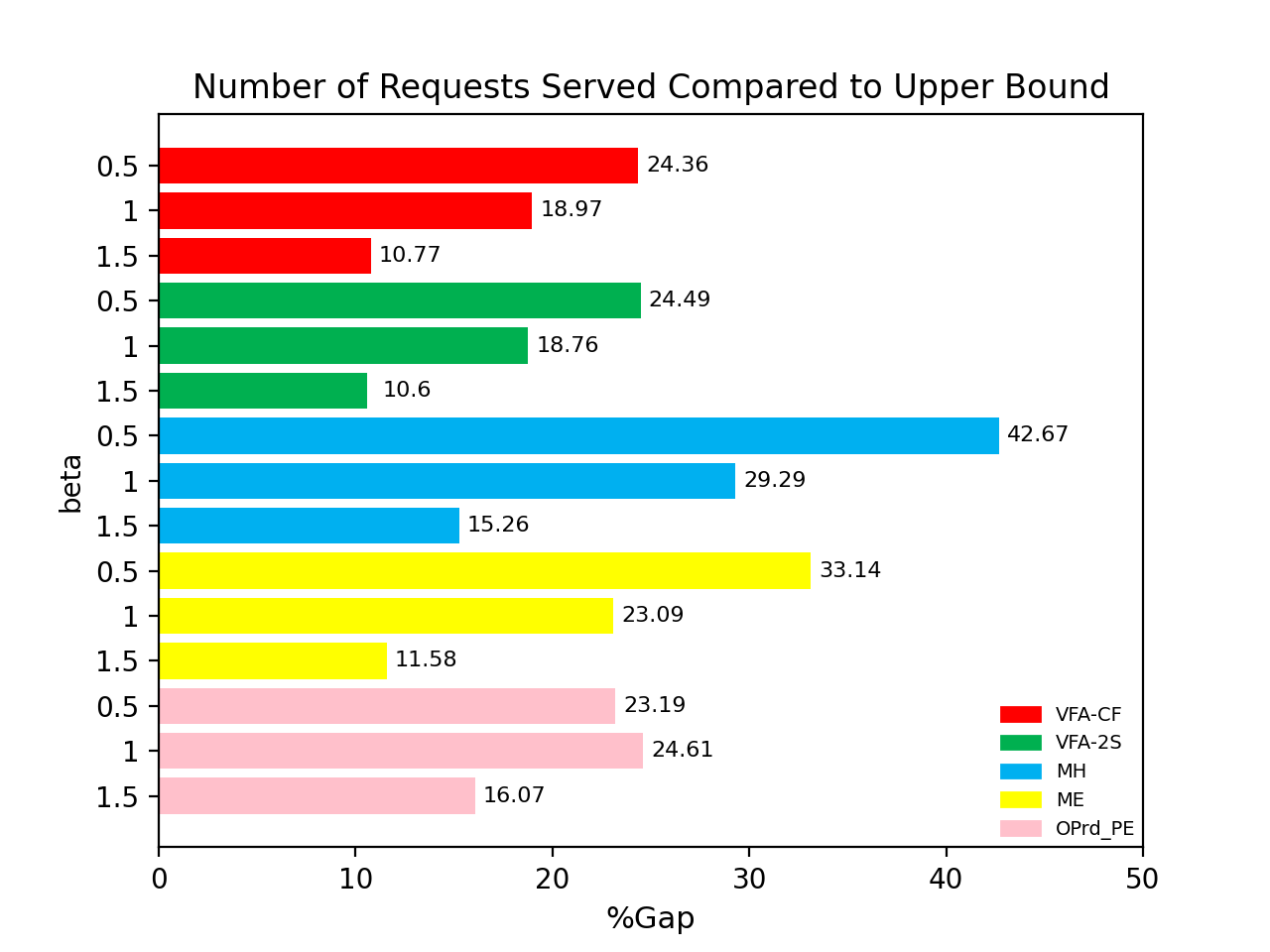}
    \end{subfigure}

    \begin{subfigure}[b]{0.5\columnwidth}
        \includegraphics[width=\textwidth]{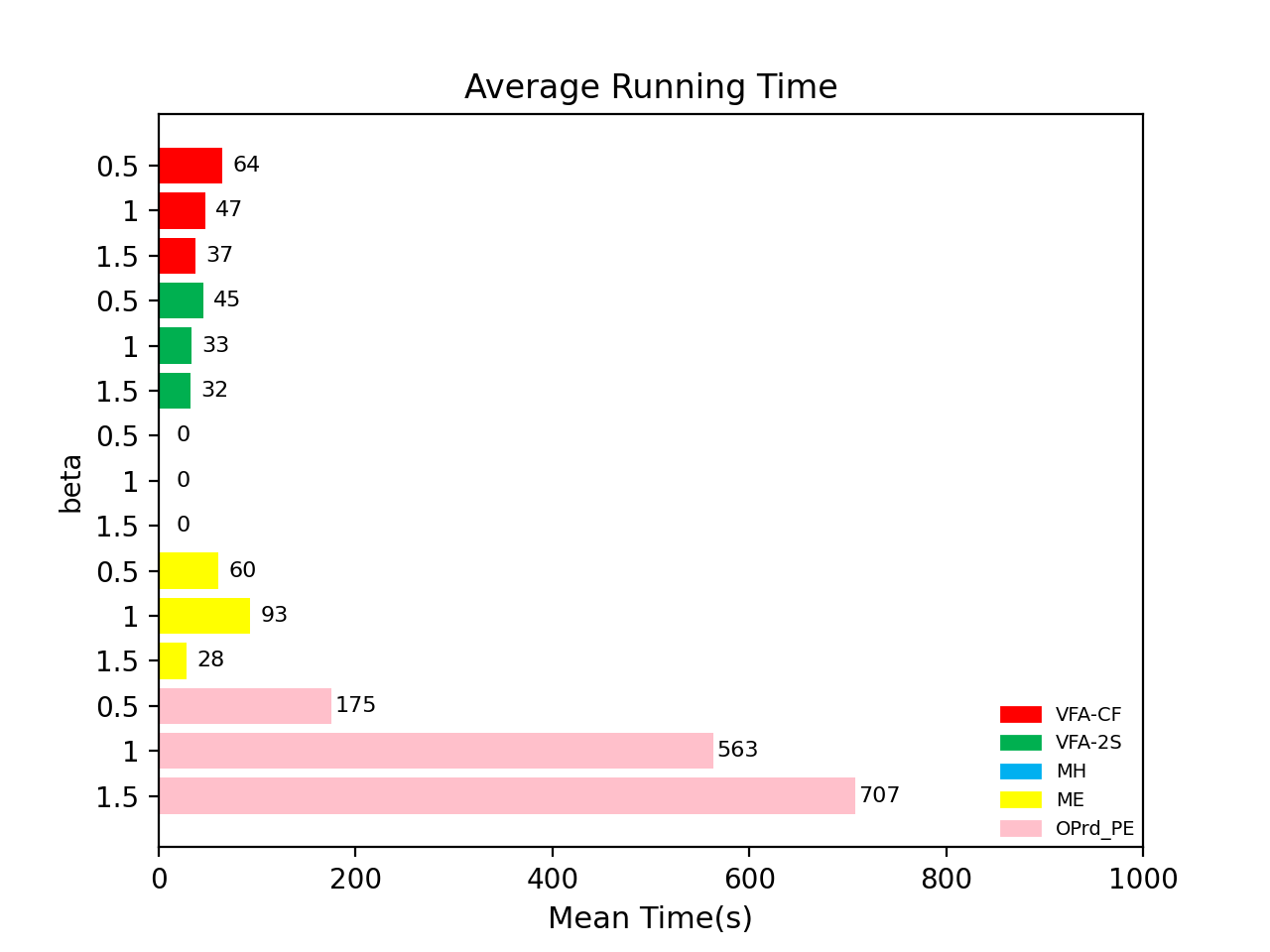}
    \end{subfigure}
    \hfill
    \begin{subfigure}[b]{0.5\columnwidth}
        \includegraphics[width=\textwidth]{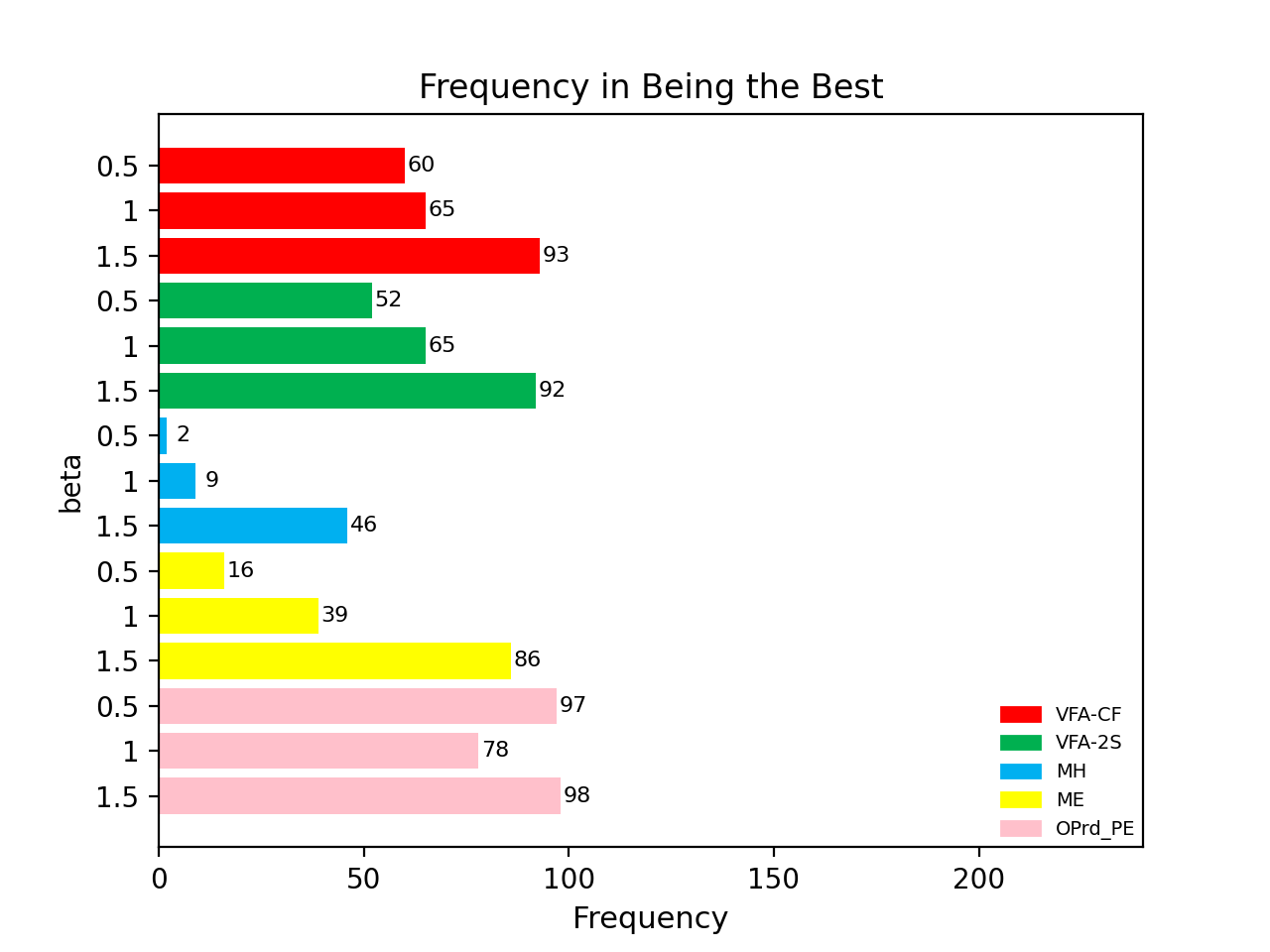}
    \end{subfigure}
    \caption{Performance of the five approaches based on the values of $\beta$}
    \label{fig:4approach_beta2}
\end{figure*}

\begin{figure*}
    \begin{subfigure}[b]{0.49\columnwidth}
        \includegraphics[width=\textwidth]{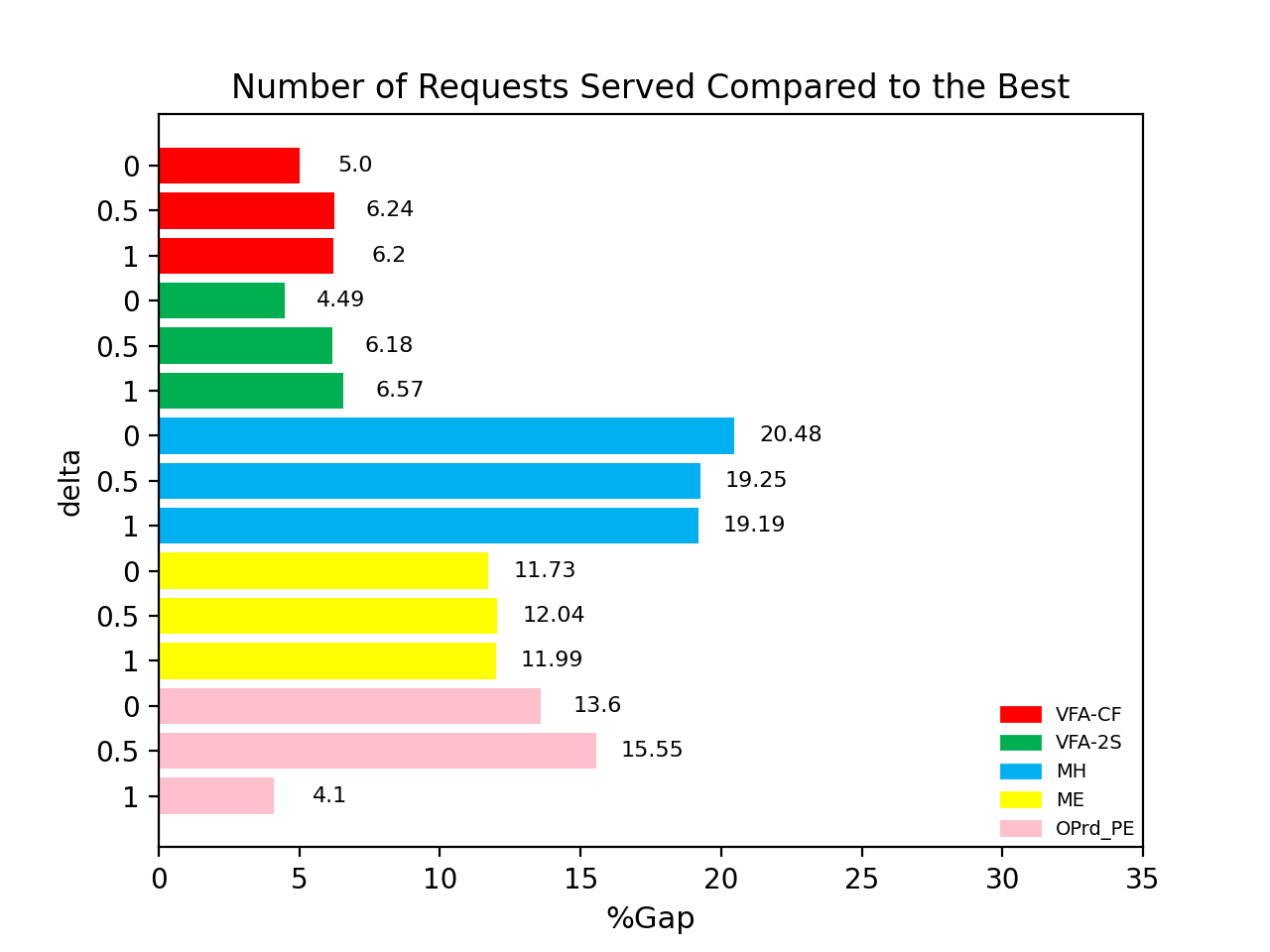}
    \end{subfigure}
    \hfill
    \begin{subfigure}[b]{0.49\columnwidth}
        \includegraphics[width=\textwidth]{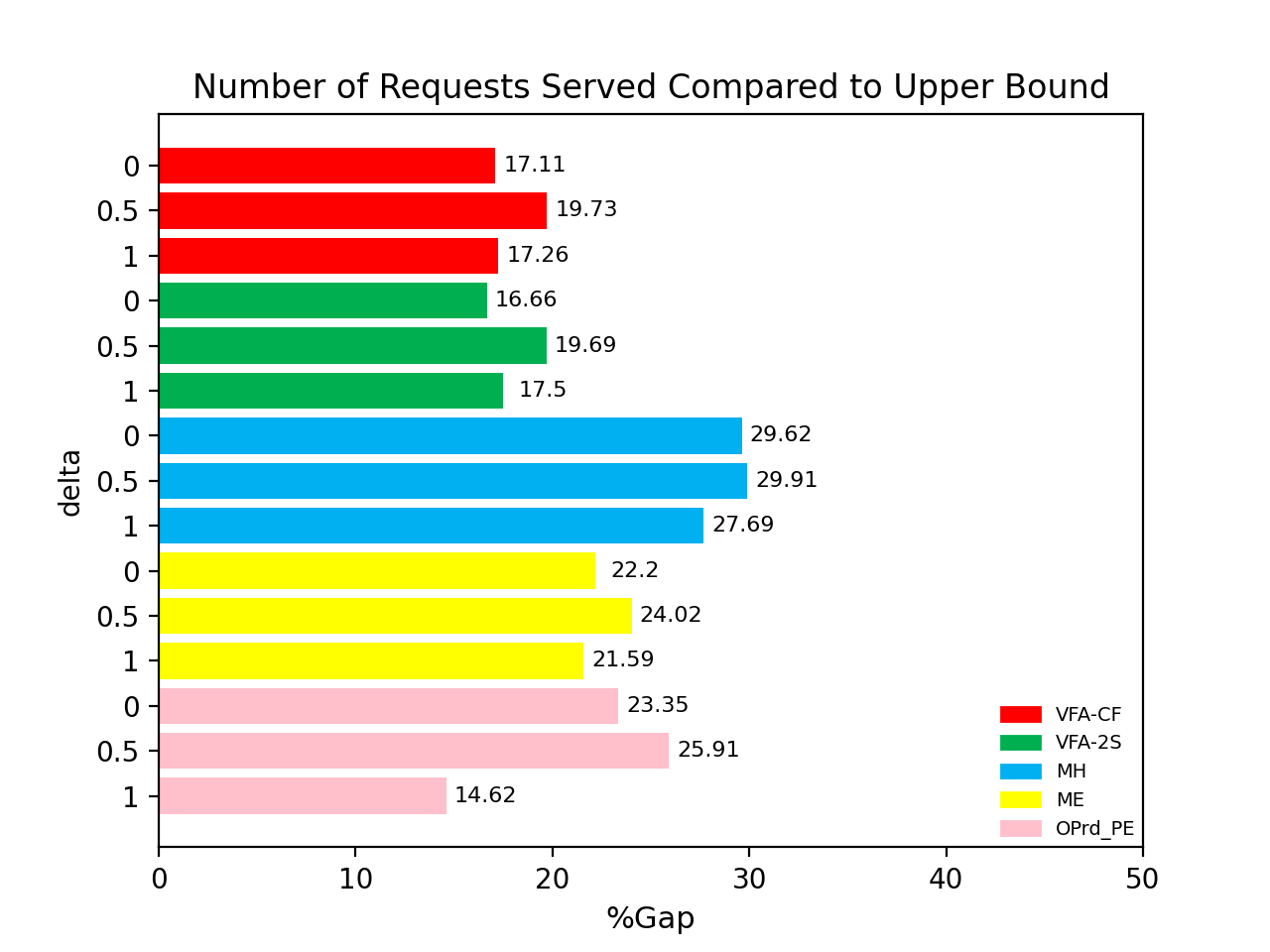}
    \end{subfigure}

    \begin{subfigure}[b]{0.49\columnwidth}
        \includegraphics[ width=\textwidth]{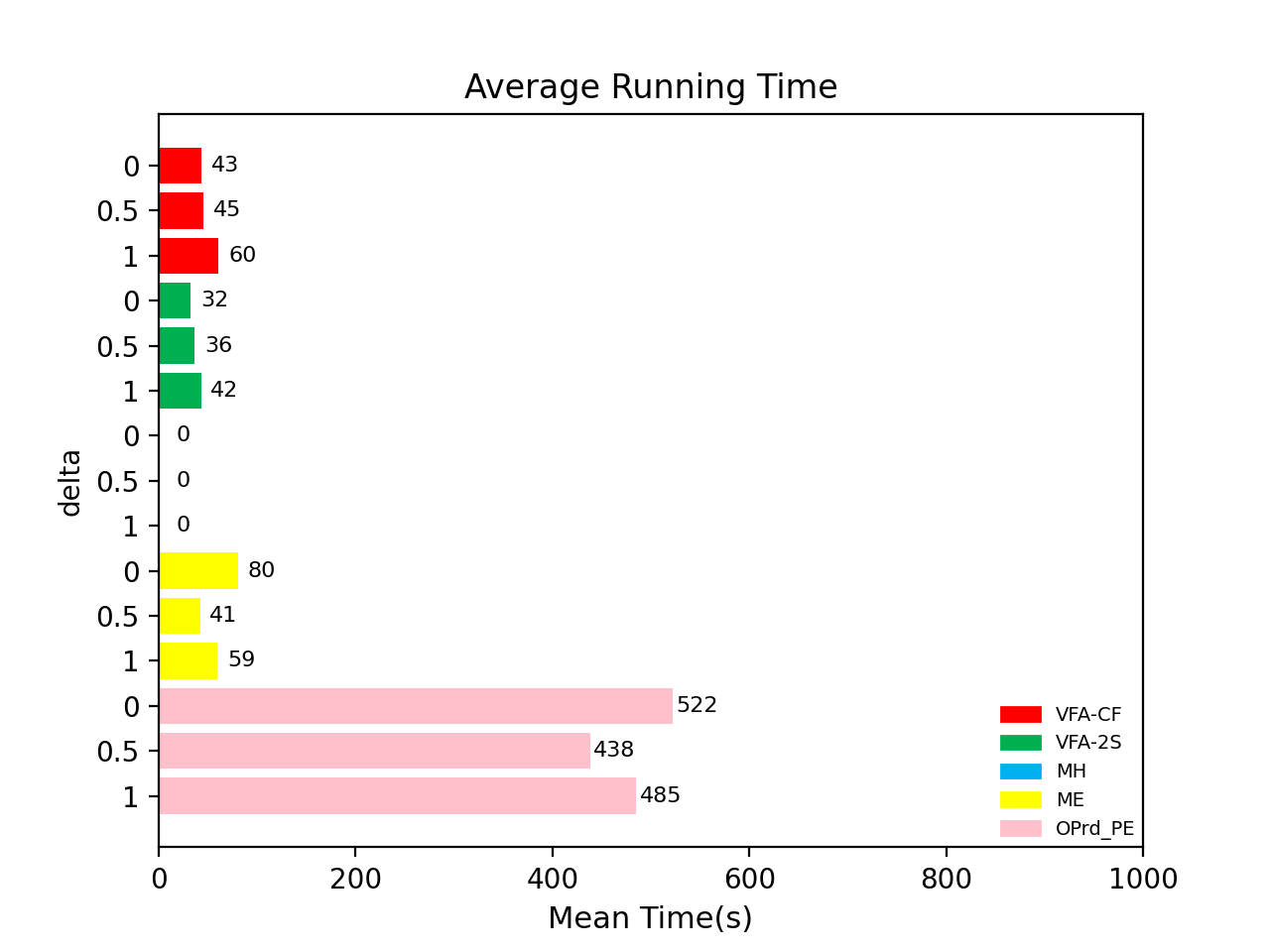}
    \end{subfigure}
    \hfill
    \begin{subfigure}[b]{0.49\columnwidth}
        \includegraphics[width=\textwidth]{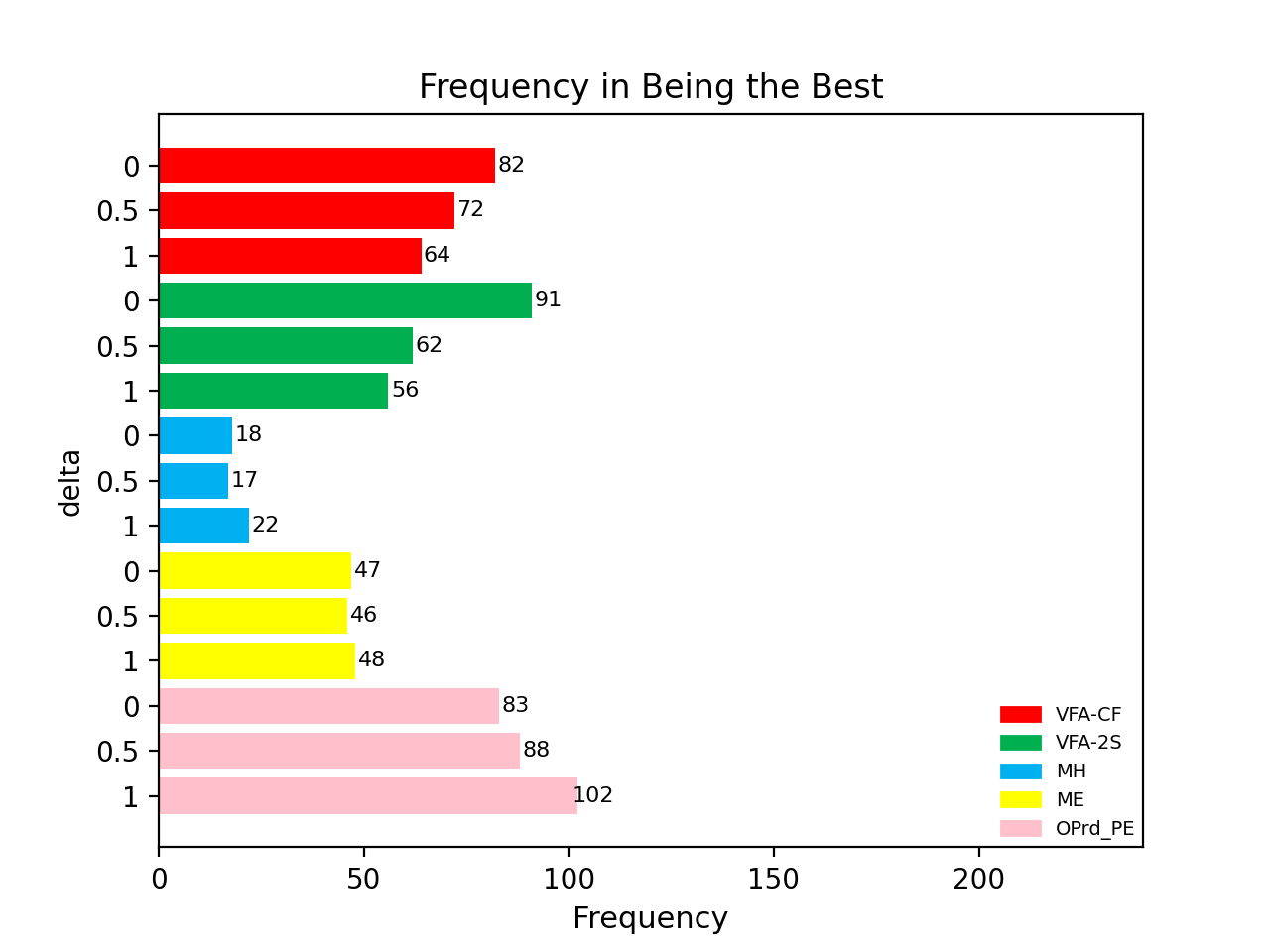}
    \end{subfigure}
    \caption{Performance of the five approaches based on the values of $\delta$}
    \label{fig:4approach_delta2}
\end{figure*}
In Figure \ref{fig:4approach_beta2}, we observe that when $\beta$ goes from 0.5 to 1.5, for both gap[\%] and gap$_{ub}$[\%], all approaches improve progressively, especially for ME and MH. 
Regarding the frequency in getting the best policy, we find that VFA-CF and VFA-2S achieve a significantly better performance when the spread of release dates is large (i.e., $\beta = 1.5$). As for OPrd\_PE, we see it is the only one that can compete with VFA-CF and VFA-2S for $\beta$ = 0.5, but its results are not stable, and the running time is almost ten times bigger. 
Besides, in Figure \ref{fig:4approach_delta2}, we see that when $\delta$ is 0.5, values of gap$_{ub}$[\%] are the largest for most approaches, which means it is challenging to solve instances containing customers with mixed dynamism degrees.

\section{Performance of the Approaches on Time Scale}
\label{sec:4approaches_timescale}
In this section, we compare the effect of the different values of the deadline $T_E$ in Table \ref{table:4approaches_TE} and Figure \ref{fig:4approaches_ts2}. Remember that $T_E$ is determined as the latest release date in the instance ($T_{standard}$) multiplied by parameter $c$.

\begin{figure*}
    \begin{subfigure}[b]{0.52\columnwidth}
        \includegraphics[width=\textwidth]{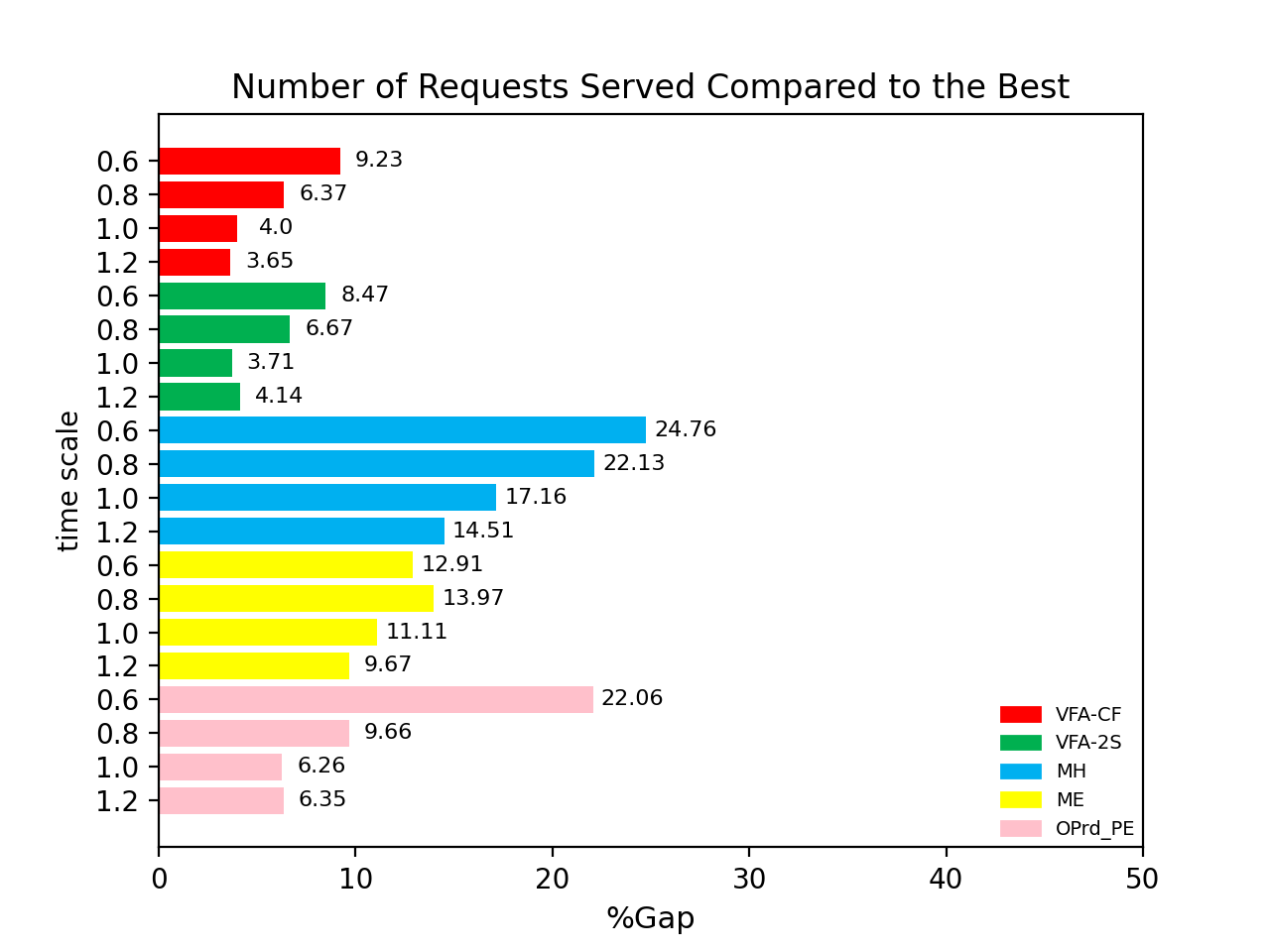}
    \end{subfigure}
    \hfill
    \begin{subfigure}[b]{0.52\columnwidth}
        \includegraphics[width=\textwidth]{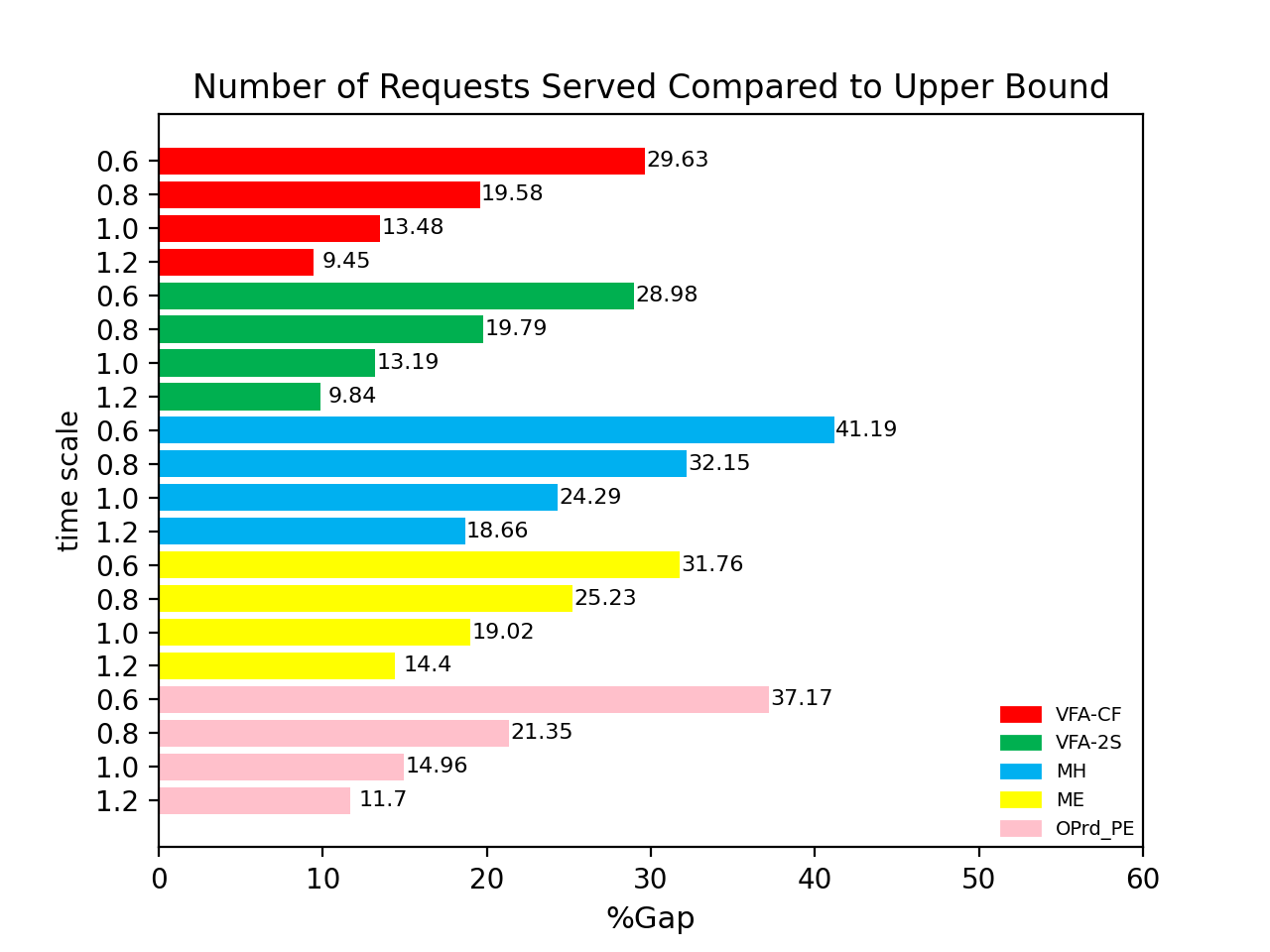}
    \end{subfigure}

    \begin{subfigure}[b]{0.52\columnwidth}
        \includegraphics[width=\textwidth]{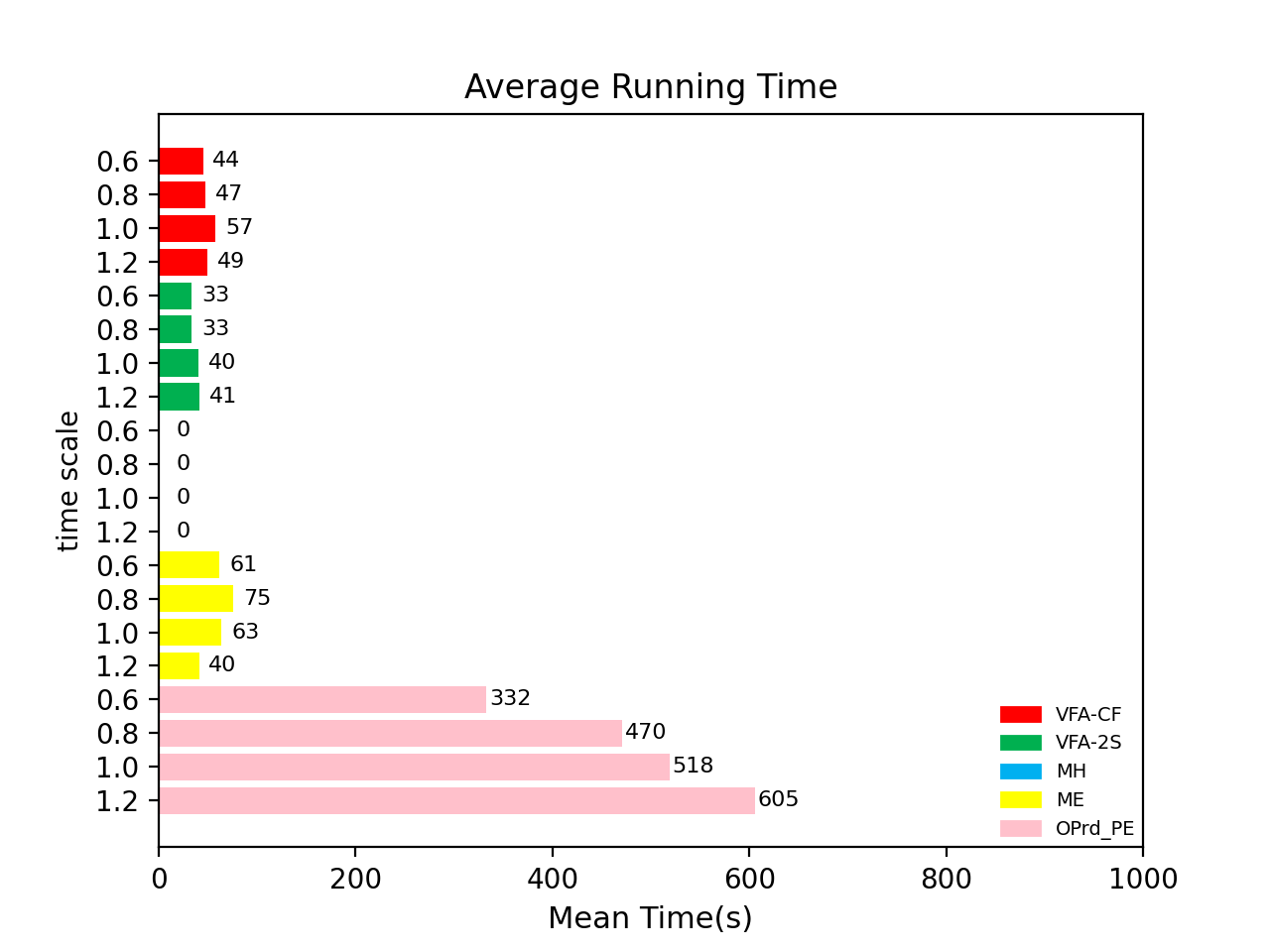}
    \end{subfigure}
    \hfill
    \begin{subfigure}[b]{0.52\columnwidth}
        \includegraphics[width=\textwidth]{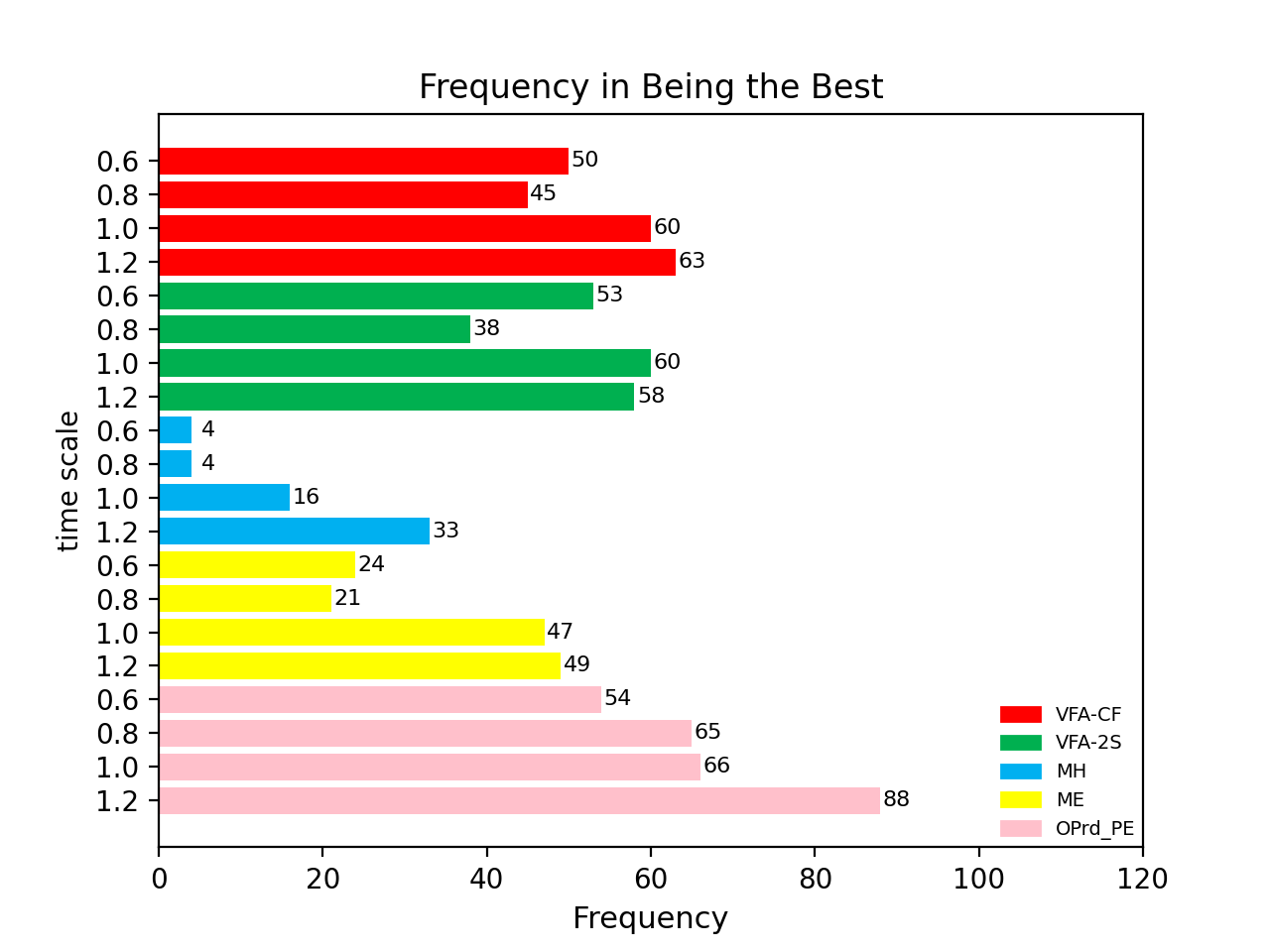}
    \end{subfigure}
    \caption{Performance of the five approaches based on the value of $c$}
    \label{fig:4approaches_ts2}
\end{figure*}

\begin{table}[hpt]
\caption{The average performance of the five approaches  on the basis of $c$}
\label{table:4approaches_TE}
\begin{subtable}{\linewidth}\centering
\scalebox{0.6}{\color{black}
\begin{tabular}{l|l|cccc|cccc}
\toprule
\textbf{}    & \textbf{}                                & \multicolumn{4}{c}{\textbf{VFA-CF}}                              & \multicolumn{4}{c}{\textbf{VFA-2S}}                              \\\hline
${c}$   & \multicolumn{1}{c|}{\textbf{\#instances}} & gap{[}\%{]}   & gap\_ub{[}\%{]} & t{[}s{]}       & freq          & gap{[}\%{]}   & gap\_ub{[}\%{]} & t{[}s{]}       & freq          \\\hline
0.6          & 108                                      & 9.23          & 29.63           & 44.88          & 50            & 8.47          & 28.98           & 33.04          & 53            \\
0.8          & 108                                      & 6.37          & 19.58           & 47.34          & 45            & 6.67          & 19.79           & 33.18          & 38            \\
1            & 108                                      & 4.00          & 13.48           & 57.73          & 60            & 3.71          & 13.19           & 40.50          & 60            \\
1.2          & 108                                      & 3.65          & 9.45            & 49.36          & 63            & 4.14          & 9.84            & 41.45          & 58            \\\hline
\textbf{AVG} & \textbf{108}                             & \textbf{5.81} & \textbf{18.03}  & \textbf{49.83} & \textbf{54.5} & \textbf{5.75} & \textbf{17.95}  & \textbf{37.04} & \textbf{52.3}\\ \hline
\end{tabular}
}\end{subtable}
\newline
\vspace*{0.01\linewidth}
\newline
\begin{subtable}{\linewidth}\centering
\scalebox{0.6}{\color{black}
\begin{tabular}{l|l|cccc|cccc|cccc}
\hline
\textbf{}    & \textbf{}                                & \multicolumn{4}{c}{\textbf{MH}}                                  & \multicolumn{4}{c}{\textbf{ME}}                                   & \multicolumn{4}{c}{\textbf{OPrd\_PE}}                             \\\hline
$c$   & \multicolumn{1}{c|}{\textbf{\#instances}} & gap{[}\%{]}    & gap\_ub{[}\%{]} & t{[}s{]}      & freq          & gap{[}\%{]}    & gap\_ub{[}\%{]} & t{[}s{]}       & freq          & gap{[}\%{]}    & gap\_ub{[}\%{]} & t{[}s{]}        & freq          \\\hline
0.6          & 108                                      & 24.76          & 41.19           & 0.25          & 4             & 12.91          & 31.76           & 61.55          & 24            & 22.06          & 37.17           & 332.87          & 54            \\
0.8          & 108                                      & 22.13          & 32.15           & 0.21          & 4             & 13.97          & 25.23           & 75.94          & 21            & 9.66           & 21.35           & 470.93          & 65            \\
1            & 108                                      & 17.16          & 24.29           & 0.21          & 16            & 11.11          & 19.02           & 63.89          & 47            & 6.26           & 14.96           & 518.83          & 66            \\
1.2          & 108                                      & 14.51          & 18.66           & 0.21          & 33            & 9.67           & 14.40           & 40.93          & 49            & 6.35           & 11.70           & 605.68          & 88            \\\hline
\textbf{AVG} & \textbf{108}                             & \textbf{19.64} & \textbf{29.07}  & \textbf{0.22} & \textbf{14.3} & \textbf{11.92} & \textbf{22.60}  & \textbf{60.58} & \textbf{35.3} & \textbf{11.08} & \textbf{21.29}  & \textbf{482.08} & \textbf{68.3}\\\hline
\end{tabular}}\end{subtable}
\end{table}

In Table \ref{table:4approaches_TE}, we see that, on average, the performance of the five approaches shows an improving trend as the value of $T_E$ increases. For each $T_E$, we observe that VFA-CF gets the smallest gaps when $c$ is 0.8 and 1.2, and for VFA-2S $c$ is 0.6 and 1. 
On the other hand, we see that when $T_E=1.2\cdot T_{standard}$, the values of $freq$ are the biggest for most approaches, reflecting the fact that these instances are easier to solve. 
Finally, we observe that though $MH$ and $ME$ provide results in a short time, the policy quality is sacrificed. OPrd\_PE provides the best results among benchmark approaches, especially on the frequency of being the best. However, the average gaps are bigger than VFA-CF and VFA-2S, and the running time is high.  

\section{Illustrative Example}
\label{sec:route_example_plot}

An example is now shown to illustrate the differences among the solutions provided by the different policies. We take an instance with $\beta = 1.5$, $\delta = 0$, and $c = 0.6$. We plot the routes returned by each policy. The routes are colored in order: red - first route, black - second route, green - third route, blue - fourth route, magenta - fifth route, and orange - sixth route. The number next to each node is its release date (RD).

The points are plotted according to their coordinates. The left subfigure in Figure \ref{fig:routes_VFA-CF} refers to VFA-CF and shows five routes starting consecutively with a waiting time of 17 before starting the last route. In the right subfigure that refers to VFA-2S in Figure \ref{fig:routes_VFA-CF}, the black route illustrates that VFA-2S chooses to serve a customer with an RD value equal to 28, at the top left, instead of another customer with an RD value of 39, in the center bottom, as seen in VFA-CF. Then VFA-2S starts the green route by serving customers located in the bottom left rather than those in the top left. The following two routes are similar between VFA-2S and VFA-CF. The left and right subfigures in Figure \ref{fig:routes_MH} referring to MH and ME respectively, show that the requests served by the two myopic policies are widely dispersed (note the green route in MH visits from the node with RD valued 169). This can be attributed to the fact that these policies prioritize serving as many available requests as possible without considering future consolidation opportunities. In contrast, the other policies tend to concentrate the served requests within a smaller geographical region in each route. This example serves to explain the longer average traveling time observed for the MH and ME policies, as illustrated in Table 8. Additionally, we note that traversing the same areas multiple times is inevitable when dealing with uncertain release dates. Furthermore, in Figure \ref{fig:routes_OPrd_PE} depicting the results of OPrd\_PE, we observe that despite using more routes compared to other approaches, it does not yield superior performance compared to VFA methods.
\begin{figure}[h]
    \centering    \includegraphics[width=0.9\textwidth]{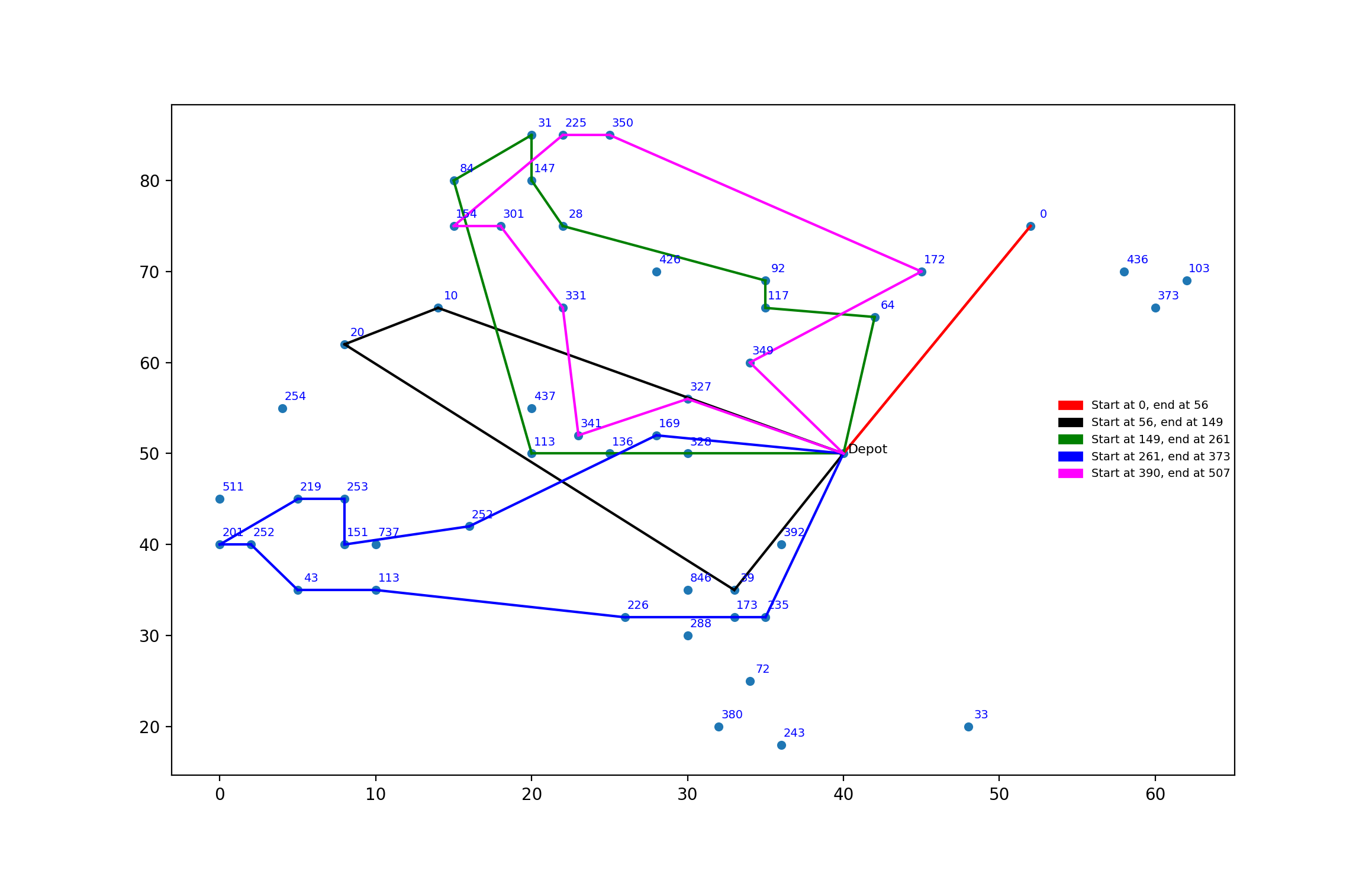}
    \includegraphics[width=0.9\textwidth]{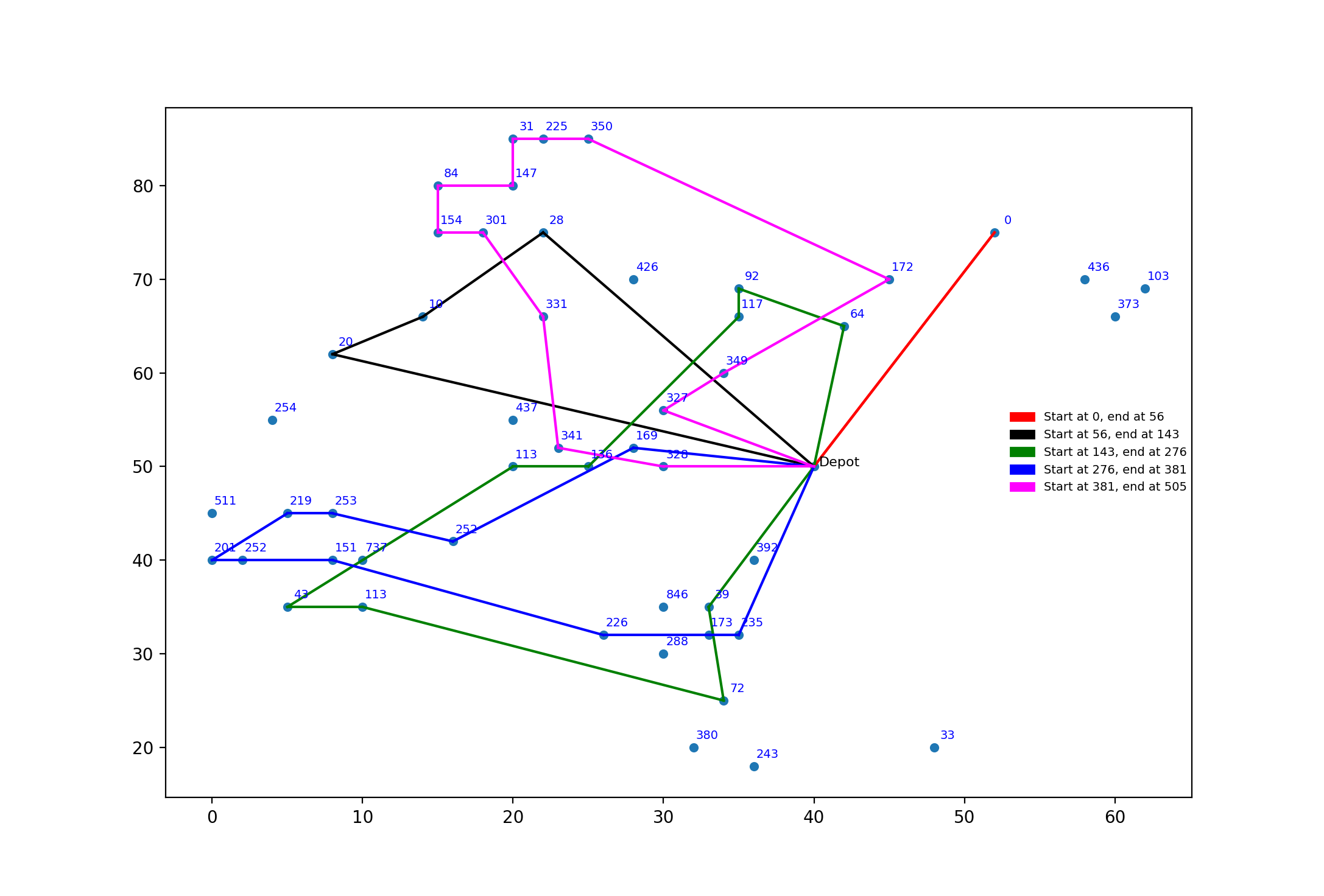}
    \caption{Up: VFA-CF returns five routes serving 34 requests, starting at time 0, 56, 149, 261, and 390. Bottom: VFA-2S returns five routes serving 36 requests, starting at time 0, 56, 143, 276, and 381.}
    \label{fig:routes_VFA-CF}
\end{figure}
\begin{figure}
    \centering    \includegraphics[width=0.9\textwidth]{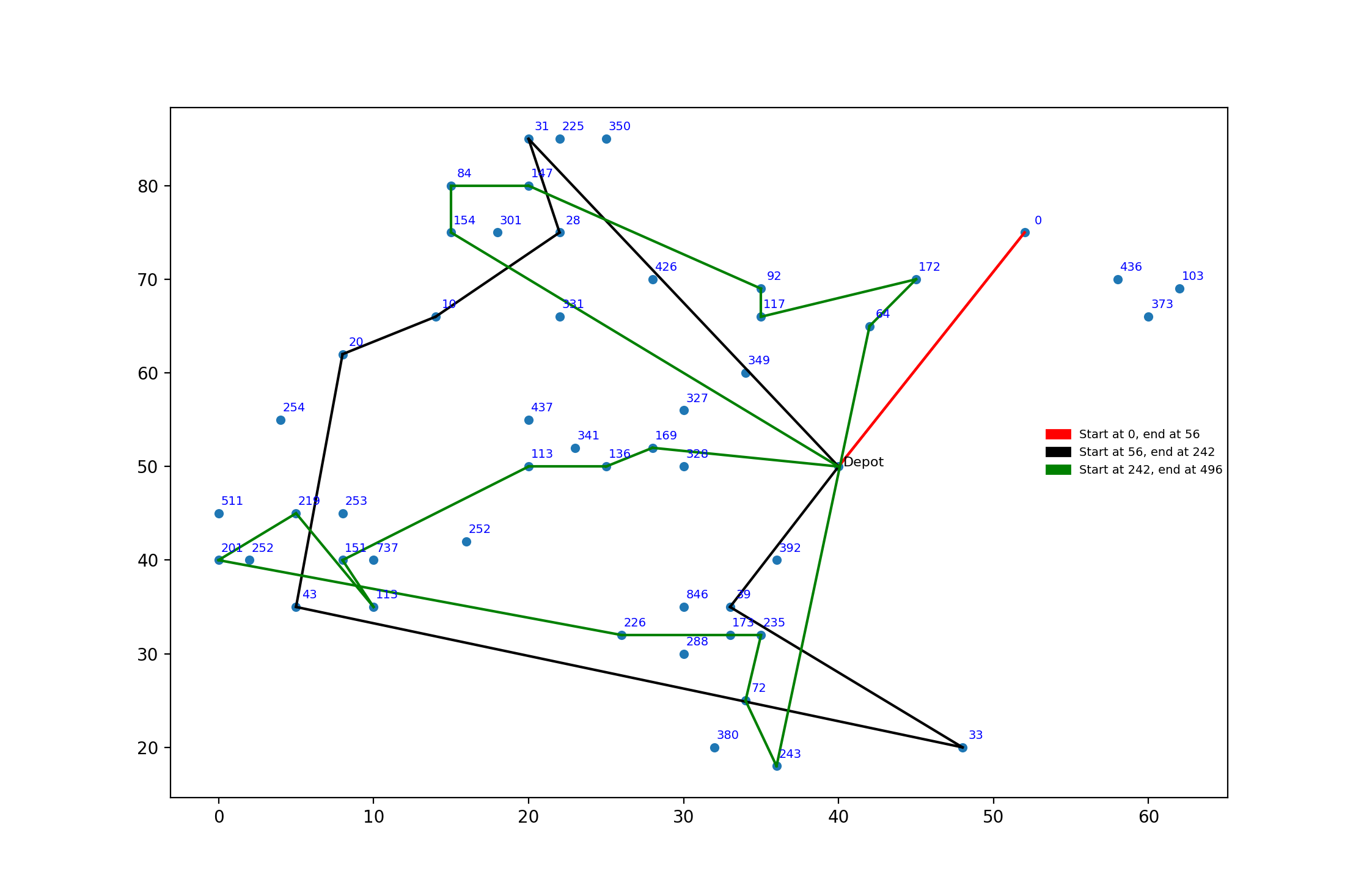}
    \includegraphics[width=0.9\textwidth]{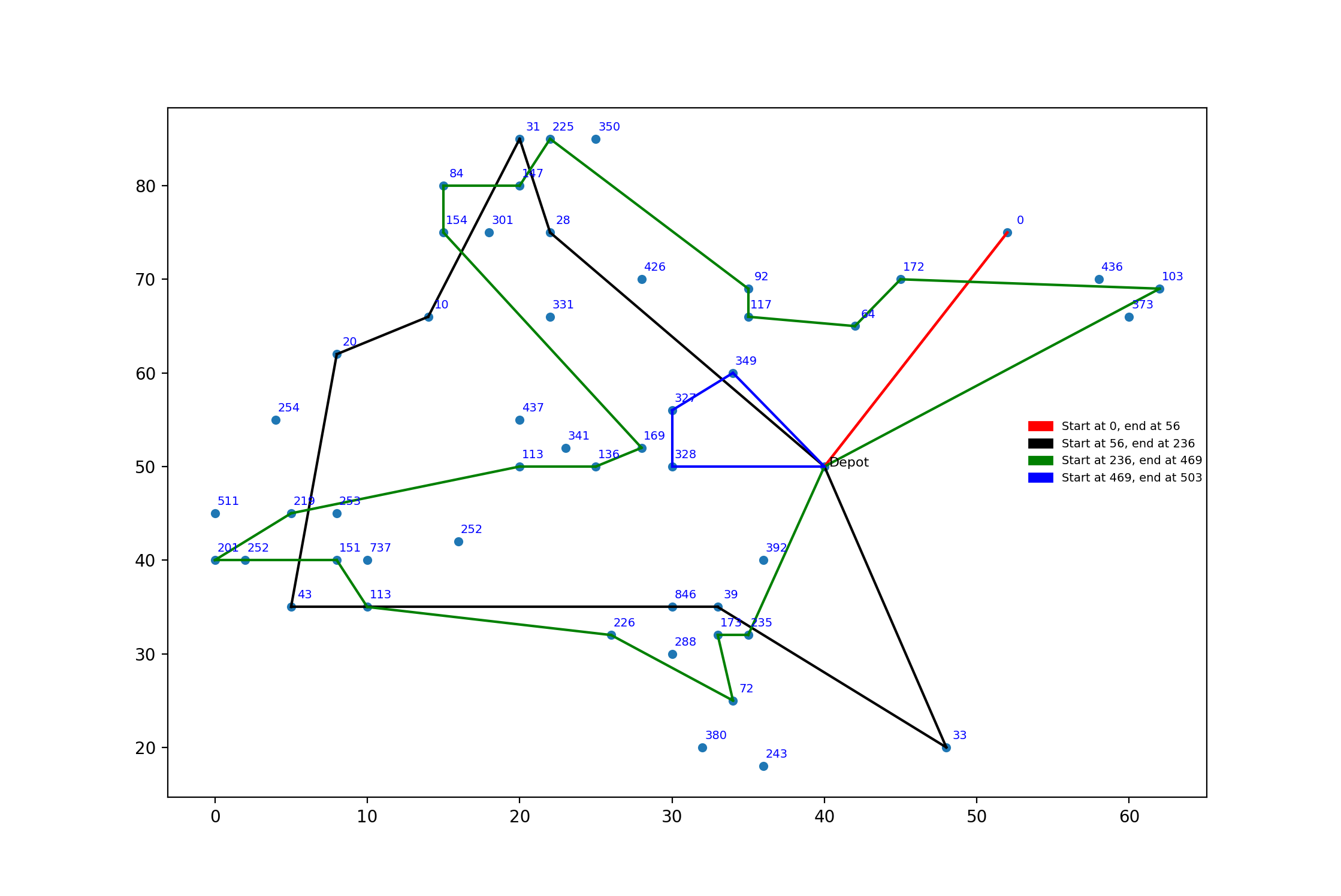}
    \caption{Up: MH returns three routes serving 27 requests, starting at time 0, 56, and 242. Bottom: ME returns four routes serving 31 requests, starting at time 0, 56, 236, and 469.}
    \label{fig:routes_MH}
\end{figure}
\begin{figure}
    \centering
    \includegraphics[width=0.9\textwidth]{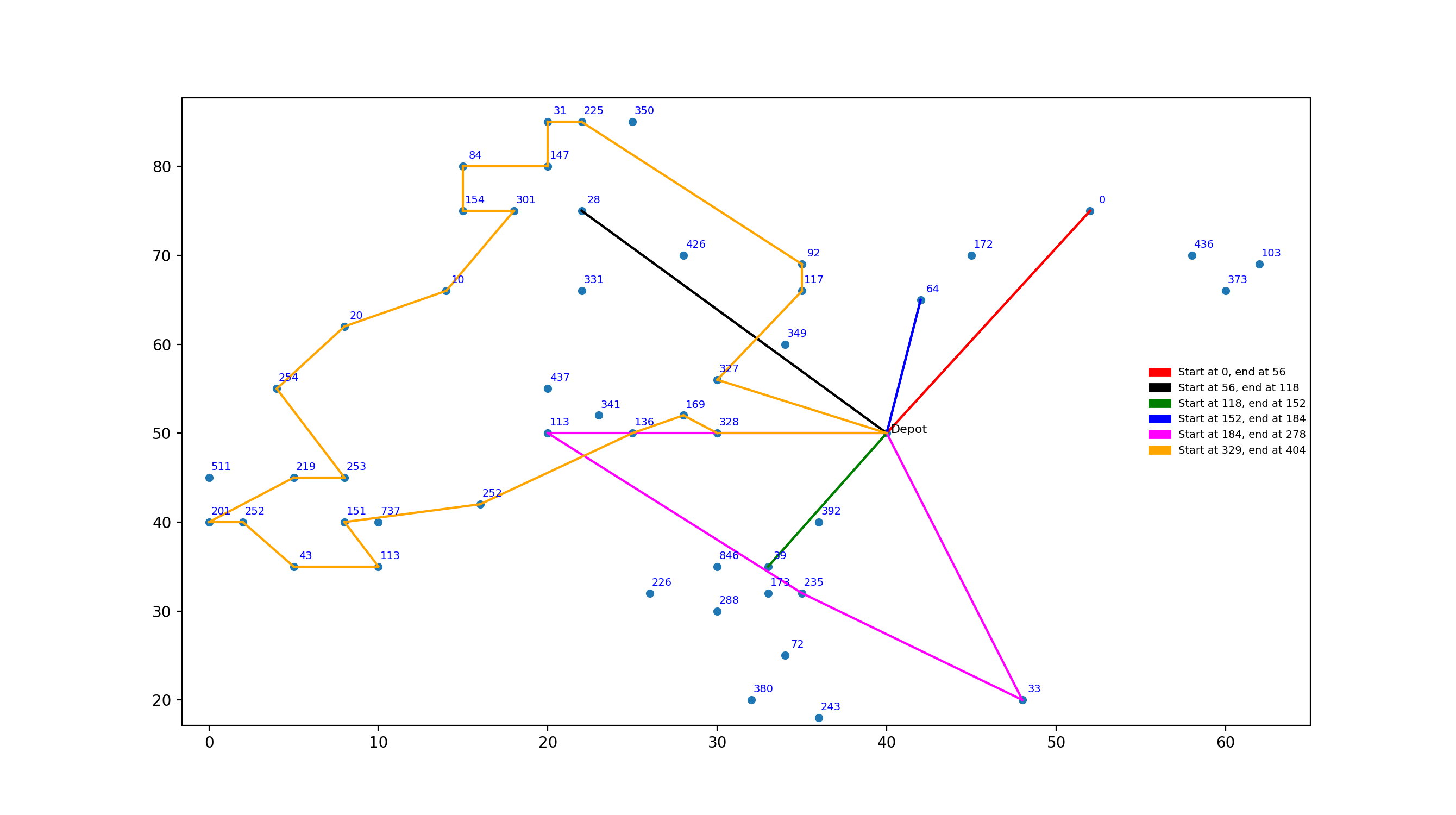}
    \caption{OPrd\_PE returns six routes serving 31 requests, starting at time 0, 56, 118, 152, 184, and 329}
    \label{fig:routes_OPrd_PE}
\end{figure}

\section{Comparison with a Deterministic and Static Approach}
\label{sec:deter}
To show the importance of considering both stochastic and dynamic release dates, we compare VFA approaches with an alternative approach, called OPrd\_Sto, that uses only the stochastic information available at time zero.
Similarly to the OPrd\_PE, OPrd\_Sto is based on the OP-rd formulation, using expected release dates at time 0 as point estimations. It performs exact evaluations of future rewards, but with a distinction: the soluton obtained at time 0 is kept and never modified afterwards. Specifically, for each route of this deterministic solution, we wait until all its scheduled requests have arrived, before sending out the vehicle. 
The route is executed only in case it does not exceed the deadline.
We test OPrd\_Sto on the same 432 test instances, each containing 50 customers, as for the VFA approaches. For each instance, we set a time limit of 30 minutes. In Table \ref{table:sto_RD}, we summarize the results over $\delta$, $\beta$, and $c$. The  percentage gap (Gap\%) is calculated as follows: for each instance, the gap with VFA-CF is calculated by 100 minus the ratio of requests served by OPrd\_Sto to those served by VFA-CF, multiplied by 100. The average  percentage gap over all instances is then computed for each parameter ($\delta$, $\beta$, or $c$). The same calculation is used for the comparison with VFA-2S. 

Compared to VFA approaches, OPrd\_Sto needs 30 times more computing time, on average. This increase in computational effort is due to the exact route calculations. Despite that, the results reveal that OPrd\_Sto serves $28\%$ fewer requests compared to VFA approaches on average, indicating a significant performance gap. OPrd\_Sto considers stochastic information only at time 0 and produces a deterministic solution. In contrast, VFA approaches account for both stochastic and dynamic aspects, which emphasizes the value of including both types of information in the decision-making process.

\section{Detailed Computational Results}
\label{detail}

\begin{table}[tb!]
\caption{Comparing VFA approaches with an approach that only considers stochastic release dates}
\label{table:sto_RD}
\centering
\scalebox{0.6}{\color{black}
\begin{tabular}{l|c||ccc||cc}
\toprule
                   & \textbf{}           & \multicolumn{3}{c}{\textbf{t{[}s{]}}}                  & \multicolumn{2}{c}{\textbf{Gap\%}} \\\hline
                   & \textbf{\#instances} & \textbf{OPrd\_Sto} & \textbf{VFA-CF} & \textbf{VFA-2S} & \textbf{VFA-CF}  & \textbf{VFA-2S} \\\hline\hline
$\delta$=0   & 144                 & 1568.19            & 43.70           & 32.08           & 16.08            & 16.60           \\
$\delta$=0.5 & 144                 & 1568.16            & 45.34           & 36.20           & 28.41            & 28.41           \\
$\delta$=1   & 144                 & 1640.35            & 60.45           & 42.85           & 40.41            & 40.22           \\\hline
$\beta$=0.5  & 144                 & 1341.71            & 64.90           & 45.50           & 20.32            & 20.31           \\
$\beta$=1    & 144                 & 1790.45            & 47.54           & 33.02           & 33.46            & 33.58           \\
$\beta$=1.5  & 144                 & 1644.53            & 37.04           & 32.61           & 31.13            & 31.34           \\\hline
$c$=0.6     & 108                 & 1422.19            & 44.88           & 33.04           & 19.55            & 20.41           \\
$c$=0.8     & 108                 & 1650.30            & 47.34           & 33.18           & 33.83            & 33.62           \\
$c$=1       & 108                 & 1618.64            & 57.73           & 40.50           & 28.56            & 28.69           \\
$c$=1.2     & 108                 & 1677.80            & 49.36           & 41.45           & 31.27            & 30.92           \\\hline
\textbf{AVG}       & \textbf{129.6}               & \textbf{1592.23}            & \textbf{49.83}           & \textbf{37.04}           & \textbf{28.30}            & \textbf{28.41}          
\\ \hline
\end{tabular}}
\end{table}

In Tables \ref{4approaches_betadelta-part 1} and \ref{4approaches_betadelta-part 2}, we display the average performance of the five approaches across various combinations of $\beta$, $\delta$, and $c$. Notably, when $\beta$ is equal to 1.5, $\delta$ is equal to 0, and $c$ is equal to 1.2, all approaches get optimal policies, which means this setting is easy to solve. On the other hand, with $\beta= 0.5$, $\delta=1$, and $c=0.6$, VFA approaches get the largest gaps, while OPrd\_PE performs well. In general, both VFA and myopic approaches perform better as $\beta$ increases, and all approaches face challenges in obtaining good policies with smaller values of $c$ in most cases. Additionally, it is noteworthy that VFA approaches exhibit more stable performance in gaps than other approaches. Moreover, the running time of VFA approaches and MH remains stable across settings, whereas ME and OPrd\_PE show more diverse runtime.
\begin{table}[tb!]
\caption{Average performance of the five approaches - part 1 (50 customers)}
\label{4approaches_betadelta-part 1}
\centering
\scalebox{0.6}{\color{black}
\begin{tabular}{l|l|l|l|cccc|cccc}
\toprule
\multicolumn{1}{l}{}     &       &     &             & \multicolumn{4}{c}{VFA-CF}                      & \multicolumn{4}{c}{VFA-2S}                      \\\hline
\multicolumn{1}{l}{$\beta$} & $\delta$ & $c$   & \#instances & gap{[}\%{]} & gap\_ub{[}\%{]} & t{[}s{]} & freq & gap{[}\%{]} & gap\_ub{[}\%{]} & t{[}s{]} & freq \\\hline\hline
\multirow{5}{*}{0.5}    & 0     & 0.6 & 12          & 6.67        & 20.24           & 38.71    & 9    & 4.44        & 18.16           & 23.49    & 9    \\
                         & 0     & 0.8 & 12          & 13.90        & 25.96           & 49.19    & 6    & 13.9        & 25.96           & 22.92    & 6    \\
                         & 0     & 1   & 12          & 9.18        & 26.45           & 36.79    & 6    & 9.18        & 26.45           & 27.33    & 6    \\
                         & 0     & 1.2 & 12          & 10.25       & 24.30            & 39.43    & 4    & 10.29       & 24.32           & 46.00       & 4    \\
                         &       & AVG & 12          & 10.00          & 24.24           & 41.03    & 6.3  & 9.45        & 23.72           & 29.93    & 6.3  \\\hline
               \multirow{5}{*}{0.5}          & 0.5   & 0.6 & 12          & 22.25       & 28.45           & 26.72    & 4    & 22.25       & 28.45           & 20.31    & 4    \\
                         & 0.5   & 0.8 & 12          & 8.81        & 22.39           & 40.90     & 5    & 10.12       & 23.46           & 28.74    & 4    \\
                         & 0.5   & 1   & 12          & 4.13        & 21.97           & 46.71    & 8    & 5.78        & 23.21           & 76.95    & 5    \\
                         & 0.5   & 1.2 & 12          & 8.53        & 20.83           & 77.94    & 3    & 10.96       & 22.91           & 43.92    & 1    \\
                         &       & AVG & 12          & 10.93       & 23.41           & 48.07    & 5.0    & 12.28       & 24.51           & 42.48    & 3.5  \\\hline
                \multirow{5}{*}{0.5}         & 1     & 0.6 & 12          & 23.69       & 31.70            & 32.92    & 3    & 22.86       & 30.87           & 25.88    & 3    \\
                         & 1     & 0.8 & 12          & 11.15       & 25.16           & 49.66    & 4    & 10.79       & 24.76           & 30.50     & 4    \\
                         & 1     & 1   & 12          & 9.19        & 24.43           & 215.04   & 3    & 7.92        & 23.43           & 84.19    & 3    \\
                         & 1     & 1.2 & 12          & 7.32        & 20.50            & 124.82   & 5    & 9.12        & 21.89           & 115.82   & 3    \\
                         &       & AVG & 12          & 12.84       & 25.45           & 105.61   & 3.8  & 12.67       & 25.24           & 64.10     & 3.3  \\\hline
                         & AVG   &     & 12          & 11.26       & 24.36           & 64.90     & 5.0    & 11.47       & 24.49           & 45.50     & 4.3  \\\hline\hline
\multirow{5}{*}{1}      & 0     & 0.6 & 12          & 6.57        & 37.98           & 62.46    & 6    & 6.27        & 37.78           & 38.64    & 7    \\
                         & 0     & 0.8 & 12          & 3.83        & 23.47           & 84.24    & 6    & 4.14        & 23.69           & 46.44    & 6    \\
                         & 0     & 1   & 12          & 2.71        & 11.98           & 37.72    & 6    & 0.58        & 9.94            & 25.83    & 10   \\
                         & 0     & 1.2 & 12          & 2.56        & 4.00               & 34.81    & 3    & 2.56        & 4.00               & 27.64    & 3    \\
                         &       & AVG & 12          & 3.92        & 19.36           & 54.81    & 5.3  & 3.39        & 18.85           & 34.64    & 6.5  \\\hline
         \multirow{5}{*}{1}                & 0.5   & 0.6 & 12          & 8.56        & 35.41           & 53.22    & 5    & 5.82        & 33.40            & 34.38    & 6    \\
                         & 0.5   & 0.8 & 12          & 2.53        & 27.44           & 50.81    & 8    & 2.90         & 27.62           & 35.36    & 6    \\
                         & 0.5   & 1   & 12          & 4.90         & 16.30            & 50.31    & 5    & 4.09        & 15.46           & 34.92    & 5    \\
                         & 0.5   & 1.2 & 12          & 2.82        & 8.67            & 38.20     & 6    & 3.12        & 8.83            & 27.42    & 5    \\
                         &       & AVG & 12          & 4.71        & 21.95           & 48.13    & 6.0    & 3.98        & 21.33           & 33.02    & 5.5  \\\hline
            \multirow{5}{*}{1}             & 1     & 0.6 & 12          & 3.52        & 29.42           & 50.48    & 7    & 5.28        & 30.66           & 38.86    & 5    \\
                         & 1     & 0.8 & 12          & 6.30         & 18.52           & 39.82    & 2    & 7.59        & 19.50            & 32.23    & 1    \\
                         & 1     & 1   & 12          & 3.80         & 10.52           & 35.97    & 3    & 3.54        & 10.35           & 27.84    & 3    \\
                         & 1     & 1.2 & 12          & 0.71        & 3.92            & 32.51    & 8    & 0.71        & 3.92            & 26.68    & 8    \\
                         &       & AVG & 12          & 3.58        & 15.60            & 39.69    & 5.0    & 4.28        & 16.11           & 31.40     & 4.3  \\\hline
                         & AVG   &     & 12          & 4.07        & 18.97           & 47.54    & 5.4  & 3.88        & 18.76           & 33.02    & 5.4  \\\hline\hline
\multirow{5}{*}{1.5}    & 0     & 0.6 & 12          & 2.04        & 25.78           & 46.31    & 6    & 0.72        & 24.74           & 40.8     & 10   \\
                         & 0     & 0.8 & 12          & 2.24        & 5.20             & 32.45    & 6    & 1.86        & 4.85            & 30.91    & 6    \\
                         & 0     & 1   & 12          & 0           & 0               & 30.46    & 12   & 0           & 0               & 27.80     & 12   \\
                         & 0     & 1.2 & 12          & 0           & 0               & 31.82    & 12   & 0           & 0               & 27.20     & 12   \\
                         &       & AVG & 12          & 1.07        & 7.75            & 35.26    & 9.0    & 0.64        & 7.40             & 31.68    & 10.0   \\\hline
             \multirow{5}{*}{1.5}            & 0.5   & 0.6 & 12          & 5.58        & 30.53           & 51.72    & 4    & 3.72        & 29.22           & 39.2     & 5    \\
                         & 0.5   & 0.8 & 12          & 5.31        & 17.18           & 41.22    & 5    & 4.07        & 16.33           & 34.39    & 2    \\
                         & 0.5   & 1   & 12          & 0.98        & 5.60             & 33.55    & 8    & 0.98        & 5.60             & 29.46    & 8    \\
                         & 0.5   & 1.2 & 12          & 0.50         & 2.00               & 32.80     & 11   & 0.33        & 1.83            & 29.39    & 11   \\
                         &       & AVG & 12          & 3.09        & 13.83           & 39.82    & 7.0    & 2.28        & 13.25           & 33.11    & 6.5  \\\hline
             \multirow{5}{*}{1.5}            & 1     & 0.6 & 12          & 4.20         & 27.11           & 41.36    & 6    & 4.88        & 27.55           & 35.84    & 4    \\
                         & 1     & 0.8 & 12          & 3.23        & 10.89           & 37.81    & 3    & 4.68        & 11.98           & 37.14    & 3    \\
                         & 1     & 1   & 12          & 1.12        & 4.08            & 33.06    & 9    & 1.31        & 4.25            & 30.23    & 8    \\
                         & 1     & 1.2 & 12          & 0.17        & 0.83            & 31.91    & 11   & 0.17        & 0.83            & 28.95    & 11   \\
                         &       & AVG & 12          & 2.18        & 10.73           & 36.04    & 7.3  & 2.76        & 11.15           & 33.04    & 6.5  \\\hline
                         & AVG   &     & 12          & 2.11        & 10.77           & 37.04    & 7.8  & 1.89        & 10.60            & 32.61    & 7.7  \\\hline\hline
\multicolumn{1}{l}{AVG}  &       &     & 12          & 5.81        & 18.03           & 49.83    & 6.1  & 5.75        & 17.95           & 37.04    & 5.8 \\\bottomrule
\end{tabular}
}
\end{table}

\begin{table}[tb!]
\caption{Average performance of the five approaches - part 2 (50 customers)}
\label{4approaches_betadelta-part 2}
\centering
\scalebox{0.6}{\color{black}
\begin{tabular}{l|l|l|l|cccc|cccc|cccc}\toprule
\multicolumn{1}{l}{}     &       &     &             & \multicolumn{4}{c}{MH}                          & \multicolumn{4}{c}{ME}                          & \multicolumn{4}{c}{OPrd\_PE}                    \\\hline
\multicolumn{1}{l}{$\beta$} & $\delta$  & $c$   & \#instances & gap{[}\%{]} & gap\_ub{[}\%{]} & t{[}s{]} & freq & gap{[}\%{]} & gap\_ub{[}\%{]} & t{[}s{]} & freq & gap{[}\%{]} & gap\_ub{[}\%{]} & t{[}s{]} & freq  \\\hline\hline
\multirow{5}{*}{0.5}    & 0     & 0.6 & 12          & 27.28       & 38.05           & 0.20      & 0    & 7.27        & 20.10            & 1.07     & 6    & 39.12       & 44.29           & 135.65   & 6    \\
                         & 0     & 0.8 & 12          & 39.74       & 47.43           & 0.19     & 0    & 29.78       & 38.00              & 21.16    & 0    & 9.64        & 19.50            & 110.64   & 6    \\
                         & 0     & 1   & 12          & 39.29       & 50.59           & 0.19     & 0    & 30.35       & 42.85           & 151.00      & 0    & 8.33        & 24.72           & 144.39   & 9    \\
                         & 0     & 1.2 & 12          & 38.48       & 47.80            & 0.18     & 0    & 27.71       & 38.59           & 155.86   & 0    & 4.41        & 18.41           & 227.02   & 9    \\
                         &       & AVG & 12          & 36.20        & 45.97           & 0.19     & 0    & 23.78       & 34.88           & 82.27    & 1.5  & 15.38       & 26.73           & 154.42   & 7.5  \\\hline
         \multirow{5}{*}{0.5}                & 0.5   & 0.6 & 12          & 23.83       & 29.04           & 0.23     & 2    & 12.33       & 18.5            & 26.05    & 5    & 6.19        & 11.52           & 97.49    & 10   \\
                         & 0.5   & 0.8 & 12          & 28.48       & 37.73           & 0.19     & 0    & 19.95       & 30.67           & 33.83    & 3    & 17.22       & 29.28           & 186.41   & 8    \\
                         & 0.5   & 1   & 12          & 32.34       & 44.21           & 0.21     & 0    & 23.67       & 36.83           & 26.18    & 2    & 25.79       & 37.36           & 156.92   & 7    \\
                         & 0.5   & 1.2 & 12          & 36.18       & 44.25           & 0.22     & 0    & 26.47       & 35.80            & 37.18    & 0    & 17.47       & 27.47           & 246.14   & 8    \\
                         &       & AVG & 12          & 30.21       & 38.81           & 0.21     & 0.5  & 20.60        & 30.45           & 30.81    & 2.50  & 16.67       & 26.41           & 171.74   & 8.3  \\\hline
            \multirow{5}{*}{0.5}             & 1     & 0.6 & 12          & 34.97       & 39.99           & 0.22     & 0    & 23.15       & 28.42           & 55.71    & 0    & 5.06        & 11.85           & 133.92   & 10   \\
                         & 1     & 0.8 & 12          & 35.31       & 44.29           & 0.18     & 0    & 26.58       & 36.83           & 90.07    & 0    & 4.31        & 18.09           & 151.79   & 7    \\
                         & 1     & 1   & 12          & 34.41       & 45.18           & 0.20      & 0    & 24.98       & 37.23           & 57.36    & 0    & 2.88        & 18.37           & 194.83   & 8    \\
                         & 1     & 1.2 & 12          & 34.49       & 43.48           & 0.20      & 0    & 23.59       & 33.90            & 65.85    & 0    & 5.01        & 17.47           & 319.38   & 9    \\
                         &       & AVG & 12          & 34.79       & 43.24           & 0.20      & 0    & 24.58       & 34.09           & 67.25    & 0    & 4.31        & 16.44           & 199.98   & 8.5  \\\hline
                         & AVG   &     & 12          & 33.73       & 42.67           & 0.20      & 0.2  & 22.99       & 33.14           & 60.11    & 1.3  & 12.12       & 23.19           & 175.38   & 8.1  \\\hline\hline
\multirow{5}{*}{1}      & 0     & 0.6 & 12          & 30.44       & 54.26           & 0.29     & 0    & 18.86       & 47.01           & 52.34    & 1    & 14.61       & 41.61           & 387.09   & 5    \\
                         & 0     & 0.8 & 12          & 24.86       & 39.84           & 0.21     & 0    & 13.10        & 30.33           & 271.39   & 0    & 2.82        & 22.53           & 514.36   & 6    \\
                         & 0     & 1   & 12          & 12.54       & 20.80            & 0.18     & 0    & 3.82        & 12.59           & 154.01   & 6    & 6.76        & 15.89           & 646.12   & 4    \\
                         & 0     & 1.2 & 12          & 4.07        & 5.50             & 0.19     & 0    & 1.01        & 2.50             & 1.35     & 6    & 22.52       & 23.50            & 947.33   & 6    \\
                         &       & AVG & 12          & 17.98       & 30.10            & 0.22     & 0    & 9.20         & 23.11           & 119.77   & 3.3  & 11.68       & 25.88           & 623.72   & 5.3  \\\hline
        \multirow{5}{*}{1}                  & 0.5   & 0.6 & 12          & 29.44       & 49.99           & 0.23     & 0    & 21.16       & 44.73           & 12.65    & 0    & 33.31       & 49.80            & 316.72   & 5    \\
                         & 0.5   & 0.8 & 12          & 24.19       & 42.85           & 0.23     & 0    & 12.38       & 34.16           & 175.12   & 2    & 37.02       & 53.02           & 409.02   & 2    \\
                         & 0.5   & 1   & 12          & 15.14       & 24.59           & 0.24     & 2    & 7.85        & 18.69           & 103.67   & 3    & 2.61        & 13.96           & 620.31   & 6    \\
                         & 0.5   & 1.2 & 12          & 8.13        & 13.17           & 0.20      & 1    & 4.41        & 10.00              & 53.67    & 2    & 6.34        & 11.83           & 688.02   & 11   \\
                         &       & AVG & 12          & 19.22       & 32.65           & 0.23     & 0.8  & 11.45       & 26.90            & 86.28    & 1.8  & 19.82       & 32.15           & 508.52   & 6    \\\hline
           \multirow{5}{*}{1}               & 1     & 0.6 & 12          & 26.56       & 46.19           & 0.29     & 0    & 12.38       & 36.06           & 75.96    & 2    & 14.28       & 37.13           & 331.45   & 2    \\
                         & 1     & 0.8 & 12          & 18.74       & 28.9            & 0.24     & 0    & 11.65       & 22.74           & 87.59    & 3    & 1.35        & 14.32           & 489.73   & 9    \\
                         & 1     & 1   & 12          & 11.67       & 17.07           & 0.22     & 1    & 6.43        & 12.63           & 76.41    & 6    & 0.78        & 7.69            & 638.89   & 11   \\
                         & 1     & 1.2 & 12          & 5.71        & 8.28            & 0.20      & 5    & 2.74        & 5.60             & 52.49    & 8    & 1.02        & 4.09            & 769.08   & 11   \\
                         &       & AVG & 12          & 15.67       & 25.11           & 0.24     & 1.5  & 8.30         & 19.26           & 73.11    & 4.8  & 4.36        & 15.81           & 557.29   & 8.3  \\\hline
                         & AVG   &     & 12          & 17.62       & 29.29           & 0.23     & 0.8  & 9.65        & 23.09           & 93.05    & 3.3  & 11.95       & 24.61           & 563.18   & 6.5  \\\hline\hline
\multirow{5}{*}{1.5}    & 0     & 0.6 & 12          & 19.91       & 39.26           & 0.21     & 0    & 6.70         & 29.22           & 152.78   & 1    & 50.04       & 61.71           & 608.24   & 3    \\
                         & 0     & 0.8 & 12          & 7.63        & 10.37           & 0.20      & 3    & 1.58        & 4.68            & 0.45     & 6    & 0.52        & 3.63            & 958.10    & 9    \\
                         & 0     & 1   & 12          & 1.53        & 1.53            & 0.21     & 3    & 0.51        & 0.51            & 0.44     & 9    & 4.42        & 4.42            & 743.77   & 8    \\
                         & 0     & 1.2 & 12          & 0           & 0               & 0.21     & 12   & 0           & 0               & 0.48     & 12   & 0           & 0               & 843.10    & 12   \\
                         &       & AVG & 12          & 7.27        & 12.79           & 0.21     & 4.5  & 2.20         & 8.60             & 38.54    & 7    & 13.74       & 17.44           & 788.30    & 8    \\\hline
        \multirow{5}{*}{1.5}                 & 0.5   & 0.6 & 12          & 15.43       & 38.71           & 0.25     & 2    & 8.04        & 33.11           & 25.6     & 5    & 25.88       & 45.78           & 494.29   & 6    \\
                         & 0.5   & 0.8 & 12          & 10.99       & 21.90            & 0.26     & 0    & 5.04        & 16.77           & 2.25     & 4    & 13.13       & 23.15           & 633.82   & 7    \\
                         & 0.5   & 1   & 12          & 4.35        & 8.66            & 0.22     & 3    & 2.25        & 6.62            & 5.09     & 10   & 1.61        & 6.29            & 766.38   & 6    \\
                         & 0.5   & 1.2 & 12          & 2.51        & 3.83            & 0.23     & 7    & 0.96        & 2.33            & 0.88     & 10   & 0           & 1.50             & 649.83   & 12   \\
                         &       & AVG & 12          & 8.32        & 18.28           & 0.24     & 3    & 4.07        & 14.71           & 8.46     & 7.3  & 10.16       & 19.18           & 636.08   & 7.8  \\\hline
            \multirow{5}{*}{1.5}             & 1     & 0.6 & 12          & 15.01       & 35.23           & 0.29     & 0    & 6.32        & 28.73           & 151.81   & 4    & 10.08       & 30.82           & 490.97   & 7    \\
                         & 1     & 0.8 & 12          & 9.29        & 16.04           & 0.21     & 1    & 5.69        & 12.88           & 1.58     & 3    & 0.93        & 8.64            & 784.49   & 11   \\
                         & 1     & 1   & 12          & 3.17        & 5.95            & 0.19     & 7    & 0.17        & 3.23            & 0.80      & 11   & 3.12        & 5.95            & 757.84   & 7    \\
                         & 1     & 1.2 & 12          & 1.02        & 1.67            & 0.23     & 8    & 0.17        & 0.83            & 0.62     & 11   & 0.34        & 1.00               & 761.20    & 10   \\
                         &       & AVG & 12          & 7.12        & 14.72           & 0.23     & 4    & 3.09        & 11.42           & 38.71    & 7.3  & 3.62        & 11.60            & 698.63   & 8.8  \\\hline
                         & AVG   &     & 12          & 7.57        & 15.26           & 0.23     & 3.8  & 3.12        & 11.58           & 28.57    & 7.2  & 9.17        & 16.07           & 707.67   & 8.2  \\\hline\hline
\multicolumn{1}{l}{AVG}  &       &     & 12          & 19.64       & 29.07           & 0.22     & 1.6  & 11.92       & 22.60            & 60.58    & 3.9  & 11.08       & 21.29           & 482.08   & 7.6 \\\bottomrule
\end{tabular}
}
\end{table}
\end{appendices}

\end{document}